\documentclass[]{amsart}
\usepackage[utf8]{inputenc}
\usepackage[margin=1.5in]{geometry}
\usepackage{amsmath}
\usepackage{amscd}
\usepackage{amssymb}
\usepackage{amsthm}
\usepackage{ascmac}
\usepackage{url}
\usepackage[normalem]{ulem}
\usepackage[X2,T1]{fontenc} 
\DeclareTextSymbolDefault{\cyrsdsc}{X2}
\usepackage{mathtools}
\usepackage{tikz-cd}
\usetikzlibrary{cd}
\usepackage{mathrsfs}
\usepackage{hyperref}
\hypersetup{breaklinks=true} 
\usepackage[all]{xy} 
\usepackage{color}
\usepackage{enumitem}
\setenumerate{
label=(\arabic*),
font=\normalfont
}
\setlist[enumerate,2]{label=(\alph*),ref=(\alph*)}
\theoremstyle{definition}
\newtheorem{thm}{Theorem}[section]
\newtheorem{lem}[thm]{Lemma}
\newtheorem{cor}[thm]{Corollary}
\newtheorem{prop}[thm]{Proposition}
\newtheorem{defn-prop}[thm]{Definition-Proposition}
\newtheorem{defn-thm}[thm]{Definition-Theorem}

\newtheorem{rem}[thm]{Remark}

\newtheorem{defn}[thm]{Definition}

\def\a{{\mathrm a}}
\def\p{{\mathfrak p}}
\def\q{{\mathfrak q}}
\def\F{{\mathbb F}}
\def\G{{\mathbb G}}
\def\Gm{{\mathbb G_{m}}}
\def\Ga{{\mathbb G_{a}}}

\def\P{{\mathbb P}}
\def\Q{{\mathbb Q}}
\def\C{{\mathbb C}}
\def\Z{{\mathbb Z}}
\def\G{{\mathbb G}}
\def\ksep{k^{\mathrm{sep}}}
\def\cO{\mathcal{O}}

\def\red{\mathrm{red}}
\def\sing{\mathrm{sing}}

\DeclareMathOperator\Res{\mathrm{Res}}

\DeclareMathOperator\SL{\mathrm{SL}}
\DeclareMathOperator\Ord{\mathrm{Ord}}

\DeclareMathOperator\Stab{\mathrm{Stab}}

\DeclareMathOperator\Aut{\mathrm{Aut}}

\DeclareMathOperator\Gal{\mathrm{Gal}}

\DeclareMathOperator\Pic{\mathrm{Pic}}

\DeclareMathOperator\Frac{\mathrm{Frac}}
\DeclareMathOperator\Spec{\mathrm{Spec}}

\DeclareMathOperator\Shaf{\mathrm{Shaf}}

\DeclareMathOperator\chara{\mathrm{char}}

\DeclareMathOperator\Hom{\mathrm{Hom}}

\DeclareMathOperator\Br{\mathrm{Br}}

\DeclareMathOperator\diag{\mathrm{diag}}
\DeclareMathOperator{\PGL}{PGL}

\def\et{\mathop{\text{\rm \'et}}}

\newcommand{\cred}{\color{black}}
\newcommand{\cblue}{\color{black}}
\newcommand{\cora}{\color{black}}

\usepackage[
backend=biber,
bibstyle=ieee-alphabetic,
citestyle=ieee-alphabetic,
sorting=nyt,
isbn=false,
url=false,
doi=false,
giveninits=true,
maxnames=10,
labelalpha=true,
maxalphanames=4,
dashed=false,
]{biblatex}
\DeclareCaseLangs{}
\DeclareFieldFormat*{date}{\mkbibparens{#1}}
\DeclareFieldFormat*{volume}{\mkbibbold{#1}}
\DeclareFieldFormat*{title}{\mkbibemph{#1\isdot}}
\DeclareFieldFormat*{journaltitle}{#1}
\DeclareFieldFormat*{pages}{{#1}}

\AtEveryBibitem{%
  \ifentrytype{online}{%
    \DeclareFieldFormat*{date}{#1}%
  }{%
  }%
}

\title[Fano threefolds of genus 12 in positive and mixed characteristic]{Fano threefolds of genus 12 with large automorphism group in positive and mixed characteristic}

\begin{document}

\numberwithin{equation}{subsection}

\author{Tetsushi Ito}
\address{
Department of Mathematics, Faculty of Science, Kyoto University
Kyoto, 606--8502, Japan}
\email{tetsushi@math.kyoto-u.ac.jp}
\thanks{Supported by JSPS KAKENHI Grant Numbers
JP21K18577, JP23K20204, JP23K20786, JP24K21512 and JP25K00905.}

\author{Akihiro Kanemitsu}
\address{Department of Mathematical Sciences, Graduate School of Science, Tokyo Metropolitan University, 1-1 Minami-Osawa, Hachioji-shi, Tokyo 192-0397, Japan}
\email{kanemitsu@tmu.ac.jp}
\thanks{Supported by JSPS KAKENHI Grant Numbers 23K12948.}

\author{Teppei Takamatsu}
\address{Department of Mathematics, Faculty of Science,
Saitama University,
255 Shimo-Okubo, Sakura-ku,
Saitama-shi, Saitama 338-8570,
Japan}
\email{teppeitakamatsu.math@gmail.com}
\thanks{Supported by JSPS KAKENHI Grant Numbers JP22KJ1780 and JP25K17228.}

\author{Yuuji Tanaka}
\address{Beijing Institute of Mathematical Sciences and Applications (BIMSA), No. 544, Hefangkou Village, Huaibei Town, Huairou District, Beijing 101408, China}
\email{ytanaka@bimsa.cn}
\thanks{Supported by JSPS KAKENHI Grant Numbers JP21H00973, JP21K03246, Startup Grant at BIMSA, the Beijing NSF Beijing International Scientist Project IS25031 }

\begin{abstract}
We study prime Fano threefolds of genus 12 ($V_{22}$-varieties) with positive-dimensional automorphism groups in positive and mixed characteristic.
We classify such varieties over any perfect field.
In particular, we prove that $V_{22}$-varieties of Mukai--Umemura type over $k$ exist if and only if $\mathrm{char}\ k \neq 2$, $5$.
We also prove the same result for $\Ga$-type.
As arithmetic applications, we show that the Shafarevich conjecture holds for $V_{22}$-varieties of Mukai--Umemura type and of $\Gm$-type,
while it fails for $V_{22}$-varieties of $\Ga$-type.
Moreover, we prove that there exists $V_{22}$-variety over $\Z$,
whereas there do not exist $V_{22}$-varieties over $\Z$ whose generic fiber has a positive-dimensional automorphism group.
\end{abstract}


\maketitle

\section{Introduction}

A smooth projective variety $X$ over a field $k$ is called a Fano variety when the anti-canonical divisor $-K_{X}$ is ample.
Among Fano threefolds, those with Picard rank 1 and index 1 — commonly referred to as prime Fano threefolds — form one of the most important classes.
A fundamental result in the classification theory of Fano threefolds states that complex prime Fano threefolds are classified into ten types according to their genus
\[
g \in \{2, 3, \ldots, 10,12\}.
\]
(see, e.g., \cite{Iskovskikh-Prokhorov} and references therein.)
Here, the genus of a prime Fano threefold over a field is defined by
\[
g = \frac{(-K_{X})^3+2}{2}.
\]
A similar classification of Fano threefolds over algebraically closed fields of positive characteristic was recently given by Tanaka (\cite{Tanaka1} and \cite{Tanaka2}).

In this paper, we undertake a more detailed algebraic and arithmetic geometric study of prime Fano threefolds of the highest genus, the ones of genus 12, which we refer to as {\it $V_{22}$-varieties}.
More precisely, we discuss the classification of $V_{22}$-varieties over a perfect field whose automorphism groups have dimension at least one, generalizing the results in \cite{Mukai-Umemura}, \cite{Prokhorov_automorphism}, \cite{DKK17},
\cite{Kuznetsov-ProkhorovGm}, \cite{Kuznetsov-Prokhorov-Shramov} (cf.\ \cite{duboulozFujitaKishimoto}).
Moreover, we study the behavior of their reductions along discrete valuation rings, and we prove the non-existence and finiteness results for $V_{22}$-varieties over a ring of integers.

\subsection{Classification of $V_{22}$-varieties with positive-dimensional automorphism groups}

Our first main theorem is an algebro-geometric classification of $V_{22}$-varieties with positive-dimensional automorphism groups.

\begin{thm}[Theorems \ref{thm:MukaiUmemuraclassification}, \ref{thm:Gaclassification}, and \ref{thm:Gmclassification}]
\label{thm:intromainthm1}
Let $k$ be an algebraically closed field, and $X$ a $V_{22}$-variety such that $\dim \Aut_{X/k} \geq 1$.
Then $\Aut_{X/k,\red}^{\circ}$ is one of $\Gm$, $\Ga$, and $\PGL_{2}$.
Moreover, we obtain the following:
\begin{enumerate}
\item 
There exists a unique $V_{22}$-variety $X$ with $\Aut_{X/k, \red}^{\circ} \simeq \PGL_{2,k}$ if $\chara k \neq 2,5$. 
Furthermore, we have $\Aut_{X/k, \red} \simeq \PGL_{2,k}$ in this case.
On the other hand, there exists no such variety if $\chara k = 2$ or $5$.

\item 
There exists a unique $V_{22}$-variety $X$ with $\Aut_{X/k, \red}^{\circ} \simeq \Ga$ if $\chara k \neq 2,5$.
Furthermore, we have $\Aut_{X/k, \red} \simeq \Ga \rtimes \mu_4$.
On the other hand, there exists no such variety if $\chara k = 2$ or $5$.

\item 
The set of isomorphism classes of $V_{22}$-varieties $X$ with $\Aut_{X/k, \red}^{\circ} \simeq \Gm$ has a natural bijection with 
\[
\P^1 (k) \setminus \{(0:1), (1:1), (5:4), (1:0)\}.
\]
Furthermore, we have $\Aut_{X/k, \red} \simeq \Gm \rtimes \Z/2\Z$.
\end{enumerate}
\end{thm}

When $k= \C$, Theorem \ref{thm:intromainthm1} was proved by 
\cite{Mukai-Umemura}, \cite{Prokhorov_automorphism}, \cite{DKK17},
\cite{Kuznetsov-ProkhorovGm}, \cite{Kuznetsov-Prokhorov-Shramov}.
Moreover, the non-existence in Theorem \ref{thm:intromainthm1} (1) provides an answer to a question posed by Shigeru Mukai to the authors.

A $V_{22}$-variety $X$ over a general field $k$ is called a {\it Mukai--Umemura variety} (or a $V_{22}$-variety of $\PGL_2$-type), if $X_{\overline{k}}$ satisfies (1) in Theorem \ref{thm:intromainthm1}.
$V_{22}$-varieties of $\Ga$-type and $V_{22}$-varieties of $\Gm$-type over $k$ are defined similarly.

As a corollary of the above theorem, we also construct a $W(k)$-lift of $V_{22}$-varieties of $\PGL_2$-/$\Ga$-/$\Gm$-type over an algebraically closed field $k$ preserving the type of $\Aut^{\circ}_{\red}$ (Corollaries \ref{cor:liftMukaiUmemura}, \ref{cor:liftGa}, and \ref{cor:liftGm}).

Furthermore, we give the classification of $V_{22}$-varieties of $\PGL_2$-/$\Ga$-/$\Gm$-type over a general perfect field (Theorems \ref{thm:Mukai-Umemurageneral}, \ref{thm:GaV22overk}, and \ref{thm:Gmclassificationgeneral}).
Restricting to the case of finite fields, the results are as follows:
\begin{thm}
\label{thm:intromainthm1'}
For $\ast \in \{ \PGL_{2}, \Ga, \Gm \}$,
let $N_{\ast, \F_q}$ be the number of isomorphism classes of $V_{22}$-variety of $\ast$-type over $\F_{q}$. 
Then we have the following.
\begin{enumerate}
\item
\[
N_{\PGL_{2}, \F_q} =
\begin{cases}
    0 & \textup{if $2|q$ or $5 |q$,} \\
    1 & \textup{otherwise.}
\end{cases}
\]
\item 
\[
N_{\Ga, \F_q} =
\# (\F_q^{\times} / \F_q^{\times 4})
= 
\begin{cases}
    0 & \textup{if $2|q$ or $5 |q$,} \\
\gcd (4, q-1)& \textup{otherwise.}
\end{cases}
\]
\item
\begin{eqnarray*}
N_{\Gm, \F_q} &=&
2 \# (\P^1 (\F_q) \setminus \{(0:1), (1:1), (5:4), (1:0)\}) \\
&=& 
\begin{cases}
2q-4 &\textup{if $2|q$ or $5|q$}, \\
2q-6 &\textup{otherwise}.
\end{cases}
\end{eqnarray*}
\end{enumerate}
\end{thm}

Moreover, as a by-product, we found examples in which a $V_{22}$-variety of $\Gm$-type degenerates into a $V_{22}$-variety of $\Ga$-type, as well as examples in which it degenerates into Mukai--Umemura type in an unusual manner (see Definition \ref{defn:Gmreduction}, Remark \ref{rem:twistednew}, and Proposition \ref{prop:grcforsplitGm}). 
Such degenerations do not appear to have been extensively investigated in the literature.

\subsection{Arithmetic applications}

The second main theorem concerns the Shafarevich conjecture, the finiteness of isomorphism classes of certain classes of varieties over a number field admitting good reduction outside a fixed finite set of finite places.

For a $V_{22}$-variety $X$ over a number field $F$ and a finite place $\p$ of $F$, we say $X$ admits \emph{good reduction} at $\p$ if there exists a smooth projective scheme $\mathcal{X}$ over the completion $\cO_{F,\p}$ such that $X_{F_{\p}} \simeq \mathcal{X}_{F_{\p}}$.
In this case, the special fiber $\mathcal{X}_{\kappa(\p)}$
is a $V_{22}$-variety over the residue field $\kappa(\p)$.

Fix a number field $F$ and a finite set $S$ of finite places.
For $\ast \in \{ \PGL_2, \Ga, \G_m \}$, we set
\[
\mathrm{Shaf}_{*, F, S} :=
\{X \colon V_{22}\textup{-variety of } \ast\textup{-type over }F \mid 
X \textup{ admits good reduction at all } \p\notin S
\}/F\textup{-isom}.
\]
The following is the second main theorem of this paper (Theorems \ref{thm:Mukai-Umemura_Shafarevich} and \ref{thm:GmShafarevich}, and Proposition \ref{prop:GaV22shaffails}).

\begin{thm}
\label{thm:intromainthm2}
The sets $\Shaf_{\PGL_2, F, S}$ and $\Shaf_{\Gm, F, S}$ are finite sets,
but the set $\Shaf_{\Ga, F, S}$ is an infinite set if $10 \in \cO_{F,S}^{\times}$.
\end{thm}

More precisely, for $\Shaf_{\mathrm{\PGL_2},F,S}$, we prove the following precise formula
\[
\# \Shaf_{\PGL_2, F, S} =
\begin{cases}
0 & \textup{if } 10 \notin \cO_{F,S}^{\times}, \\
2^{\#S+r-1}  & \textup{if } 10 \in \cO_{F,S}^{\times},
\end{cases}
\]
where $r$ is the number of real places of $F$.
Moreover, $\# \Shaf_{\Gm,F,S}$ is related to the number of solutions of the unit equation $\alpha + \beta =1$ ($\alpha, \beta \in \cO_{F,S}^{\times}$) and to the heights of those solutions.

On the other hand, since $\Shaf_{\Ga,F,S}$ may be an infinite set, we introduce the following set: 
\[
\Shaf'_{\Ga,F,S} \coloneqq
\{X \colon \Ga\textup{-}V_{22}\textup{-variety over }F \mid 
X \textup{ admits } \Ga\textup{-good reduction at all }\p\notin S
\}/F\textup{-isom},
\]
which is a subset of $\Shaf_{\Ga,F,S}$.
Here, we say $X$ admits \emph{$\Ga$-good reduction} if there exists a smooth projective scheme $\mathcal{X}$ over $\cO_{F,\p}$ such that $X_{F_{\p}} \simeq \mathcal{X}_{F_{\p}}$ and the special fiber $\mathcal{X}_{k (\p)}$ is a $V_{22}$-variety of $\Ga$-type.
Then we prove that the set $\Shaf'_{\Ga}$ is a finite set with an explicit upper bound
 (see Theorem \ref{thm:shafGatoGa} for details).

Theorem \ref{thm:intromainthm2} completely determines whether the Shafarevich conjecture holds or not for $V_{22}$-varieties with positive-dimensional automorphism groups.
While this gives a negative answer to Javanpeykar--Loughran's question (\cite[Question 1.6]{Javanpeykar-Loughran:GoodReductionFano}), the finiteness of $\Shaf_{\Ga,F,S}'$ suggests that we may have a positive answer to their question under an appropriate modification.

We also consider the problem of the existence of smooth projective schemes over $\Z$.
It has been expected that smooth projective scheme over $\Z$ is very few (cf.\ \cite{Mazur}).
Abrashkin and Fontaine (\cite{Abrashkin}, \cite{Fontaine}) proved that if smooth projective variety $X$ admits a smooth projective model $\mathcal{X}$ over $\Z$, then we have $h^{i+j} (\mathcal{X}_{\C})=0$ for $i+j \leq 3$ and $i \neq j$.
Recently, it was proved that there is no Enriques surface over $\Z$ (\cite{Schroer}). Also, the authors proved that there exists a Mukai $n$-folds of genus $7$ over $\Z$ if and only if $5 \leq n\leq 10$ (\cite{IKTT}).

For $V_{22}$-varieties, we obtain the following result (Theorems \ref{thm:noV22withlargeautomoverZ} and \ref{thm:V22overZ}):

\begin{thm}
\label{thm:intromainthm3}
\begin{enumerate}
    \item 
There is no smooth projective scheme $\mathcal{X}$ over $\Z$ such that $\mathcal{X}_{\Q}$ is a $V_{22}$-variety with $\dim \Aut_{\mathcal{X}_{\Q}/\Q} >0$.

\item 
There is a smooth projective scheme $\mathcal{X}$ over $\Z$ such that $\mathcal{X}_{\Q}$ is a $V_{22}$-variety.
\end{enumerate}
\end{thm}

\begin{rem}
Many results are known about the Shafarevich conjecture.
It was originally proved
by Shafarevich himself for elliptic curves.
Faltings (\cite{FaltingsShaf}) and Zarhin (\cite{Zarhin}) proved it for abelian varieties of fixed dimension and curves of fixed genus greater than or equal to $2$.
It was also studied for other classes of varieties,
including del Pezzo surfaces and Brauer--Severi varieties \cite{Scholl}, $K3$ and hyper-K\"ahler varieties
\cite{Andre}, \cite{She}, \cite{Takamatsu}, \cite{Fu-Li-Takamatsu-Zou}, flag varieties (\cite{Javanpeykar-Loughran:flag}), proper hyperbolic polycurves (\cite{Nagamachi-Takamatsu}, see also \cite{Javanpeykarcanonicallypolarized}), Enriques surfaces (\cite{TakamatsuEnriques}), bielliptic surfaces (\cite{Takamatsubielliptic}),
hypersurfaces of abelian varieties (\cite{Lawrence-Sawin}),
certain very irregular varieties (\cite{Kramer-Maculan}), and varieties with globally generated cotangent bundles (\cite{Kramer-Maculan_globally_generated}).

Recently, the Shafarevich conjecture for
some classes of Fano threefolds of Picard rank $1$ were proved by Javanpeykar--Loughran (\cite{Javanpeykar-Loughran:GoodReductionFano}, \cite{Javanpeykar-Loughran:completeintersection})
and Licht \cite{Licht}.
More precisely, they proved the Shafarevich conjecture for Fano threefolds of Picard rank 1 and index greater than or equal to $2$, and prime Fano threefolds of genus $2,3,4,5$.

In \cite{IKTT}, the authors proved the Shafarevich conjecture for prime Fano threefolds of genus $7$. Moreover, we studied the Shafarevich conjecture for Mukai varieties of genus 7, which are higher-dimensional analogues of prime Fano threefolds of genus $7$.
\end{rem}

\subsection{Methods}
As in \cite{Kuznetsov-Prokhorov-Shramov}, one of the key ingredients in our proof is the two-ray game connecting a $V_{22}$-variety and a quintic del Pezzo threefold $V_5$. The existence of this operation in positive characteristic is ensured by \cite{Tanakaflop}.
This procedure reduces the classification of $V_{22}$-varieties with large automorphism group to the classification of quintic curves in $V_5$-varieties with large stabilizer.
To obtain such a classification, we use geometric properties of $V_5$-varieties and the explicit description of their automorphism group, which were studied by the authors (\cite{V5}).
We also need another two-ray game connecting a $V_{22}$-variety and a quadric threefold, for studying the structure of the automorphism groups of $V_{22}$-varieties of $\Gm$-type.

Furthermore, to examine the phenomena of degeneration and reduction of $V_{22}$-varieties, we employ the two-ray games in the relative setting (with a one-dimensional base), as we did in \cite{IKTT}.

Concerning the Shafarevich conjecture, because the Torelli theorem fails for $V_{22}$-varieties, we cannot reduce the problem to abelian varieties.
Firstly, we need a more careful understanding of the behavior of reduction in positive characteristic, namely, the establishment of some kind of good-reduction criteria is the key to the proof (see Proposition \ref{prop:goodreductioncriteriaMukaiUmemura}, Proposition \ref{prop:GaV22degenerates}, and Proposition \ref{prop:grcforsplitGm}).
Based on the above criteria, 
we reduce the question to various finiteness properties of rings of integers.
In the Mukai--Umemura type case, the finiteness of the Brauer group of the ring of integers; in the $\Ga$-type case, the finite generation of the unit group; and in the $\Gm$-type case, Siegel–Mahler–Lang's theorem (which was also used in \cite{Scholl}) play crucial roles.

\subsection{Organization of the paper}

The organization of this paper is as follows.
In Section \ref{section:preliminaries}, 
we introduce the basic terminology for $V_{22}$-varieties and $V_5$-varieties. We also review the necessary facts about $V_5$-varieties established in \cite{IKTT}.
In Section \ref{section:quintics_in_V5}, we classify smooth rational quintic curves with large stabilizer in a (split) $V_5$-variety via explicit computations, distinguishing the case of characteristic two from the other cases. 
We also classify quintic curves over a discrete valuation ring for subsequent degeneration studies.
In Section \ref{section:two-ray},
we establish the preliminaries for the two-ray game. 
In particular, we establish relative two-ray games starting from lines and conics on $V_{22}$-varieties over Dedekind schemes and their converse.
In Section \ref{section:V_22-variety_in_positive_mixed},
we use the results of the preceding sections to give a classification of $V_{22}$-varieties with positive dimensional stabilizer
in positive characteristic (Theorem \ref{thm:intromainthm1} and Theorem \ref{thm:intromainthm1'}).
Moreover, we prove existence and non-existence results for $V_{22}$-varieties over $\Z$ 
(Theorem \ref{thm:intromainthm3}).
In Section \ref{section:Shafarevich}, we study degenerations (reductions) of $V_{22}$-varieties with large automorphism group and prove results concerning the Shafarevich conjecture (Theorem \ref{thm:intromainthm2}).

\subsection*{Notations and Convention}

\begin{itemize}
\item 
Let $k$ be a field.
A variety over $k$ is a geometrically integral
scheme that is separated and of finite type over $k$.
A curve over $k$ is a variety of pure dimension onel.
For a variety $X$ over $k$, we denote its Picard rank by $\rho (X)$.

\item
A rational normal curve of degree $d$ over $k$ is a smooth rational curve $C$ with an embedding $C \hookrightarrow \P^d$ of degree $d$ that linearly spans $\P^d$.

\item 
Let $X$ be a flat projective scheme over a Noetherian scheme $B$.
We denote the relative Picard functor by $\Pic_{X/B}$.
We denote the relative automorphism scheme by $\Aut_{X/B}$.

    \item 
    Let $X$ be a flat projective scheme over a Noetherian scheme $B$.
    Let $Z \subset X$ be a closed subscheme which is flat over $B$.
    We denote the stabilizer scheme of $Z$ 
    \[
    \Aut_{X/B} \times_{\Hom_{B}(Z, X)} \Aut _{Z/B}
    \]
    by $\Aut_{(X,Z)/B}$.
    Here, $\Hom_B(Z,X)$ is the homomorphism scheme from $Z$ to $X$ over $B$.
    When $B$ is a spectrum of a field, we denote $\Aut_{X/B} (B)$ by $\Aut (X)$ and $\Aut_{(X,Z)/B} (B)$ by $\Aut (X,Z) = \Stab_X (Z)$.
    \item 
    For any commutative ring $R$, we denote the ring of Witt vectors of $R$ by $W (R)$.
\item 
For a scheme $X$, we denote its reduction by $X_{\red}$.
For a group scheme $G$ over a perfect field $k$, we denote its reduction, which is a reduced $k$-group scheme, by $G_{\red}$.
\item 
For a reduced scheme $X$ over a scheme $B$, we denote its non-smooth locus with the reduced closed subscheme structure by $X_{\mathrm{sing}}$.
    \item 
    For an algebraic group $G$ over a field $k$, we denote its center by $Z_{G}$.
    We denote its identity component by $G^{\circ}$.
    \item 
    Let $B$ be a scheme.
    A geometric point $s$ on $B$ is a morphism 
    \[
    s \colon \Spec k (s) \rightarrow B
    \]
    from an algebraically closed field $k(s).$
    For any scheme $X$ over $B$, the geometric fiber $X_s$ over $s$ is the fiber product $X \times_B\Spec k(s)$ via $s$.
    \item 
    For a field $k$, we frequently denote its characteristic by $p$.
    Unless otherwise noted, $p$ can be $0$.
    \item 
    For any number field $F$ and any $p$-adic field $K$, we denote their ring of integers by $\cO_F$ and $\cO_K$, respectively.
    For a finite set $S$ of finite places of $F$, we denote the ring of $S$-integers of $F$ by $\cO_{F,S}$. 
    \item 
    Let $F$ be a number field, and $\p$ a finite place of $F$.
    We denote the $\p$-adic completion of $F$ by $F_{\p}$, its ring of integers by $\cO_{F,\p}$,
    and the localization of $\cO_{F}$ at $\p$ by $\cO_{F, (\p)}$.
    We denote the normalized additive valuation on $F$ or $F_{\p}$ associated with $\p$ by $v_\p$.
    \item
    For a discrete valuation ring $R$, we denote its completion by $\widehat{R}$.
    \item 
    For a field $k$, we denote its separable closure by $\ksep$.
    We denote its absolute Galois group $\Gal (\ksep/k)$ by $G_k$.
    \item 
    For a smooth projective variety over a field $k$, we denote its $i$-th Betti number by $b_i (X)$.
\end{itemize}

\section{Preliminaries}
\label{section:preliminaries}

\subsection{Basic definitions and terminologies}

\begin{defn}
\item 
Let $k$ be an algebraically closed field, and $B$ a scheme.
\begin{enumerate}
\item 
A \emph{Fano variety} over $k$ is a smooth projective variety $X$ over $k$ such that the anti-canonical divisor $-K_{X}$ is ample.
A \emph{Fano scheme} over $B$ is a smooth projective scheme over $B$ such that any geometric fiber is a Fano variety. 
\item
Let $X$ be a Fano variety over $k$.
The index $i(X)$ of $X$ is a maximal integer $i$ such that
$-K_{X}$ is $i$-divisible in $\Pic (X)$.
\end{enumerate}
\end{defn}

\begin{defn}
Let $B$ be a scheme, and $f \colon X \rightarrow S$ a smooth projective morphism.
\begin{enumerate}
\item
We say $f$ is a $V_5$-scheme (or quintic del Pezzo threefold) when any geometric fiber $X_s$ over $s$ is a Fano threefold such that $\rho(X_s) =1$, $i(X_s) =2$, and $(-K_{X_{s}})^3 =40$.
\item 
We say $f$ is a $V_{22}$-scheme (or prime Fano threefold of genus 12) when any geometric fiber $X_s$ over $s$ is a Fano threefold such that $\rho(X_s) =1$, $i(X_s) =1$, and $(-K_{X_{s}})^{3} =22.$
\end{enumerate}
When $B$ is a spectrum of a field $k$, we simply say that $Y$ is a $V_5$-variety or a $V_{22}$-variety over $k$.
\end{defn}

\begin{defn}
\label{defn:degree}
\begin{enumerate}
    \item 
Let $X$ be a Fano variety of index $i$
over an algebraically closed field $k$.
Assume that $H:=-K_{X}/i$ is very ample.
Let $C \subset X$ be a 1-dimensional integral closed subscheme.
We say $C$ is a \emph{curve of degree $d$} when 
\[
(C, -K_{X}/i) = (C,H)  =d.
\]
\item 
Let $B$ be a scheme, and $X$ a Fano scheme over $B$.
Assume that any geometric fiber of $X \rightarrow B$ satisfies the assumption of (1).
Let $C \subset X$ be a closed subscheme which is flat over $B$.
We say $C$ is a \emph{(relative) curve of degree $d$} when $C_{s} \subset X_{s}$ is a curve of degree $d$ for any geometric point $s$ in $B$.
\end{enumerate}
We frequently use terms \emph{line}, \emph{conic}, \emph{cubic curve}, and so on, instead of \emph{curve of degree $d$}.
\end{defn}

\begin{rem}
\label{rem:genus0rational}
Let $X$ be a Fano variety over a field $k$, and $i$ be the geometric index $i(X_{\overline{k}})$.
Assume that $H= -K_X/i$ is defined over $k$ and is very ample.
Let $C \subset X$ a smooth curve of degree $d \geq 1$.
If $g(C)=0$ and $d$ is odd, then $C$ is automatically a rational curve, 
since $\Pic (C)$ contains $\cO(d)$ and $\cO(2g(C)-2) \simeq \cO(K_{C})$.
On the other hand, if $d$ is even, $C$ may not be rational over $k$ even when $g(C)=0$.
Therefore, when $d$ is even, we will distinguish between the terms \emph{genus $0$} and \emph{rational}.
Note that,
\begin{itemize}
    \item 
    When $d=1$, $C$ is always rational.
    \item 
    When $d=2$, $C$ is always of genus 0 (though possibly non-rational).
\end{itemize}
\end{rem}

\begin{defn}
\label{defn:Sigma}
Let $Y$ be a $V_5$-scheme over an integral scheme $B$
and $X$ a $V_{22}$-scheme over $B$.
\begin{enumerate}
    \item 
We denote the Hilbert schemes of relative lines on $Y$ and $X$ over $B$ by
$\Sigma(Y/B)$ and $\Sigma(X/B)$, respectively.
    \item 
    We denote the universal line over $\Sigma(Y/B)$ by $\mathcal{U}_{Y/B}$.
    Let   $p \colon \mathcal{U}_{Y/B} \rightarrow Y$ and $q \colon \mathcal{U}_{Y/B} \rightarrow \Sigma (Y/B)$
    be the natural projections.
    \item 
    For a flat closed subscheme $Z\subset Y$ over $B$,
    we denote by $\Sigma_{Z} (Y/B)$ the closed subscheme $q (p^{-1}(Z))$ with the reduced structure .
    \footnote{In the literature (\cite{Kuznetsov-Prokhorov-Shramov}), the authors define $\Sigma_{Z} (Y/B)$ as a possibly non-reduced scheme and study its non-reduced structure in detail.
However, as considering only the reduction suffices for our purposes, we adopt this definition.}
\end{enumerate} 
We omit the base $B$ from the notation when $B$ is a spectrum of a field, and write simply $\Sigma(Y)$ etc.
\end{defn}

\subsection{Reviews of $V_5$-varieties}

In this subsection, we recall the basic properties of $V_{5}$-varieties (and $V_5$-schemes) obtained in \cite{V5}.

Recall that a $V_5$-variety $Y$ is embedded into $\P^6$ via the fundamental linear system $\left|-K_Y/2\right|$.
More generally, for a $V_5$-scheme $\pi \colon Y \rightarrow B$ over $B$, the relatively ample generator $H \in \Pic_{Y/B}(B)$ is a \emph{line bundle}, and defines an embedding $Y \hookrightarrow \P_B (\pi_{\ast} H)$
into a six-dimensional projective bundle $\P_B (\pi_{\ast} H) \rightarrow B$ (see \cite[Lemma 2.4]{V5}).
The following provides a model of $V_5$, which plays an important role in the present paper.

\begin{defn-prop}[{\cite[Definition-Proposition 3.12]{V5}}]
Let $B$ be a scheme.
The  closed subscheme $Y_B \subset \P_B^6$ defined by the following homogeneous polynomials is a $V_5$-scheme over $B$:
\begin{equation}
\label{eqn:W5}
\begin{cases}
a_0 a_4 - a_1 a_3 + a_2^2,\\
a_0 a_5 - a_1 a_4 + a_2 a_3,\\
a_0 a_6 - a_2 a_4 + a_3^2,\\
a_1 a_6 - a_2 a_5 + a_3 a_4,\\
a_2 a_6 - a_3 a_5 + a_4^2.
\end{cases}
\end{equation}
Moreover, the hyperplane section of $Y_B \subset \P^6_B$ corresponds to the ample generator of $\Pic_{Y/B} (B)$.
We say a $V_5$-scheme $Y$ over $B$ is a \emph{split $V_5$-scheme over $B$}
when $Y$ is isomorphic to $Y_{B}$ as above.
Moreover, any $V_5$-scheme over an algebraically closed field or a finite field $k$ is split.
\end{defn-prop}

The following describes the automorphism groups of the split $V_5$-schemes.
\begin{prop}[{\cite[Proposition 4.1 and Remark 4.4]{V5}}]
\label{prop:W5automorphisms}
Let $B$ be a scheme, and $Y \subset \P^6_B$ the split $V_5$-scheme defined by (\ref{eqn:W5}) over $B$.
\begin{enumerate}
    \item 
    Suppose that $B = \Spec \Z[1/2]$. 
    Then we have an isomorphism $\PGL_{2,B} \simeq \Aut_{Y/B}$ such that the composite $\sigma \colon \PGL_{2,B} \simeq \Aut_{Y/B} \hookrightarrow \PGL_{7,B}$ is given by
\[
\begin{pmatrix}
a & b \\
c & d \\
\end{pmatrix}
\mapsto
\]
\begin{equation}
\label{eqn:actionPGL2}
\begin{psmallmatrix}
a^6 & 2a^5b & 10a^4b^2 & 20a^3b^3 & 20a^2b^4 & 8ab^5 & 8b^6 \\
3a^5c &  5a^4bc + a^5d & 20a^3b^2c + 10a^4bd & 30a^2b^3c + 30a^3b^2d & 20ab^4c + 40a^2b^3d & 4b^5c + 20ab^4d & 24b^5d \\
\frac{3}{2}a^4c^2 & 2a^3bc^2 + a^4cd & 6a^2b^2c^2 + 8a^3bcd + a^4d^2 & 6ab^3c^2 + 18a^2b^2cd + 6a^3bd^2 & 2b^4c^2 + 16ab^3cd +12 a^2b^2d^2 & 4b^4cd + 8ab^3d^2 & 12b^4d^2\\
a^3c^3 & a^2bc^3 + a^3c^2d & 2ab^2c^3 + 6a^2bc^2d +2a^3cd^2 &  b^3c^3 + 9ab^2c^2d + 9a^2bcd^2 + a^3d^3 & 4b^3c^2d + 12ab^2cd^2 + 4a^2bd^3 & 4b^3cd^2 + 4ab^2d^3 & 8b^3d^3 \\
\frac{3}{4}a^2c^4 & \frac{1}{2}abc^4 + a^2c^3d & \frac{1}{2}b^2c^4 + 4abc^3d + 3a^2c^2d^2 & 3b^2c^3d + 9abc^2d^2 + 3a^2cd^3 & 6b^2c^2d^2 + 8abcd^3 + a^2d^4 & 4b^2cd^3 + 2abd^4 & 6b^2d^4 \\
\frac{3}{4}ac^5 & \frac{1}{4}bc^5 + \frac{5}{4}ac^4d & \frac{5}{2}bc^4 + 5ac^3d^2 & \frac{15}{2}bc^3d^2 + \frac{15}{2}ac^2d^3 & 10bc^2d^3 + 5acd^4 & 5bcd^4 + ad^5 & 6bd^5 \\
\frac{1}{8}c^6 & \frac{1}{4}c^5d & \frac{5}{4}c^4d^2 &  \frac{5}{2}c^3d^3 & \frac{5}{2}c^2d^4 & cd^5 & d^6 \\
\end{psmallmatrix}
\end{equation}
\item
Suppose that $B = \Spec \F_{2}$.
Then we have an isomorphism $\SL_{2,B} \simeq \Aut_{Y/B, \red}$ 
\footnote{In \cite[Proposition 4.1]{V5}, it is shown that $\Aut_{Y/k}$ is non-reduced in characteristic two. However, since the non-reduced structure is irrelevant to this paper, we consider only the reduced part.}
such that the composite $\sigma' \colon \SL_{2,B} \simeq \Aut_{Y/B, \red} \hookrightarrow \PGL_{7,B}$ is given by
\begin{equation}
\label{eqn:sigma'}
\begin{pmatrix}
a & b \\
c & d \\
\end{pmatrix}
\mapsto 
\begin{pmatrix}
a^3 & 0 &a^2b &0 & ab^2 &0 &b^3\\
0&a^2&0&ab&0&b^2&0\\
a^2c&0&a^2d&0&b^2c&0&b^2d\\
0&0&0&1&0&0&0\\
ac^2&0&bc^2&0&ad^2&0&bd^2\\
0&c^2&0&cd&0&d^2&0\\
c^3&0&c^2d&0&cd^2&0&d^3\\
\end{pmatrix}.
\end{equation}
\item 
Suppose that $B =\Spec  \Z[\sqrt 2]$.
Let $c \colon \PGL_{2,\eta_B} \rightarrow \PGL_{2, \eta_B}$ be the conjugation by
$
\begin{pmatrix}
1 & 0 \\
0 & \sqrt2
\end{pmatrix}.
$
Then the composite of $c$ and \eqref{eqn:actionPGL2} extends to a map
\[
\widetilde{\sigma}\colon \PGL_{2,B} \rightarrow \PGL_{7,B},
\]
and $\widetilde{\sigma}$ factors through $\Aut_{Y/B} \hookrightarrow \PGL_{7, B}$.
On the fiber over $\F_2$, the map $\widetilde{\sigma}_{\F_2}$ factors as $\PGL_{2, \F_{2}}  \xrightarrow{F} \SL_{2,\F_2} \xrightarrow{\sigma'} \PGL_{7,\F_{2}}$,
where $F$ is induced by the Frobenius morphism
$\SL_{2,\F_{2}} \xrightarrow{F} \SL_{2,\F_{2}}$
\end{enumerate}
\end{prop}

Next, we review the orbit decomposition of the split $V_{5}$-variety separately for the case of characteristic not equal to two and the case of characteristic two.
\begin{prop}[{\cite[Proposition 6.1]{V5}}]
\label{prop:orbitp>2}
Let $Y \subset \P^6_k$ be the split $V_5$-variety over a field $k$ defined by (\ref{eqn:W5}).
We suppose that $\chara k \neq 2$.
\begin{enumerate}
    \item 
There exists a stratification
$Y = O_3^{(p)} \sqcup O_2^{(p)} \sqcup O_1^{(p)}$
that gives the orbit decomposition with respect to $\sigma$ after the base change to $\overline{k}$.
Here, $O_i^{(p)}$ is an $i$-dimensional orbit.
Each orbit contains a $k$-rational point as follows:
\begin{align*}
& (0:1:0:0:0:1:0) \in O_3^{(p)}, \\
& (0:0:0:0:0:1:0) \in O_2^{(p)}, \\
& (0:0:0:0:0:0:1) \in O_1^{(p)}.
\end{align*}
    \item
$D \coloneqq  O_2^{(p)} \sqcup O_1^{(p)}$ is a prime divisor on $Y$, which is defined over $\Z$ with the equation
\begin{equation}
\label{eqn:D}
    5a_2a_4-4a_1a_5+27a_0a_6=0.
\end{equation}

\item
The normalization of $D$ is $ \P^1_{k} \times_k \P^1_k $.
More precisely, the following map $\nu \colon \P^1_{k} \times_k \P^1_k \rightarrow D$ gives the normalization  of $D$:
\begin{equation}
\label{eqn:nu}
\begin{aligned}
& ((\alpha: \gamma), (\beta: \delta))\\ 
&\mapsto
(8\alpha\beta^5: 4\beta^5\gamma + 20\alpha\beta^4\delta:4\beta^4\gamma\delta + 8\alpha\beta^3\delta^2: 4\beta^3\gamma\delta^2 + 4\alpha\beta^2\delta^3: 4\beta^2\gamma\delta^3 + 2\alpha\beta\delta^4 : 5\beta\gamma\delta^4 + \alpha\delta^5 : \gamma\delta^5).
\end{aligned}
\end{equation}
Moreover, the map $\nu$ is $\PGL_{2,k}$-equivariant via the isomorphism $\PGL_{2,k} \simeq \Aut_{Y/k}$ in Proposition \ref{prop:W5automorphisms}, and the action of $\PGL_{2,k}$ on $\P^1_k \times_k \P^1_k$ given by the product of
\begin{equation}
\label{eqn:PGL2P1action}
\begin{pmatrix}
    a & b \\
    c & d
\end{pmatrix}
\colon \begin{pmatrix}
    \alpha \\
    \gamma
\end{pmatrix}
\mapsto
\begin{pmatrix}
    a & b \\
    c & d
\end{pmatrix}
\begin{pmatrix}
    \alpha\\
    \gamma
\end{pmatrix}.
\end{equation}
\end{enumerate}
\end{prop}

\begin{prop}[{\cite[Proposition 6.3]{V5}}]
\label{prop:orbitp=2}
Let $Y \subset \P^6_k$ be the split $V_5$-variety over a field $k$ defined by (\ref{eqn:W5}).
Suppose $\chara k = 2$.
\begin{enumerate}
    \item 
    There exists a stratification
    \[
    Y = O_3^{(2)} \bigsqcup O_2^{(2)} \bigsqcup O_1^{(2)} \bigsqcup O_1^{(2)'},
    \]
    that gives the orbit decomposition with respect to the action $\sigma'$ after the base change to $\overline{k}$.
    Here, $O_i^{(2)}$ and $O_i^{(2)'}$ are $i$-dimensional orbits.
    Moreover, each stratum contains a $k$-rational point as follows:
     \begin{align*}
&(1:0:0:1:0:0:1) \in O_3^{(2)},\\
&(0:0:0:0:0:1:1) \in O_2^{(2)},\\
&(0:0:0:0:0:0:1) \in O_1^{(2)},\\
&(0:0:0:0:0:1:0)  \in O_1^{(2)'}.
\end{align*}
    \item 
    Let $D$ be the divisor on $Y$ defined by (\ref{eqn:D}).
    Then $D$ is a double-divisor, and we have $D_{\mathrm{red}} = O_2^{(2)} \sqcup O_1^{(2)} \sqcup O_1^{(2)'}$.
    \item
The normalization of $D_{\mathrm{red}}$ is $\mathrm{Bl}_{(0:0:1)} \P^2_{k}$.
More precisely, the normalization map $\nu' \colon \mathrm{Bl}_{(0:0:1)} \P^2_{k} \rightarrow D_{\red}$ is induced by the following map $\P^2_{k} \dashrightarrow D_{\red}$:
\[
 (x:y:z) \mapsto (x^3 : x^2 z: x^2y:0:xy^2:y^2z:y^3).
\]
The map $\nu'$ is $\SL_{2,k}$-equivariant via the isomorphism $\SL_{2,k} \simeq \Aut_{Y/k, \red}$ in Proposition \ref{prop:W5automorphisms}, and the action of $\SL_{2,k}$ on the domain induced by that on $\P^2_k$  given as
\begin{equation}
\label{eqn:SL2P2action}
\begin{pmatrix}
    a & b \\
    c & d
\end{pmatrix}
\colon \begin{pmatrix}
    x \\
    y \\
    z
\end{pmatrix}
\mapsto
\begin{pmatrix}
    a & b &0 \\
    c & d &0 \\
    0 & 0 & 1
\end{pmatrix}
\begin{pmatrix}
    x\\
    y\\
    z
\end{pmatrix}.
\end{equation}

\item $O_1^{(2)'}$ is a twisted cubic given by the image of $\nu'|_{\{z=0\}} \colon \{z =0 \} \rightarrow D_{\F_{2}, \red}$ that sends
\[
(x:y:0) \mapsto (x^3:0:x^2y:0:xy^2:0:y^3).
\]
$O_{1}^{(2)}$ is the line given by the image of $\nu'|_{E} \colon E \simeq \P^1 \rightarrow D_{\F_{2},\red}$ that sends
\[
 (x,y) \mapsto (0:x^2:0:0:0:y^2:0).
\]
\item  $\nu'$ is a bijective.
\end{enumerate}
\end{prop}

Finally, we summarize the classification and some properties of lines on $V_5$-varieties.

\begin{prop}[{\cite[(4.0.4) and Corollary 5.6]{V5}}]
\label{prop:SigmaofW5}
Let $k$ be an algebraically closed field, and $Y \subset \P^6$ the split $V_5$-variety over $k$ defined by (\ref{eqn:W5}).
Then the following hold.
\begin{enumerate}
\item 
We have an isomorphism $f \colon \Sigma(Y) \simeq \P^2_k$.
Moreover, if $\chara k \neq 2$, the isomorphism $f$ is $\PGL_{2,k}$-equivariant
with respect to the isomorphism in Proposition~\ref{prop:W5automorphisms};
if $\chara k = 2$, it is $\SL_{2,k}$-equivariant.
Here, when $\chara k \neq 2,$ the action of $\PGL_{2,k}$ on $\P^2$ is given by
\[
 \begin{pmatrix}
 a&b\\c&d
\end{pmatrix}
\mapsto
\begin{pmatrix}
 a^2&2b^2&-2ab\\
 \frac{1}{2}c^2&d^2&-cd\\
 -ac&-2bd&ad+bc
\end{pmatrix}.
\]
When $\chara k =2$, the action of $\SL_{2,k}$ on $\P^2$ is given by \eqref{eqn:SL2P2action}.
For any point $(x:y:z) \in \P^2 (k)$, we denote the line corresponding to $f^{-1} (x:y:z)$ by $l_{(x:y:z)}$.

\item 
Suppose that $\chara k \neq 2$.
Then there exist two $\Aut_{Y/k}$-orbits on $\Sigma (Y) \simeq \P^2_k$, which is written as
\[
V (z^2 -2xy) \bigsqcup (\P^2 \setminus V (z^2-2xy)).
\]
\begin{enumerate}
\item The line $[l_{(0:1:0)}] \in V (z^2 -2xy)$ is the image of $\P^1 \to Y \subset \P^6$ defined by
\begin{equation}
\label{eqn:line56}
(s\colon t) \mapsto (0: 0: 0: 0: 0:s:t).
\end{equation}
\item The line $[l_{(0:0:1)}] \not \in V (z^2 -2xy)$ is the image of  $\P^1 \to Y$ defined by
\begin{equation}
\label{eqn:line15}
(s\colon t) \mapsto (0: s: 0: 0: 0:t:0).
\end{equation}
\end{enumerate}

\item 
Suppose that $\chara k =2$.
Then there exist three $\Aut_{Y/k,\red}$-orbits on $\Sigma (Y) \simeq \P^2$, which is written as 
\[
\{(0:0:1)\} \bigsqcup V(z) \bigsqcup \P^2 \setminus (V(z) \sqcup \{(0:0:1)\}).
\]
\begin{enumerate}
 \item The line $l_{(0:0:1)}$ is  $O_1^{(2)'}$, which is the image of $\P^1 \to Y$ defined by
\begin{equation}
\label{eqn:line15p=2}
(s\colon t) \mapsto (0: s: 0: 0: 0:t:0).
\end{equation}
 \item The line $[l_{(0:1:0)}]\in V(z)$  is the image of  $\P^1 \to Y$ defined by
\begin{equation}
\label{eqn:line56p=2}
(s\colon t) \mapsto (0: 0: 0: 0: 0:s:t).
\end{equation}
\item The line $[l_{(1:1:1)}]  \not \in V(z) \sqcup \{(0:0:1)\}$  is the image $\P^1 \to Y$ defined by
\begin{equation}
\label{eqn:line1p2}
    (s:t ) \mapsto (s:t:t:s+t:t:t:s).
\end{equation}
\end{enumerate}
\end{enumerate}
\end{prop}

\begin{defn}
Let $Y \subset \P^6$ be the split $V_5$-variety over an algebraically closed field $k$, and let the notations be as in Proposition~\ref{prop:SigmaofW5}.
\begin{enumerate}
    \item Suppose that $\chara k \neq 2$.
    A line $l \subset Y$ is said to be \emph{ordinary} if $[l]
     \in \Sigma(Y)$ is contained in the open orbit.
     Otherwise $l$ is called \emph{special}.
     \item Similarly, when $\chara k = 2$, an \emph{ordinary} line $l \subset Y$ is a line $[l] \in \Sigma(Y)$ which is contained in the open orbit.
     A \emph{Special} line $l$ is a line $[l] \in \Sigma(Y)$ which is in the $1$-dimensional orbit $V(z)$.
     The \emph{exceptional} line is the line $l_{(0:0:1)}$.
\end{enumerate}
\end{defn}

\begin{prop}[{\cite[Propositions 6.7 and 6.10]{V5}}]
\label{prop:numberoflines}
Let $Y \subset \P^6_k$ be the split $V_5$-variety over an algebraically closed field $k$ defined by (\ref{eqn:W5}).
We consider the actions $\sigma$ or $\sigma'$ in Proposition \ref{prop:W5automorphisms}.
Let $P \in Y(k)$.
\begin{enumerate}
    \item 
    Suppose that $\chara k \neq 2$.
Then the number of lines passing through $P$ is 
\[
\begin{cases}
3  & \textup{if } P\in O_3^{(p)}(k), \\
2  & \textup{if } P\in O_2^{(p)}(k), \\
1  & \textup{if } P\in O_1^{(p)}(k).
\end{cases}
\]
In the first case, three lines are ordinary.
In the second case, one line is ordinary, and the other is special.
In the third case, the line is special.
\item 
Suppose that $\chara k = 2$.
Then the number of lines passing through $P$ is 
\[
\begin{cases}
3  & \textup{if } P\in O_3^{(2)}(k), \\
2  & \textup{if } P\in O_2^{(2)}(k), \\
1  & \textup{if } P\in O_1^{(2)}(k), \\
2  & \textup{if } P \in O_1^{(2)'}(k).
\end{cases}
\]
In the first case, three lines are ordinary.
In the second case, one line is ordinary, and the other is special.
In the third case, the line is special.
In the fourth case, one line is specail, and the other is exceptional.
\end{enumerate}
\end{prop}

Finally, we recall the following useful lemma.

\begin{lem}
\label{lem:split_after_finite_extension}
\begin{enumerate}
\item
Let $R$ be a local ring, and $\mathcal{Y}$ a $V_5$-scheme over $R$.
Then there exists a finite free extension $R'$ of $R$ such that $\mathcal{Y}_{R'}$ is a split $V_5$-scheme.
\item 
Let $R$ be a discrete valuation with $2 \in R^{\times},$ $\mathcal{X}$ a $V_{5}$-scheme over $R$. 
If $\mathcal{X}_{\Frac R}$ is a split $V_5$-variety over $\Frac R$, then $\mathcal{X}$ is a split $V_5$-scheme over $R$.
\end{enumerate}
\end{lem}

\begin{proof}
(1) follows from the proof of \cite[Lemma 7.10]{V5}.
(2) follows from \cite[Proposition 7.9]{V5} since the corresponding Brauer-Severi curve admits a section over $R$.
\end{proof}

\section{Quintic curves with large stabilizers in $V_5$ varieties}
\label{section:quintics_in_V5}

Smooth rational quintic curves in $V_5$-varieties are important to study $V_{22}$-varieties (cf.\ Subsection \ref{subsection:tworaylines}).
In this section, we classify smooth rational quintic curves whose stabilizers have positive dimensions. 
Moreover, we study the smooth degeneration of such quintic curves.
Since it is necessary to examine the action of automorphisms in detail, we divide the discussion here into the cases where the characteristic is not two and where it is two.

\subsection{Characteristic $\neq 2$}

\begin{thm}[Classification of quintic curves]
\label{thm:BGaquintic>2}
Let $k$ be a field of characteristic $p\neq 2$,
$Y=Y_{k} \subset \P^6 = \P^6_k$ the split $V_5$-variety over $k$ given by (\ref{eqn:W5}), $\sigma$ the action of $\Aut_{Y/k} \simeq \PGL_{2,k}$ on $Y$ given by Proposition \ref{prop:W5automorphisms}, and $\nu = \nu_{k} \colon \P^1 \times \P^1 \to D$ the normalization as in Proposition \ref{prop:orbitp>2}.

Let $Z \subset Y$ be a smooth rational quintic curve
and $G$ the stabilizer of $Z$ with respect to $\sigma$.
Suppose that $\dim G \geq 1$.
Then $G$ is reduced, and exactly one of the following holds.
\begin{enumerate}
    \item 
We have $p\neq 5$, $G \subset \PGL_{2,k}$ is a Borel subgroup, and $Z$ is equal to $\nu (\{ P \} \times \P^1) $ for some $P \in \P^1(k)$.
In particular, by an element of $\PGL_{2,k} (k)$ via the action $\sigma$, $Z$ is mapped to $Z_{\mathrm{MU}} :=\nu (\{(1:0)\}\times \P^1)$, which is a smooth rational quintic curve defined by the map
\begin{equation}
\label{eqn:ZMU}
       \P^1 \rightarrow Y \subset \P^6,\quad
        (s,t) \mapsto (8s^5: 20s^4t: 8s^3t^2: 4s^2t^3: 2st^4: t^5:0).
\end{equation}
Moreover, $Z$ is equal to (\ref{eqn:ZMU}) if and only if $G$ is equal to the upper triangular Borel subgroup of $\PGL_{2,k}$.
    \item 
We have $p \neq 5$. $G \simeq \Ga \rtimes \mu_4$ with respect to 
\begin{equation}
\label{eqn:semi-directGa}
\mu_4 \rightarrow \Aut_{\Ga/k},\ \zeta \mapsto (\ast \mapsto \zeta \ast ),
\end{equation}
and 
$Z$ is mapped to
the smooth rational quintic curve $Z_{\xi}^{(\mathrm{a})}$ defined by
\begin{equation}
\label{eqn:ZGageneral}
\P^1 \rightarrow Y \subset \P^6, \quad
(s\colon t) \mapsto 
(8s^5+2 \xi st^4: 20s^4t+ \xi t^5: 8s^3t^2: 4s^2t^3: 2st^4:t^5:0)
\end{equation}
for some $\xi \in k^{\times}$ by an element of $\PGL_{2.k} (k)$ via the action $\sigma$.
Moreover, if $k= \overline{k}$, then $Z$ is mapped to the smooth rational quintic curve $Z_1^{(\mathrm{a})}$ by an element of $\PGL_{2,k} (k)$ via the action $\sigma$.
Furthermore, $Z$ is equal to $Z_{\xi}^{(\mathrm{a})}$ for some $\xi \in k^{\times}$ if and only if $G$ is equal to the unipotent subgroup of upper triangular matrices in $\PGL_{2,k}$.
    \item
We have $G \simeq \Gm$, 
and
$Z$ is mapped to $Z_u$ for some
\[
u =(u_0: u_1) \in \P^1 (k) \setminus \{ (1:0), (1:1), (0:1), (5:4)\}
\]
(note that, $(4:5) = (1:0)$ if $p=5$) by an element of $\PGL_{2,k} (k)$ via the action $\sigma$.
Here, $Z_u$ is the smooth rational quintic curve defined by the following morphism.
\begin{equation}
\label{eqn:ZGm}
\P^1 \rightarrow Y \subset \P^6,
\quad
(s\colon t) \mapsto 
(s^5 (u-1): s^4tu: s^3t^2: s^2 t^3: st^4: t^5:0 ),
\end{equation}
where we identify $u$ with $u_0/u_1$.
Moreover, $Z$ is equal to $Z_u$ or $Z_u'$ for some $u$ if and only if $G$ is equal to the diagonal torus of $\PGL_{2,k}$.
Here, $Z_u'$ is the smooth rational quintic curve $\sigma
\begin{pmatrix}
    0 &1\\
    1 &0
\end{pmatrix}
 (Z_u),
$
which is equal to the image of 
\begin{equation}
\label{eqn:ZGm'}
\P^1 \rightarrow Y \subset \P^6,
\quad 
(s\colon t) \mapsto 
(0:t^5:st^4: s^2t^3:s^3t^2:s^4tu:s^5(u-1)).
\end{equation}
In particular, for any $v \neq u,$ $Z_u$ cannot be mapped to $Z_{v}$ by an element of $\PGL_{2,k} (k)$ via the action $\sigma$.
\end{enumerate}
Conversely, the cases (1), (2), and (3) actually occur, i.e.,\ quintic curves appearing in these cases are smooth rational quintic curves with stabilizers as indicated, provided that the respective assumptions on $p$ are satisfied.
\end{thm}

\begin{proof}
Since the proof over a general field is a bit involved, we prove it in several steps.

\noindent{\bf Step A. The case where $\dim G \geq 2$.}\\
Essentially, the strategy is the same as in \cite[Lemma 5.2.9]{Kuznetsov-Prokhorov-Shramov}.
Clearly, we have $Z \subset O_1^{(p)} \sqcup O_2^{(p)} = D_{k}$.
Since $Z$ is a quintic curve, by the definition of $\nu$, $Z$ is given by $\nu (\{P\}\times \P^1 )$ for some $P \in \P^1 (k)$, which is mapped to 
$\nu(\{(1:0) \} \times \P^1)$ by an element of $\PGL_{2,k} (k)$ via the action $\sigma$.
By using (\ref{eqn:nu}), we can show that the equation of $\nu (\{ (1:0) \} \times \P^1)$ is given by (\ref{eqn:ZMU}).
Note that, if (and only if) $p=5$, then 
$\nu (\{P\} \times \P^1)$ is not smooth.
Since $\nu$ is $\PGL_{2,k}$-equivariant, we can easily show that $G$ is reduced and equal to a Borel subgroup of $\PGL_{2,k}$.
Moreover, since the stabilizer of $P$ is equal to the upper triangular Borel subgroup if and only if $P= (1:0)$, the final statement in (1) also holds true.
Therefore, (1) holds in this case.

\noindent{\bf Step B. The case where $H= G^\circ_{\red}$ is conjugate to the standard subgroup.}\\
By Step A, we may assume $\dim G =1$.
Here, we deal with the case where $H \coloneqq G^\circ_{\red}$ is conjugate under $G(k)$ to one of the following closed subscheme:
\begin{enumerate}[label=\textbf{(B-\arabic*)}]
    \item \label{B-1}
The unipotent subgroup $U=
\left\{\begin{pmatrix}
    1&b\\
    0&1\\
\end{pmatrix}\right\}
\hookrightarrow \PGL_{2,k}$.
    \item \label{B-2}
The diagonal torus $T=
\left\{\begin{pmatrix}
    a&0\\
    0&1\\
\end{pmatrix}\right\}
\hookrightarrow \PGL_{2,k}$.
\end{enumerate}
In this case $H$ is geometrically reduced, and thus $H \subset  \PGL_{2,k}$ is a subgroup scheme.

\noindent{\bf Case \ref{B-1}.}
Assume that $H =U$.
The restriction of $\sigma$ in Proposition \ref{prop:W5automorphisms} to $U$ is given by
\begin{eqnarray}
\label{eqn:unipaction}
\begin{pmatrix}
1 & b \\
0 & 1 \\
\end{pmatrix}
\mapsto
\begin{pmatrix}
1 & 2b & 10b^2 & 20b^3 & 20b^4 & 8b^5 & 8b^6 \\
0 & 1 & 10b & 30b^2 & 40b^3 & 20b^4 & 24b^5 \\
0 & 0 & 1 & 6b & 12b^2 & 8b^3 & 12b^4 \\
0 & 0 & 0 & 1 & 4b & 4b^2 & 8b^3 \\
0 & 0 & 0 & 0 & 1 & 2b & 6b^2 \\
0 & 0 & 0 & 0 & 0 & 1 & 6b \\
0 & 0 & 0 & 0 & 0 & 0 & 1 \\
\end{pmatrix}.
\end{eqnarray}
Therefore, clearly the action of $U$ on $Z$ is not trivial.
Take a point $P \in Z(k)$ with a 1-dimensional $U$-orbit.
Since $Z$ is a quintic curve, 
$P$ is contained in
\begin{equation}
\label{eqn:A2parammor}
\mathbb{A}^2 \overset{\Phi}\simeq Y \cap \{ a_6 =0\} \cap \{ a_5 \neq 0 \} \subset \P^6.
\end{equation}
Here, the isomorphism $\Phi$ is given by
\begin{equation}
\label{eqn:A2param}
(x,y) \mapsto (xy-y^5: x: y^3:y^2:y:1:0).
\end{equation}
Write $P= (x_Py_P-y_P^5: x_P: y_P^3:y_P^2:y_P:1:0)$.
Since we have
\begin{equation}
\label{eqn:actPxi}
\sigma 
\begin{pmatrix}
    1 & -(y_P/2) \\
    0 &1
\end{pmatrix}
( P )= 
(0: x_P - \frac{5}{4}y_P^4:0: 0: 0: 1:0),
\end{equation}
we may assume that 
\[
P= P_{\xi} := (0:\xi: 0:0:0:1:0)
\]
for some $\xi \in k$.
In this case, by (\ref{eqn:unipaction}), the quintic curve $Z$ is equal to $(\ref{eqn:ZGageneral})$ (we use the same equation even when $\xi =0$).
When $\xi =0$, this is equal to (\ref{eqn:ZMU}), which is a quintic curve in Step A.
Therefore, we have $\xi \in k^{\times}$ and $Z = Z_{\xi}^{(\a)}$.
Note that $Z_{\xi}^{(\a)}$ is smooth if and only if $p\neq5$.
Moreover, since we have
\[
\sigma 
\begin{pmatrix}
    \xi^{1/4} & 0 \\
    0 &1
\end{pmatrix} (P_1) = P_{\xi},
\]
$Z_{\xi}^{(\a)}$ is mapped to $Z_{1}^{(\a)}$ by an element of $\PGL_{2,k} (k)$ via the action $\sigma$ if $k=\overline k$.
Next, we shall compute the stabilizer $G$ of $Z=Z_{\xi}^{(\a)}$.
By computing the orbit of $(0:0:0:0:0:0:1)\in Z$, we have $G\subset B$, where $B$ is the upper triangular Borel subgroup.
By computing the $B$-orbit of $P_{\xi}$, we obtain $G \simeq U \rtimes \mu_4$, where $\mu_4$ is the kernel of fourth power map on the diagonal torus $T$.
By the above argument, the final statement of (2) also holds.
Therefore, (2) holds in this case.\\
\noindent{\bf Case \ref{B-2}.}
Assume that $H=T$.
The restriction of $\sigma$ in Proposition \ref{prop:W5automorphisms} to $T$ is given by
\begin{eqnarray}
\label{eqn:diagonalaction}
\begin{pmatrix}
a & 0 \\
0 & 1 \\
\end{pmatrix}
\mapsto
\diag(a^6, a^5, a^4, a^3, a^2, a, 1).
\end{eqnarray}
Therefore, the action of $T$ on $Z$ is not trivial.
Take a point $P \in Z(k)$ with a 1-dimensional $T$-orbit.
Since $Z$ is a quintic curve, $P$ is contained in either 
\begin{equation}
\label{eqn:formerY}
Y \cap \{ a_6 =0 \} \cap \{ a_5 \neq 0 \} \cap \{ a_0 \neq 0 \} \subset \P^6
\end{equation}
or 
\[
Y \cap \{ a_0=0 \} \cap \{ a_1 \neq 0 \} \cap \{ a_6 \neq 0 \} \subset \P^6.
\]
Since they are interchanged by the action of
\[
\begin{pmatrix}
    0 &1 \\
    1 &0\\
\end{pmatrix},
\]
which normalizes $T$, we may assume that $P$ lies in the former.
We write
\[
P =(x_P y_P-y_P^5:x_P:y_P^3:y_P^2:y_P:1:0)
\]
by using the parametrization given by (\ref{eqn:A2param}).
Then we have $y_P \neq 0$ and $x_P - y_P^4 \neq 0$
since $P$ is contained in (\ref{eqn:formerY}).
By replacing $P$ by $\sigma (\diag(1/y_P,1)) P$, we may assume that 
\[
P = (u-1: u :1:1:1:1:0)
\]
for some $u \in k \setminus \{1\}$.
In this case, by (\ref{eqn:diagonalaction}), $Z$ is equal to (\ref{eqn:ZGm}) (we use the same equation even when $u = 0$, $1$, $5/4$).
When $u = 5/4$, $Z$ is equal to (\ref{eqn:ZMU}), so we have $u \neq 5/4$.
When $u=0$, $Z$ is singular, so we have $u \neq 0$.
Therefore, $Z$ is equal to $Z_u$ as in the statement of (3).
Note that, conversely, (\ref{eqn:ZGm}) is a smooth rational quintic curve when $u\neq 0,1,5/4.$
Next, we shall compute the stabilizer $G$ of $Z =Z_u$.
By computing the orbit of $(1:0:0:0:0:0:0)$, we obtain $G \subset B$, where $B$ is the upper triangular Borel subgroup.
By computing the $B$-orbit of $(0:0:0:0:0:1:0)$, we obtain $G \simeq T$.
By the above argument, the final statement of (3) also holds.
Therefore, (3) holds in this case.\\
\noindent{\bf Step C. General case.}
Here we deal with the general case with $\dim G =1$.
If $k=\overline {k}$, then the assertion follows by Step B.
Note that, in this case, $G$ is reduced.
Let $k$ be a general field.
Since $G_{\overline{k}}$ is the stabilizer of $Z_{\overline{k}}$ with respect to $\sigma_{\overline{k}}$, $G_{\overline{k}}$ is also reduced.\footnote{Since the reduced part of a group scheme is not necessarily a group scheme in general, we should treat the case when $k= \overline{k}$ firstly.}
Therefore, $G$ is geometrically reduced.
For the identity component $H$ of $G$, we have one of the following;
\begin{enumerate}[label=\textbf{(C-\arabic*)}]
    \item \label{C-1}
$H_{\overline{k}} \simeq \G_{a,\overline{k}}$.
    \item \label{C-2}
$H_{\overline{k}} \simeq \G_{m,\overline{k}}$.
\end{enumerate}

\noindent{\bf Case \ref{C-1}.}
We show that $H$ is conjugate to $U$ by an element of $\PGL_{2,k} (k)$.
By \cite[Corollary 7.2.4]{conrad2020reductive}, it suffices to show that $H \simeq \Ga$.
Note that, the scheme theoretic intersection $Z_{\overline{k}} \cap D_{\overline{k}}$ is one 10-fold point, which is the $H_{\overline{k}}$-fixed point of $Z_{\overline{k}}$
(this can be checked by the direct computation using (\ref{eqn:ZGageneral}) and (\ref{eqn:D})).
Therefore, $(Z \cap D)_{\red}$ is one $k$-rational point $Q$.
We may take a point $P \in Z \setminus \{Q\}  (k) \simeq \mathbb{A}^1 (k)$.
The $H$-action on $Z$ induces an isomorphism  $H = H\cdot P \simeq \mathbb{A}^1$ of $k$-schemes.
Since $\Pic (\mathbb{A}^1) = 0$, by \cite[Lemma 9.3]{Totaro-pseudo-abelian}, we obtain $H \simeq \Ga$ as desired.

\noindent{\bf Case \ref{C-2}.}
we show that $H=G$ is conjugate to the diagonal torus $T$ by an element of $\PGL_{2,k} (k)$.
By \cite[Corollary 7.2.4]{conrad2020reductive}, it suffices to show that $H \simeq \Gm$.
Note that we have $H_{\ksep} \simeq \Gm$ \cite[Proposition 8.11]{Borellinear}.
Therefore, $Z_{\ksep}$ satisfies (2).
In particular, $Z_{\ksep}$ contains two $\ksep$-rational points that are $H_{\ksep}$-fixed
(to show this, we may assume that $Z_{\ksep}$ is given by (\ref{eqn:ZGm}) for $u \in \ksep \setminus \{0,1,5/4\}$, and then the two points are $(1:0:0:0:0:0:0)$ and $(0:0:0:0:0:1:0)$).
Moreover, these two points are given by 
$(Z_{\ksep} \cap O^{(p)}_{1,\ksep} )(\ksep)$ and
$(Z_{\ksep} \cap O^{(p)}_{2,\ksep}) (\ksep)$.
Therefore, these two points descend to two $k$-rational points $Q_1$, $Q_2$ on $Z$.\footnote{Here, in order to ensure that no inseparable extension occurs, we discuss $\ksep$ first.}
We may take a point $P \in Z \setminus \{Q_1, Q_2\}$.
Then $H$-action induces an isomorphism$H = H\cdot P \simeq \mathbb{A}^1 \setminus  \{0\}$
of $k$-schemes.
Note that unitary group is not isomorphic to $\mathbb{A}^1 \setminus \{0\}$.
Indeed, the unitary group with respect to $k(\sqrt d)/k$ is isomorphic to 
\[
\Spec k[x,y]/(x^2-dy^2-1),
\]
whose complement of regular compactification is $\Spec k(\sqrt d)$.
Therefore, we have $H \simeq \Gm$ as desired.
\end{proof}

\begin{rem}
\label{rem:stabilizeractiononZ_p>2}
In the above proof, we can show that the restriction morphism $G \rightarrow \Aut_{Z/k}$ is a closed immersion.
\end{rem}

\begin{defn}
\label{defn:typeofquintic}
Under the assumption in Theorem \ref{thm:BGaquintic>2}, we say:
\begin{enumerate}
    \item 
    $Z$ is of Mukai--Umemura type when Theorem \ref{thm:BGaquintic>2} (1) holds (i.e.,\ $G$ is a Borel subgroup).
    \item 
    $Z$ is of $\Ga$-type when Theorem \ref{thm:BGaquintic>2} (2) holds (i.e.,\ $G$ is isomorphic to $\Ga \rtimes \mu_4$).
    \item
    $Z$ is of $\Gm$-type when Theorem \ref{thm:BGaquintic>2} (3) holds (i.e.,\ $G$ is isomorphic to $\Gm$).
\end{enumerate}
Note that, by Theorem \ref{thm:BGaquintic>2}, precisely one of (1), (2), and (3) holds.
\end{defn}

The following describes the lines that intersect with the quintic curves $Z$  in Theorem \ref{thm:BGaquintic>2}:
\begin{prop}
\label{prop:SigmaZp>2}
Let the notations be as in Theorem \ref{thm:BGaquintic>2}. 
Further, we assume that $k=\overline{k}$.
As in Definition \ref{defn:Sigma}, we denote the universal line on the Hilbert scheme of lines on $Y$ by
$q \colon \mathcal{U}_{Y} \rightarrow \Sigma(Y)$,
and let $p \colon \mathcal{U}_{Y} \rightarrow Y$ be the natural projection.
Let $f\colon \P^2 \simeq \Sigma(Y)$ be as in Proposition \ref{prop:SigmaofW5}. 
Then we have the following.
\begin{enumerate}
\item 
Suppose that $p \neq 5$, and $Z$ is of Mukai--Umemura type.
Then $\Sigma_{Z} (Y) \subset \Sigma (Y) \overset{f^{-1}}\simeq \P^2$ is the union of a line $L$ and a conic, which are tangent to each other at a point $[l] \in \Sigma (Y)$.
\footnote{If we equip a non-reduced scheme structure on $\Sigma_Z (Y)$ as in \cite{Kuznetsov-Prokhorov-Shramov}, then this conic becomes a double-conic (see \cite[Lemma 5.4.1]{Kuznetsov-Prokhorov-Shramov}) at least in characteristic zero.}
\item 
Suppose that $p \neq 5$, and $Z$ is of $\Ga$-type.
Then $\Sigma_{Z} (Y) \subset \Sigma (Y) \overset{f^{-1}}\simeq \P^2$ is a union of a line $L$ and two conics, where two conics are $4$-tangent at one point $[l] \in \Sigma (Y)$, and the line $L$ is tangent to both conics at the point $[l] \in \Sigma (Y)$.
\item
Suppose that $Z$ is of $\Gm$-type.
Then $\Sigma_{Z} (Y) \subset \Sigma (Y) \overset{f^{-1}}\simeq \P^2$ is a union of a line $L$ and two conics, where two conics are tangent at two points $[l] \in \Sigma (Y)$ and $[l'] \in \Sigma (Y)$, and the line $L$ is tangent to both conics at the point $[l] \in \Sigma (Y)$. 
\end{enumerate}
In each case, $l$ is the unique bisecant line of $Z$.
Moreover, $\{ [m]\in \Sigma_{Z} (Y) \mid l \cap m \neq \emptyset \} =L$. 
\end{prop}

\begin{proof}
In the proof, we freely use the notation $l_{(x:y:z)}$ in Proposition \ref{prop:SigmaofW5}.
If $\chara k =0$, these statements follow from \cite[Lemma 5.4.1]{Kuznetsov-Prokhorov-Shramov}.

We first show (1).
We may assume that $Z$ equals to the curve (\ref{eqn:ZMU}), which is given by the Zariski closure of 
\[
B \cdot (0:0:0:0:0:1:0) = U \cdot (0:0:0:0:0:1:0), 
\]
where $B \subset \PGL_{2,k}$ is the upper triangular Borel subgroup, and $U$ is the unipotent subgroup of upper triangular unipotent matrices.
By Propositions \ref{prop:SigmaofW5} and \ref{prop:numberoflines},
the lines passing through $(0:0:0:0:0:1:0)$ are written by  (\ref{eqn:line56}) and (\ref{eqn:line15}), which is equal to $l_{(0:1:0)}$ and $l_{(0:0:1)}$ respectively.
Therefore, as a set, we have
\[
q (p^{-1} ( U \cdot (0:0:0:0:0:1:0)) =  \{(-2s:0:1)\} \cup \{(2s^2 : 1: -2s)\},
\]
where $s$ is a $(1,2)$-entry of $U$.
Since $p$ and $q$ are flat, $q (p^{-1} (Z)$ is the closure of $q (p^{-1} ( U \cdot (0:0:0:0:0:1:0))$, which is the union $V (y) \cup V(2xy - z^2)$, and (1) follows.

Note that we have $Z \setminus (U \cdot (0:0:0:0:0:1:0)= (1:0:0:0:0:0:0) \}$, and the line passing through this point is unique, which is $l_{(1:0:0)}$ given by the image of the following map $\P^1 \to Y$  (see Propositions \ref{prop:SigmaofW5} and \ref{prop:numberoflines}):
\begin{equation}
\label{eqn:line01}
        (s:t) \mapsto (s:t:0:0:0:0:0).
\end{equation}

Next, we shall show (2).
We set 
\begin{equation}
\label{eqn:Pcoord}
P = (0:-4:0:0:0:1:0).
\end{equation}
We may assume that $Z$ is given by the Zariski closure of $U\cdot P$.
By Proposition \ref{prop:numberoflines} (more precisely, the proof of \cite[Proposition 6.7]{V5}), the lines passing through $P$ are $l_{(0:0:1)}$ and $l_{(2:\mp1:0)}$.
The latter are given by maps $\P^1 \to Y$ which sends 
    \begin{equation}
    \label{eqn:line15conj}
(s\colon t) \mapsto (8s: -4t: \pm 4s: 0: -2s: t: \mp s).
    \end{equation}
By the same computation using the $U$-action, we have 
\[
q (p^{-1} ( U \cdot P)) =  \{(-2s:0:1)\} \cup \{(2s^2-2:1:-2s)\} \cup \{(2s^2+2: 1: -2s)\}
\]
where $s$ is a (1,2)-entry of $U$.
Therefore, $\Sigma_{Z} (Y) \subset \P^2_{[x:y:z]}$ is given by the union $V(y) \cup V(2(x+2y)y-z^2) \cup V(2(x-2y)y-z^2)$ and (2) follows.
Note that we have $Z \setminus (U \cdot P) = \{ (1:0:0:0:0:0:0)\}$
(cf. \ref{eqn:actPxi}), and we have only one line $l_{(1:0:0)}$ passing through this point (see \eqref{eqn:line01}).

Next, we shall show (3).
We may assume that $Z= Z_u$ ($u \in k \setminus\{0,1,5/4\}$), where $Z_u$ is as in (\ref{eqn:ZGm}).
Take $v \in k$ satisfying $v^{-4} = 5-4u$, and let $P_v \in Z_u(k)$ be the point
\[
P_{v} \coloneqq (8v^5-8v: 20v^4-4:8v^3:4v^2:2v:1:0) = \sigma
\begin{pmatrix}
    1 & v \\
    0& 1
\end{pmatrix}
P,
\]
where $P$ is as in (\ref{eqn:Pcoord}).
Let $T \subset \PGL_{2,k}$ be the diagonal torus.
Then we have
\[
q(p^{-1}(T\cdot P_{v})) = \{ (-2t^2v: 0: t) \} \cup \{(2 t^2(v^2-1 ): 1: -2tv) \} \cup \{(2t^2 (v^2 +1) : 1: -2tv)\}
\]
where $t$ is the coordinate of $T$ obtained by dividing the (1,1)-entry by the (2,2)-entry.
Therefore, 
$q (p^{-1} (Z)) \subset \P^2_{[x:y:z]}$ is the union $V(y) \cup V(2xyv^2 - z^2 (v^2-1) ) 
\cup V(2xyv^2 - z^2 (v^2+1))$, and (3) follows.
Note that we have
\[
Z \setminus (T \cdot P_{v})
= \{ (0:0:0:0:0:1:0), (1:0:0:0:0:0:0) \}.
\]
By the proof of (1), the lines passing through $ (0:0:0:0:0:1:0)$
are $l_{(0:0:1)}$ and $l_{(0:1:0)}$, and the line passing through $(1:0:0:0:0:0:0)$ is $l_{(1:0:0)}$.

The uniqueness and the description of a bisecant line clearly follow from the above proofs.
Note that, if $Z$ is written as (\ref{eqn:ZMU}), (\ref{eqn:ZGageneral}) with $\xi=1$, or (\ref{eqn:ZGm}), then the bisecant line is given by $l_{(1:0:0)}$ ((\ref{eqn:line01})).
Moreover, the line $l_{(0:0:1)}$ ((\ref{eqn:line15})) intersects the line $l_{(1:0:0)}$ ((\ref{eqn:line01})).
On the other hand, lines $l_{(0:1:0)}$ ((\ref{eqn:line56})), $l_{(2:-1:0)}$, and $l_{(2:1:0)}$ ((\ref{eqn:line15conj})) do not intersect $l_{(1:0:0)}$ ((\ref{eqn:line01})).
Therefore, the remaining statements follow.
\end{proof}

\subsection{Characteristic $2$}

\begin{thm}[Classification of quintic curves]
\label{thm:BGaquintic2}
Suppose that $Y=Y_{k} \subset \P^6 = \P^6_k$ is the split $V_5$-variety over a field  of characteristic two, and let the notations be as in 
Proposition \ref{prop:orbitp=2}.
Let $\sigma'$ be the action of $\Aut_{Y/k, \red} = \SL_{2,k}$ on $Y$ given by Proposition \ref{prop:W5automorphisms}.

Let  $Z \subset Y$ be a smooth rational quintic curve,
and $G$ the identity component of the stabilizer of $Z$ with respect to $\sigma'$.
Assume that $\dim G \geq 1$.
Then $G \simeq \Gm$, and $Z$ is mapped to $Z_u$ for some
\[
u =(u_0: u_1) \in \P^1 (k) \setminus \{ (1:0), (1:1), (0:1)\},
\]
by an element of $\SL_{2,k} (k)$ via the action $\sigma'$.
Here, $Z_u$ is a smooth rational quintic curve defined by the same equation as (\ref{eqn:ZGm}).
Moreover, $Z$ is equal to $Z_u$ or $Z_u'$ if and only if $G$ is equal to the diagonal torus of $\SL_{2,k}$.
Here, $Z_u'$ is defined by the same equation as (\ref{eqn:ZGm'}).
In particular, for any $u \neq u'$, $Z_u$ cannot be mapped to $Z_{u'}$ by an element of $\SL_{2,k} (k)$ via the action $\sigma'$.
Conversely, the quintic curve $Z_u$ as above is a smooth rational quintic curve with the stabilizer $\Gm$.
\end{thm}

\begin{proof}
First, we assume that $k$ is an algebraically closed field.
Let $H$ be the reduced identity component of $G$. Then we have one of the following:
\begin{itemize}
    \item[(1)]
    $\dim H \geq 2$.
    \item[(2)]
    $H \simeq \Ga.$
    \item[(3)]
    $H \simeq \Gm$.
\end{itemize}
\noindent{\bf Case (1).}
In this case, we have $Z \subset O_2^{(2)} \sqcup O_1^{(2)} \sqcup O_1^{(2)'} = D_{ \red}$.
Since $Z$ is a smooth rational quintic curve, $Z$ is not contained in $O_1^{(2)} \cup O_{1}^{(2)'}$.
Let $Z'$ be the strict transform of $Z$ via $(\nu')^{-1}$.
By the above fact, $Z'$ is not equal to $E$ nor $\pi^{-1}(\{ z= 0\})$, where $\pi\colon \mathrm{Bl}_{(0:0:1)} \P^2 \rightarrow \P^2$ is the natural projection.
Moreover, since $\nu'$ is an isomorphism over $O_2^{(2)},$ the morphism $Z' \rightarrow Z$ is a proper birational morphism, so it is an isomorphism.
We take a point $P \in \pi (Z') \cap \{z=0\}$.
Since $G \cdot P$ must be $0$-dimensional and $\dim G \geq 2$, $G$ is equal to the Borel subgroup of $\PGL_{2,k}$ that is a stabilizer of $P$.
We may assume that $P = (1:0:0)$, and hence $G$ is the upper triangular Borel subgroup.
In this case, $Z'$ is the strict transform of a line $\{ y =0\} \subset \P^2$.
Then, we have 
\[
Z \cdot \cO (1) = Z'\cdot( 3H-2E) = 1
\]
and it contradicts the fact that $Z$ is a quintic curve.
Therefore, there is no such $Z$.\\
\noindent{\bf Case (2).} 
Since $\G_a \subset \SL_{2,k}$ is conjugate to the unipotent subgroup of upper triangular matrices $U \subset \SL_{2,k}$, we may assume that $\G_a = U.$
Then the restriction $\sigma'|_{U}$ of $\sigma'$ in Proposition \ref{prop:W5automorphisms} sends 
$
\begin{pmatrix}
1&b\\
0&1\\
\end{pmatrix}
$
to the matrix whose entries are of $b$-degree $\leq 3$.
Therefore, the degree of $Z$ is at most $3$, and a contradiction is obtained.
This completes the proof.\\
\noindent{\bf Case (3).} 
Note that the restriction of $\sigma'$ to the torus $T = 
\begin{pmatrix}
    a &0 \\
    0 & a^{-1} \\
\end{pmatrix} 
\hookrightarrow \SL_{2,k}$
is given by 
\begin{equation}
\label{eqn:diagonalactionp=2}
\begin{pmatrix}
    a &0 \\
    0 & a^{-1} \\
\end{pmatrix} 
\mapsto \diag(a^6, a^5, a^4, a^3, a^2, a, 1).
\end{equation}
Therefore, by the same argument as in Case \ref{B-2} of Theorem \ref{thm:BGaquintic>2},
we obtain the desired result.

Next, we treat the case where $k$ is a general field.
By the above argument, we have $G_{\overline{k}}\simeq\Gm$.
By the same argument as in Case \ref{C-2} of Theorem \ref{thm:BGaquintic>2},
we have $G_{\ksep} = \Gm$.
Note that two $G_{\ksep}$-fixed points on $Z_{\ksep}$ are given by
$(Z_{\ksep} \cap O^{(2)'}_{1,\ksep}) (\ksep)$ and
$(Z_{\ksep} \cap O^{(2)}_{1,\ksep} ) (\ksep)$
(to see this, we may assume that $Z_{\ksep}$ is equal to (\ref{eqn:ZGm})).
Therefore, as in Case \ref{C-2} of Theorem \ref{thm:BGaquintic>2}, we have $H =\Gm$.
The remaining holds from the same argument as in Theorem \ref{thm:BGaquintic>2}.
\end{proof}

\begin{rem}
\label{rem:stabilizeractiononZ_p=2}
In the above proof, we can easily show that the restriction morphism $G \rightarrow \Aut_{Z/k}$ is a closed immersion.
\end{rem}

\begin{defn}
Under the assumption in Theorem \ref{thm:BGaquintic2}, we say a smooth rational quintic curve $Z$ is of $\Gm$-type.
\end{defn}

\begin{prop}
\label{prop:SigmaZp=2}
Let the notations be as in
Theorem \ref{thm:BGaquintic2}.
Further assume $k=\overline{k}$.
As in Definition \ref{defn:Sigma}, we denote the universal line on the Hilbert scheme of lines on $Y$ by
$q \colon \mathcal{U}_{Y} \rightarrow \Sigma(Y)$,
and let $p \colon \mathcal{U}_{Y} \rightarrow Y$ be the natural projection.
Let $f\colon \P^2 \simeq \Sigma(Y)$ be as in Proposition \ref{prop:SigmaofW5}.

Then $\Sigma_{Z} (Y) \subset \Sigma (Y) \overset{f^{-1}}\simeq \P^2$ is a union of a line $L$ and two conics.
These two conics are tangent at two points $[l]\in \Sigma (Y)$ and $[l'] \in \Sigma (Y)$ respectively, and the line $L$ is tangent to both conics at the point $[l]$. 

Moreover, $l$ is the unique bisecant line of $Z$.
Furthermore, $\{[m] \in \Sigma_{Z} (Y) \mid m\cap l \neq \emptyset \} = L$. 
\end{prop}

\begin{proof}
In the proof, we identify $\Sigma(Y)$ with $\P^2 = \P^2_{[x:y:z]}$ by $f$.
We may assume that $Z= Z_u$ as in (\ref{eqn:ZGm}).
Write $u = u_0/u_1 \in k$, and take $a \in k$ satisfying $u= a^2+a$.
Then $Z_u$ passes the point $P_u \coloneqq (u-1:u:1:1:1:1:0) $, which is equal to
\[
(a^2+a+1: a^2+a:1:1:1:1:0) 
= \sigma'
\begin{pmatrix}
 a & a+1 \\ 1&1 
\end{pmatrix}
(1:0:0:1:0:0:1)
 \in \P^6.
\]
By Proposition \ref{prop:numberoflines} (more precisely, the proof of \cite[Proposition 6.10]{V5}), the lines passing through $P = (1:0:0:1:0:0:1)$ are $l_{(1:1:1)}$ given by \ref{eqn:line1p2} and $l_{(\omega^{\mp1}:\omega^{\pm1}:1)}$ given by
\begin{equation}
\label{eqn:lineomega}
 \P^1 \ni   (s,t) \mapsto (s,t, t\omega^{\mp1} , s+t\omega^{\pm1} , t, t \omega^{\mp1}, s).
\end{equation}
Then we have
\[
q(p^{-1}(T \cdot P_{u}))
= \{(t:0:1)\}
\cup 
\{t^2(a+\omega):1:t\}
\cup
\{t^2(a+\omega^{-1}):1:t\}.
\]
Therefore, $q(p^{-1}(Z)) \subset \P^2_{[x:y:z]}$ is the union
$V(y) \cup V(z^2 (a+\omega)-x y) \cup V(z^2(a+\omega^{-1}) - x y)$.

Note that we have $Z \setminus T\cdot P_u =\{(0:0:0:0:0:1:0), (1:0:0:0:0:0:0)\}$.
Moreover, by Propositions \ref{prop:SigmaofW5} and \ref{prop:numberoflines},
the lines passing through $(0:0:0:0:0:1:0) \in O_1^{(2)'}$ are $l_{(0:0:1)}$ ((\ref{eqn:line15p=2})) and $l_{(0:1:0)}$ ((\ref{eqn:line56p=2})).
Also, by Proposition \ref{prop:numberoflines}, we only have one line passing through $(1:0:0:0:0:0:0)$.
This is $l_{(1:0:0)}$, which is the image of a map $\P^1 \to Y \subset \P^6$ given by
\begin{equation}
\label{eqn:line01p=2}
        (s:t) \mapsto (s:t:0:0:0:0:0).
\end{equation}

The uniqueness and the description of bisecant lines follow from the above proof. Note that, if $Z$ is equal to $Z_u$ given by (\ref{eqn:ZGm}), then the bisecant line is $l_{(1:0:0)}$ (\ref{eqn:line01p=2}).
Then $l_{(1:1:1)}$ and $l_{(0:0:1)}$ (\ref{eqn:line15p=2}) intersect $l_{(1:0:0)}$ (\ref{eqn:line01p=2}), but either $l_{(\omega^2:\omega:1)}$, $l_{(\omega:\omega^2:1)}$, nor $l_{(0:1:0)}$ (\ref{eqn:line56p=2}) does not intersect (\ref{eqn:line01p=2}).
Therefore, the remaining statement follows.
\end{proof}

\subsection{Degeneration of rational quintic curves with large stabilizers}
\begin{rem}
\label{rem:ZaZmtoZGU}
\begin{enumerate}
\item
The $\Ga$-type smooth quintic curve (\ref{eqn:ZGageneral}) with $\xi =1$ isotrivially degenerates to (\ref{eqn:ZMU}) via the torus action.
Indeed, the family over $\mathbb{A}^1_{a}$ defined by 
\begin{eqnarray*}
&    \P^1 \rightarrow Y \subset \P^6,\\
 &   (s,t) \mapsto (8s^5+2a^4st^4: 20s^4t+a^4t^5: 8s^3t^2: 4s^2t^3: 2st^4:t^5:0),
\end{eqnarray*}
has general fibers ($a \neq 0$) isomorphic to $Z^{(a)}_1$ and the special fiber ($a=0$) isomorphic to (\ref{eqn:ZMU}). 
See Proposition \ref{prop:degenerationMUorGa} for details.
\item 
The $\Gm$-type smooth rational quintic curves (\ref{eqn:ZGm}) degenerate to (\ref{eqn:ZMU}).
Indeed, $Z_{(5:4)}$ is isomorphic to the curve (\ref{eqn:ZMU}), so we have an obvious degeneration with parameters $u$.
On the other hand, we can show that there exist other degenerations from $\Gm$-type to $\Ga$-type or Mukai--Umemura-type. See Proposition \ref{prop:degenerationGm}.
\end{enumerate}
\end{rem}

\begin{prop}
    \label{prop:degenerationMUorGa}
            Let $R$ be a discrete valuation ring, $K$ the fraction field of $R$, and $k$ the residue field of $R$.
    Suppose that $\chara K \neq 2, 5$.
    Let $Y := Y_{R}$ be the split $V_5$-scheme,
    and $\mathcal{Z} \subset Y$ a relative smooth rational quintic curve such that $\mathcal{Z}_K$ is of $\Ga$-type (resp.\ Mukai--Umemura type).
    Then we have $\chara k \neq 2,5$, and 
    $\mathcal{Z}$ is mapped, by an element of $\PGL_{2,R} (R)$ via the action $\sigma$ (see Proposition \ref{prop:W5automorphisms}), to the quintic curve $\mathcal{Z}_{\xi}^{(\a)}$ defined by
\begin{equation}
\label{eqn:ZGageneraldvr}
\begin{gathered}
\P^1_{R} \rightarrow Y \subset \P^6,\\
(s\colon t) \mapsto 
(8s^5+2 \xi st^4: 20s^4t+ \xi t^5: 8s^3t^2: 4s^2t^3: 2st^4:t^5:0)
\end{gathered}
\end{equation}
    for some $\xi \in R \setminus\{0\}$ (resp.\ $\xi=0$).
    Conversely, the quintic curve $\mathcal{Z}_{\xi}^{(\a)}$ is a smooth rational quintic curve if $\chara k \neq 2,5$.
\end{prop}

\begin{proof}
First, we consider the case where $\chara k \neq 2$.
Let $G \subset \PGL_{2,K}$ be the identity component of the stabilizer of $\mathcal{Z}_K$ in $\PGL_{2,K} \simeq \Aut_{Y_K/K, \red}$ with respect to $\sigma$.
By the assumption, we have $G \simeq \Ga$, or $G$ is isomorphic to a Borel subgroup of $\PGL_{2,K}$.
We take a Borel subgroup $B \subset \PGL_{2,K}$ containing $G$.
Let $B_0 \subset \PGL_{2,K}$ be the upper triangular Borel subgroup.
Then there exists an element $g \in \PGL_{2,K} (K)$ such that $g B g^{-1} =B_{0}$.
Recall that the quotient map $\pi\colon\PGL_{2,R} \rightarrow \PGL_{2,R} /B_{0,R}$
(where $B_{0,R}$ is the upper triangular Borel subgroup scheme over $R$)
is a $B_{0,R}$-torsor and the right-hand side is proper over $R$.
Let $\overline{g} \in \PGL_{2,R} /B_{0,R} (R)$ be the image of $g$.
Then $\overline{g}^{\ast} \pi$ is a $B_{0,R}$-torsor over $R$.
We have $H^1_{\et} (\Spec R, \Ga) = H^1_{\et} (\Spec R, \Gm)=0$ since $R$ is a discrete valuation ring.
Therefore, the torsor $\overline{g}^{\ast} \pi$ is split, and we obtain a lift of $\overline{g}$ in $\PGL_{2,R} (R)$.
By replacing $g$ with this lift, we may assume that $g \in \PGL_{2,R} (R)$.
By applying $\sigma(g)$,
we may assume that $B= B_{0}$ (note that we need $\chara k \neq 2$ for the existence of $\sigma$).
In this case, by Theorem \ref{thm:BGaquintic>2}, $\mathcal{Z}_K$ is equal to $Z_{\xi}^{(\mathrm{a})}$ as in (\ref{eqn:ZGageneral}) for some $\xi \in K^{\times}$ or $Z_{\mathrm{MU}}$ (i.e.,\ $Z_{\xi}^{(\mathrm{a})}$ for $\xi =0$) as in (\ref{eqn:ZMU}).
If $\xi \notin R$, then clearly $\mathcal{Z}_k$ contains the line defined by $a_{i} =0$ for $2\leq i \leq 6$ (\ref{eqn:line01}), and it contradicts the fact that $\mathcal{Z}$ is a relative smooth quintic curve.
Therefore, we obtain $\xi \in R$, and hence $\mathcal{Z}$, which is the Zariski closure of $\mathcal{Z}_K$, is given by (\ref{eqn:ZGageneraldvr}) as desired.
By the direct computation, $\mathcal{Z}_{\xi}^{(\mathrm{a})}$ is smooth if and only if $\chara k \neq 2,5$. This completes the proof for $\chara k \neq 2$.

When $\chara k =2$, it suffices to show the non-existence of $\mathcal{Z}$.
Suppose, for contradiction, that there exists $\mathcal{Z}$, and
we may assume that $R$ is complete and there exists an injection $\Z[\sqrt 2] \rightarrow R$.
By Proposition \ref{prop:W5automorphisms}, we obtain the action $\widetilde \sigma$ of $\PGL_{2,R}$ on $Y$.
We modify the above argument as follows:
\begin{itemize}
\item We set $G$ as the reduced identity component of the stabilizer with respect to $\widetilde \sigma_K$ rather than $\sigma$.
\item 
We choose $g\in \PGL_{2,K} (K)$ such that
\[
gBg^{-1} =  \begin{pmatrix}
    1 & 0\\
    0 & \sqrt2
\end{pmatrix}^{-1} B_0
\begin{pmatrix}
    1 & 0\\
    0 & \sqrt2 
\end{pmatrix},
\]
where the right-hand-side is the stabilizer of $Z_{MU}$ with respect to the action $\widetilde \sigma_K$ of $\PGL_{2,K} (K)$.
\end{itemize}
Then, by the same argument, we can show that $\mathcal{Z}$ is, by an element of $\PGL_{2,R} (R)$ via the action $\widetilde \sigma$, mapped to a quintic curve whose generic fiber is equal to $Z_{\xi}^{\mathrm{(a)}}$ for $\xi \in K.$
Moreover, we obtain $\xi \in R$ in the same way, and we obtain a contradiction since $\mathcal{Z}_{\xi}^{(\mathrm{a})}$ is non-smooth in characteristic two.
\end{proof}

\begin{prop}
    \label{prop:degenerationGm}
        Let $R$ be a discrete valuation ring, $K$ the fraction field of $R$, $k$ the residue field of $R$, $\varpi$ a uniformizer of $R$, and $\Ord$ the normalized additive valuation on $R$.
    Suppose that $\chara k \neq 2$.
    Let $Y := Y_{R}$ be the split $V_5$-scheme, and $\mathcal{Z} \subset Y$ a relative smooth rational quintic curve such that $\mathcal{Z}_K$ is of $\Gm$-type.
   We take $u\in K\setminus \{ 0,1, \frac{5}{4}\}$ such that $\mathcal{Z}_K$ is mapped, by an element of $\PGL_{2,K} (K)$ via the action $\sigma$ (cf.\ Proposition \ref{prop:W5automorphisms}), to $Z_u$, where $Z_u \subset Y_K$ is the smooth rational quintic curve defined by (\ref{eqn:ZGm}) (by Theorem \ref{thm:BGaquintic>2}, we can take such a $u$).
    Then we have one of the following:
    \begin{enumerate}
        \item 
        The relative smooth rational quintic curve $\mathcal{Z}$ is mapped, by an element of $\PGL_{2,R} (R)$ via the action $\sigma$, to the quintic curve $\mathcal{Z}_u$ defined by
        \begin{eqnarray}
        \label{eqn:ZuGmdvr}
        \begin{gathered}
        \P^1_R \rightarrow Y \subset \P^6_R \\
   (s,t) \mapsto   (s^5 (u-1): s^4tu: s^3t^2: s^2 t^3: st^4: t^5:0 ).
        \end{gathered}
        \end{eqnarray}
        In this case, we have $u \in R$ and $\overline{u} \in k$ satisfies $\overline{u} \neq 0$, $1$.
        Conversely, $\mathcal{Z}_u$ is a relative smooth rational quintic curve if $u\in R$ and $\overline{u} \neq 0$, $1$.
        \item
        The relative smooth rational quintic curve $\mathcal{Z}$ is mapped, by an element of $\PGL_{2,R} (R)$ via the action $\sigma$, to the quintic curve $\mathcal{Z}_{u,b}$, where $\mathcal{Z}_{u,b}$ is the Zariski closure of
        \[
        \sigma \begin{pmatrix}
            1 & b \\
            0 & 1
        \end{pmatrix}
        (Z_u) \subset Y_K \subset \P^6_K
        \]
        in $Y$.
        Here, $\sigma$ is the action in Proposition \ref{prop:W5automorphisms} and $b$ is some element of $K \setminus R$.
        In this case, we have $ \chara k \neq 5$ and
        \begin{equation}
        \label{eqn:cord}
        c:=\Ord (u- \frac{5}{4} ) \geq - 4 \Ord (b) \geq 4.
        \end{equation}
        Moreover, $\mathcal{Z}_{u,b,k}$ is equal to 
 \begin{equation}
 \label{eqn:reductiontwist}
 \begin{gathered}
\P^1_k \rightarrow Y_k \subset \P^6_k,\\
(s\colon t) \mapsto 
(8s^5+2 \xi st^4: 20s^4t+ \xi t^5: 8s^3t^2: 4s^2t^3: 2st^4:t^5:0),
\end{gathered}
\end{equation}
  where 
  \[
  \xi = \begin{cases}
      \overline{16 b^{4} (u-\frac{5}{4})} & \textup{ if $c = -4 \Ord (b)$, } \\
      0 & \textup{ if $c > -4 \Ord (b)$. }
  \end{cases}
  \]
  Conversely, a relative curve $\mathcal{Z}_{u,b}$ for $b \in K\setminus R$ is a relative smooth rational quintic curve if $\chara k\neq 5$ and (\ref{eqn:cord}) is satisfied.
Note that $\mathcal{Z}_{u,b,k}$ is a $\Ga$-type smooth rational quintic curve (\ref{eqn:ZGageneral}) when $c = -4 \Ord(b)$,
and a Mukai--Umemura type one (\ref{eqn:ZMU}) when $c > -4 \Ord(b)$.
    \end{enumerate}
\end{prop}

\begin{proof}
Let $G \subset \PGL_{2,K}$ be the stabilizer of $\mathcal{Z}_K$ with respect to $\sigma$.
By Theorem \ref{thm:BGaquintic>2}, we have $G \simeq \Gm$.
We take a Borel subgroup $B \subset \PGL_{2,K}$ containing $G$.
Then by the proof of Proposition \ref{prop:degenerationMUorGa}, 
we can take an element $g \in \PGL_{2,R} (R)$ such that 
$
g B g^{-1}
$
is the upper triangular Borel subgroup $B_0 \subset \PGL_{2,K}$.
Therefore, we may assume $B = B_0$.
Then by Theorem \ref{thm:BGaquintic>2},
there exists $b\in K$ such that
we have
        \begin{equation}
        \label{eqn:bZu}
      \mathcal{Z}_K =       \sigma \begin{pmatrix}
            1 & b \\
            0 & 1
        \end{pmatrix}
        (Z_u) \subset Y_K \subset \P^6_K,
        \end{equation}
        or 
                \begin{equation}
                \label{eqn:bZ'u}
                \mathcal{Z}_K =
        \sigma \begin{pmatrix}
            1 & b \\
            0 & 1
        \end{pmatrix}
        (Z'_u) \subset Y_K \subset \P^6_K.
        \end{equation}
       
        First, we assume that the former (\ref{eqn:bZu}) holds.
        If $b\in R$, then we may assume that $b=0$, and $\mathcal{Z}_K$ is as in (1).
        In this case, if $u\notin R$, $\mathcal{Z}_k$ contains a line defined by $x_i=0$ for $2 \leq i\leq 6$ (\ref{eqn:line01}), and it contradicts the fact that $\mathcal{Z}_k$ is a smooth quintic curve. Therefore, we have $u\in R$ and the statement (1) follows.
In the following, we assume that $b \in K \setminus R$.

Set $P \coloneqq(u-1,u,1,1,1,1,0)$,  $\xi \coloneqq (2b)^4(u-5/4)$, and
 $P_\xi=(0,\xi,0,0,0,1,0)$.
Then $\mathcal{Z}_K$ is the closure of the following orbit of  $T = \left\{
\begin{pmatrix}
a & 0 \\ 0 & 1
\end{pmatrix}
\right\}$:
\[
\begin{pmatrix}
1 & b \\
0 & 1
\end{pmatrix}
T \cdot P
=
\begin{pmatrix}
a & b \\
0 & 1
\end{pmatrix}
P
=
\begin{pmatrix}
a & b \\
0 & 1
\end{pmatrix}
\begin{pmatrix}
1 & 1/2 \\
0 & 1
\end{pmatrix}
\begin{pmatrix}
-1/2b & 0 \\
0 & 1
\end{pmatrix}
P_{\xi}
=
\begin{pmatrix}
-a/2b & (a+2b)/2 \\
0 & 1
\end{pmatrix}
P_{\xi}.
\]
The closure of the limit of 
$\left\{\begin{pmatrix}
-a/2b & (a+2b)/2 \\
0 & 1
\end{pmatrix}\mid a\in K^\times \right\}$
is
$H = \left\{\begin{pmatrix}
1 & B \\
0 & 1
\end{pmatrix} \mid B \in k \right\}$.
Also, the limit of $P_{\xi} = (0,\xi,0,0,0,1,0)$ is
\[
\begin{cases}
P_{\overline \xi} & \xi \in R\setminus (\varpi),\\
(0,0,0,0,0,1,0) & \xi \in (\varpi),\\
(0,1,0,0,0,0,0) & \xi \not \in R.
\end{cases}
\]
In the last case, the limit $\mathcal{Z}_k$ contains a line, which contradicts to the assumption on $\mathcal{Z}$.
Thus the limit $\mathcal{Z}_k$ is $Z_{\overline{\xi}}^{(a)}$, or $Z_{\mathrm{MU}}$ respectively.

Next, we assume that the latter \ref{eqn:bZ'u} holds.
The case where $b \in R$ can be treated by the same way as in the argument for (\ref{eqn:bZu}).
We consider the case where $b \in K \setminus R$.
Note that we have
\[
        \sigma \begin{pmatrix}
            1 & b \\
            0 & 1
        \end{pmatrix}
        (Z'_u) = 
                \sigma 
                \begin{pmatrix}
            1 & b \\
            0 & 1
        \end{pmatrix}
        \sigma \begin{pmatrix}
            0 & 1\\
            1 &0
        \end{pmatrix}
        (Z_u).
\]
Let 
$g:= \begin{pmatrix}
    1 & 1\\
    b^{-1} & -1
\end{pmatrix}
\in \PGL_2 (R)
$.
Since we have
\[
\begin{pmatrix}
    1 & -b \\
    0 & 1
\end{pmatrix}
g 
\begin{pmatrix}
    1 & b \\
    0 & 1
\end{pmatrix}
\begin{pmatrix}
    0 & 1 \\
    1 &0
\end{pmatrix}
= \begin{pmatrix}
    1+b & 0 \\
    0 & b^{-1}
\end{pmatrix}
\in T(K),
\]
we obtain
\[
             \sigma (g)   \sigma 
                \begin{pmatrix}
            1 & b \\
            0 & 1
        \end{pmatrix}
        \sigma \begin{pmatrix}
            0 & 1\\
            1 &0
        \end{pmatrix}
        (Z_u)
        =
        \sigma \begin{pmatrix}
            1 & b \\
            0 & 1
        \end{pmatrix} (Z_u).
\]
Therefore, we can reduce to the case where (\ref{eqn:bZu}) holds.
This completes the proof.
\end{proof}

\begin{rem}
\label{rem:lineGmreduction}
We work in the setting of Proposition \ref{prop:degenerationGm},
and assume further that case (2) of Proposition \ref{prop:degenerationGm} holds.
By the proof of Proposition \ref{prop:SigmaZp>2}.(3), the two singular points of $\Sigma_{\mathcal{Z}_K} (Y)$ correspond to the following lines on $Y \subset \P^6_K$: 
\[
L_1' :=
\sigma \begin{pmatrix}
    1 & b \\
    0 & 1
\end{pmatrix} (l_{(1:0:0)}) , \quad L_2' :=
\sigma \begin{pmatrix}
    1 & b \\
    0 & 1
\end{pmatrix} (l_{(0:1:0)}).
\]
By easy calculation, the flat limit $L'_{1,k}$ of $L_1'$ is the line defined by the same equation as in $l_{(1:0:0)}$.
On the other hand, the flat limit $L'_{2,k}$ of $L'_{2}$ contains a point 
\[
\sigma \begin{pmatrix}
    1 & -a_0 \\
    0 & 1\\
\end{pmatrix} (0:1:0:0:0:0:0)
= (-2a_0:1:0:0:0:0:0).
\]
Indeed, $L_2'$ contains a point
\[
\sigma \begin{pmatrix}
    1 & b \\
    0 & 1
\end{pmatrix}
(0:0:0:0:0:1: - \frac{1}{b}) 
= (0: 4b^5:4b^4:4b^4:4b^2:5b:1)
\]
whose reduction is $(0:1:0:0:0:0:0)$,
and $L_2'$ is stable under the action of 
\[
\begin{pmatrix}
    1+ \varpi^{-\Ord (b)} a_0& -a_0 b \varpi^{- \Ord (b)} \\
    0 & 1
\end{pmatrix}
\]
for any $a_0 \in R^{\times}$
as in the proof of Proposition \ref{prop:degenerationGm}.
Therefore, we have $L_{1,k'} = L_{2,k'}$.
This is in contrast to the degeneration in Proposition \ref{prop:degenerationGm}(1), where two singular points of $\Sigma_{\mathcal{Z}_K}(Y_K)$ degenerate to two disjoint points by the proof of Proposition \ref{prop:SigmaZp>2}.
\end{rem}

\begin{prop}
    \label{prop:degenerationGmp=2}
    Let $R$ be a discrete valuation ring, $K$ the fraction field of $R$ and $k$ the residue field of $R$.
    Suppose that $\chara k =2$.
    Let $Y := Y_{R}$ be the split $V_5$-scheme over $R$, and $\mathcal{Z} \subset Y$ be a relative smooth rational quintic curve such that $\mathcal{Z}_K$ is of $\Gm$-type.
By Theorems \ref{thm:BGaquintic>2} and \ref{thm:BGaquintic2}, $\mathcal{Z}_K$ is mapped to $Z_u$ ($u \in K \setminus \{0,1,\frac{5}{4}\}$) by an element of $\Aut_{Y,K}(K)$, where $Z_u \subset Y_K$ is the smooth rational quintic curve defined by (\ref{eqn:ZGm}).

Suppose further that, if $\chara K =0$, then $R$ contains a square root $\sqrt 2$ of 2, and fix a $\Z[\sqrt 2]$-algebra structure on $R$.

Then we have $u\in R$ and $\overline{u} \in k$ satisfies $\overline{u} \neq 0,1$.
    Moreover, $\mathcal{Z}$ is mapped, by an element of $\Aut_{Y,R} (R)$, to the quintic curve $\mathcal{Z}_{u}$ defined by
      \begin{eqnarray*}
        \begin{gathered}
        \P^1_R \rightarrow Y \subset \P^6_R \\
   (s,t) \mapsto   (s^5 (u-1): s^4tu: s^3t^2: s^2 t^3: st^4: t^5:0 ).
        \end{gathered}
        \end{eqnarray*}
        Conversely, $\mathcal{Z}_u$ is a relative smooth rational quintic curve if $u \in R$ and $\overline{u} \neq 0,1$.
\end{prop}

\begin{proof}
The argument is similar to that of Proposition \ref{prop:degenerationGm}, but we include it.

We only show (2) since the proof of (1) is similar.
As in the proof of Proposition \ref{prop:degenerationGm}, we may assume that there exists $\widetilde{b}\in K$ such that
we have
        \begin{equation}
      \mathcal{Z}_K =       \widetilde{\sigma} 
      \begin{pmatrix}
            1 & \widetilde{b} \\
            0 & 1
        \end{pmatrix}
        (Z_u)
        = \sigma 
        \begin{pmatrix}
            1 & b\\
            0 &1
        \end{pmatrix}  (Z_u)
        \subset Y_K \subset \P^6_K
        \end{equation}
        or 
                \begin{equation}
                \mathcal{Z}_K =
        \widetilde{\sigma} \begin{pmatrix}
            1 & \widetilde{b} \\
            0 & 1
        \end{pmatrix}
        (Z'_u) \subset Y_K \subset \P^6_K,
        \end{equation}
        where we set $b := \widetilde{b}/\sqrt 2$.
If we have $\widetilde{b} \in R$, by the argument in Proposition \ref{prop:degenerationGm}, we obtain $u\in R$ and the statement holds true.
Therefore, it suffices to show that $\widetilde{b} \in R$.
Suppose, for contradiction, that $\widetilde{b} \in K \setminus R$.
In this case, 
$\mathcal{Z}_K$ is stable under the action of
\[
\begin{pmatrix}
1 & \widetilde{b} \\
0 & 1
\end{pmatrix}
T
\begin{pmatrix}
    1 & -\widetilde{b}\\
    0 & 1
\end{pmatrix},
\]
whose $\overline{K}$-valued points are given by
\[
\left\{\begin{pmatrix}
    \widetilde{a} & \widetilde{b}(1-\widetilde{a})\\
    0& 1
\end{pmatrix} \in \PGL_{2,K} (\overline{K}) \mid \widetilde{a} \in \overline{K}^{^\times}
\right\},
\]
via $\widetilde{\sigma}$.
Here, $T \subset \PGL_{2,K}$ is the diagonal torus.
Therefore,
$\mathcal{Z}_k$ is 
$\widetilde{\sigma}_k(
U_k    
)$-stable, where $U_k \subset \PGL_{2,k}$ is the unipotent subgroup of the upper triangular matrices.
By the definition of $\sigma'$ (see (\ref{eqn:sigma'})),
$\mathcal{Z}_k$ is $\sigma ' (U'_k)$-stable, 
where 
$
U_k' \subset \SL_{2,k}
$
is the unipotent subgroup of the upper triangular matrices.
This contradicts Theorem \ref{thm:BGaquintic2}.
This completes the proof.
\end{proof}

\begin{rem}
\label{rem:lineGmreductionp=2}
    In the setting of Proposition \ref{prop:degenerationGmp=2}, the two singular points of $\Sigma_{\mathcal{Z}_K} (Y_K)$ degenerate to two disjoint points of $\Sigma_{\mathcal{Z}_k} (Y_k)$.
\end{rem}

\section{Two-ray games for $V_{22}$-varieties}
\label{section:two-ray}

\subsection{Generalities on $V_{22}$}
Let $k$ be a field, and $X$ a $V_{22}$-variety over $k$.
Then the anti-canonical divisor $- K_{X}$ is very ample and defines an embedding
$X \hookrightarrow \P^{13}_k$ by
\cite[Theorem 1.1]{Tanaka1}.

We name certain $V_{22}$-varieties according to the structure of their automorphism groups.

\begin{defn}
\label{defn:V22name}
Let $k$ be a field, and $X$ a $V_{22}$-variety over $k$.
\begin{enumerate}
    \item 
    We say $X$ is a \emph{Mukai--Umemura variety (or $V_{22}$-variety of $\PGL_2$-type)} when \\
    $\Aut_{X_{\overline{k}}/\overline{k}, \red}^{\circ} \simeq \PGL_{2, \overline{k}}$.
    \item 
    We say $X$ is a \emph{$V_{22}$-variety of $\Ga$-type} when $\Aut_{X_{\overline{k}}/\overline{k}, \red}^{\circ} \simeq \Ga$.
    \item 
    We say $X$ is a \emph{$V_{22}$-variety of $\Gm$-type} when $\Aut_{X_{\overline{k}}/\overline{k}, \red}^{\circ} \simeq \Gm$. 
\end{enumerate}
\end{defn}

\begin{rem}
\label{rem:autdimgeq1}
\begin{enumerate}
\item
If $k$ is of characteristic zero, this definition of a Mukai--Umemura variety is the same as the classical definition (see \cite[Theorem 5.3.10]{Kuznetsov-Prokhorov-Shramov} for example).
\item 
Let $k$ and $X$ be as in Definition \ref{defn:V22name}.
If $\Aut(X_{\overline{k}})$ is an infinite group, then $X$ satisfies exactly one of the three conditions in Definition \ref{defn:V22name} by Theorem \ref{thm:MukaiUmemuraclassification} and Corollary \ref{cor:noV22Autdim2}.
\end{enumerate}
\end{rem}

\begin{defn}
    Let $X$ be a $V_{22}$-scheme over a scheme $B$.
    We say that $X$ is a \emph{Mukai--Umemura scheme} over $B$ 
if any geometric fiber $X_s$ is a Mukai--Umemura variety over $k(s)$.
\emph{$V_{22}$-schemes of $\Ga$-type} over $B$, and  \emph{$V_{22}$-schemes of $\Gm$-type} over $B$ are defined similarly.
\end{defn}

\subsection{Two-ray game for conics}
Here we describe the two-ray game for conics on $V_{22}$.
Note that, on $V_{22}$, there exists a conic:
\begin{prop}[{\cite[Corollary 10.6]{Prokhorovconicnote}, \cite[Theorem 6.7]{Tanaka2}}]
\label{prop:conicexists}
Let $k$ be an algebraically closed field, and $X$ a $V_{22}$-variety over $k$. 
Then there exists a smooth conic on $X$.
\end{prop}

Via the two-ray game for conics, $V_{22}$ varieties in characteristic zero are related to the quadric threefold $Q$ and rational sextic curves in $Q$.
Below, we generalize this two-ray game to the case of positive/mixed characteristic.

\begin{defn}[{\cite{Kuznetsov-ProkhorovGm}}]
\label{defn:quadraticallynormal}
Let $k$ be a field, and $\Gamma \subset \P^4$ a smooth sextic curve of genus $0$ over $k$.
We say $\Gamma$ is \emph{quadratically normal} if restrictions of quadrics form a complete linear system on $\Gamma$, i.e.,\ the map
$H^0 (\P^4, \cO(2)) \rightarrow H^{0} (\Gamma, \cO (2)|_{\Gamma})$
is surjective.
\end{defn}

\begin{defn}
    \label{defn:quadraticthreefold}
  Let $B$ be a scheme.
  A (relative) \emph{quadric threefold} over $B$ is a smooth hypersurface of degree $2$ in some rank 4 projective bundle $\P_{B} (E).$
\end{defn}

\begin{lem}
\label{lem:quadraticallynormal}
Let $k$ be a field, 
and $\Gamma \subset \P^4_k$ a smooth sextic curve of genus $0$ over $k$.
The following are equivalent.
\begin{enumerate}
    \item 
    $\Gamma$ is quadratically normal. 
    \item 
    $\dim H^0 (\mathcal{I}_{\Gamma} (2)) \leq 2$.
    Here, $\mathcal{I}_{\Gamma}$ is the ideal sheaf of $\Gamma \subset \P^4_k$.
\end{enumerate}
Note that (2) implies that $\Gamma$ is linearly normal, i.e.,\ $\Gamma$ is not contained in any hyperplane in $\P^4_k$.
Moreover, if $\Gamma$ is contained in a quadric threefold $Q \subset \P^4_k$, (1) and (2) are equivalent to the following.
\begin{enumerate}
\setcounter{enumi}{2}
    \item 
    $\dim H^0 (\mathcal{I}_{\Gamma,Q} (2)) \leq 1$. Here, $\mathcal{I}_{\Gamma, Q}$ is the ideal sheaf of $\Gamma \subset Q$, and  $\cO_{Q}(1)$ is the hyperplane line bundle on $Q$.
    \end{enumerate}
In this case, we have $\dim H^0 (\mathcal{I}_{\Gamma} (2)) = 2$ and $\dim H^0 (\mathcal{I}_{\Gamma,Q} (2)) = 1$.
\end{lem}

\begin{proof}
These follow from the same argument as in \cite[Remark 2.1]{Kuznetsov-ProkhorovGm}.
Note that we have $\dim H^0(\P^4,\cO(2)) =15$ and $\dim H^0(\P^1,\cO(12)) =13$, so that $\dim H^0 (\mathcal{I}_{\Gamma} (2)) \geq 2$.
\end{proof}

\begin{lem}
\label{lem:sexticcubic}
Let $k$ be an algebraically closed field, and $\Gamma \subset \P^4_k$ a quadratically normal smooth rational sextic curve over $k$.
Then $\Gamma$ is a scheme theoretic intersection of cubic hypersurfaces.
\end{lem}

\begin{proof}
This follows from the same proof as in \cite[Lemma 2.7]{Kuznetsov-ProkhorovGm} (where they use \cite{Gruson-Lazarsfeld-Peskine}).
Note that \cite{Gruson-Lazarsfeld-Peskine} is characteristic-free.
\end{proof}

\begin{lem}
\label{lem:liftquadricsextic}
Let $k$ be a perfect field, $Q \subset \P^4_k$ a quadric threefold over $k$, and $\Gamma \subset Q \subset \P^4_k$ a quadratically normal smooth rational sextic curve.
Then there exists a relative quadric threefold $\mathcal{Q} \subset \P^4_{W (k)}$ over $W(k)$ and a relative smooth rational sextic curve $\widetilde{\Gamma} \subset \mathcal{Q} \subset \P^4_{W(k)}$ that lift $(Q,\Gamma)$.
Moreover, $\widetilde{\Gamma}_{K}$ is quadratically normal. Here, $K$ is the fraction field of $W(k)$.
\end{lem}

\begin{proof}
Note that $\Gamma$ is the image of some map $\gamma \colon\P^1_k \rightarrow \P^4_k$.
By choosing the homogeneous coordinate representation of $\gamma$ and taking the lift of each component, we obtain a morphism $\widetilde{\gamma} \colon \P^1_{W(k)} \rightarrow \P^4_{W(k)}$ that lifts $\gamma$.
Let $\widetilde{\Gamma} \subset \P^4_{W(k)}$ be the scheme theoretic image of $\widetilde{\gamma}$, which is a flat lift of $\Gamma \subset \P^4$.
Let $\mathcal{I}_{\widetilde{\Gamma}}$ be the ideal sheaf of $\widetilde{\Gamma} \hookrightarrow \P^4_{W(k)}$, and $\cO_{\P^4_{W(k)}} (1)$ the hyperplane line bundle on $\P^4_{W(k)}$. 
Since $\mathcal{I}_{\widetilde{\Gamma}} (2)$ on $\P^4_{W(k)}$ is flat over $W(k)$, we obtain
\[
\dim H^0 (\mathcal{I}_{\widetilde{\Gamma}_K} (2))
\leq \dim H^0 (\mathcal{I}_{\Gamma} (2)) \leq 2. 
\]
where the last equality follows from Lemma \ref{lem:quadraticallynormal}, and the assertion follows.
Here, $\mathcal{I}_{\widetilde{\Gamma}_K}$ and $\mathcal{I}_{\Gamma}$ are the ideal sheaves of $\widetilde{\Gamma}_{K} \subset \P^4_{K}$ and $\Gamma\subset \P^4_k$, respectively.

Therefore, $\widetilde{\Gamma}_K \subset \P^4_K$ is a quadratically normal smooth rational sextic curve, and
$H^0 (\mathcal{I}_{\widetilde{\Gamma}})$ is a rank 2 free module over $W(k)$.
The quadric $Q$ corresponds to a section 
\[
s \in H^0(\mathcal{I}_{\Gamma} (2)) = H^0 (\mathcal{I}_{\widetilde{\Gamma}} (2)) \otimes_{W(k)} k.
\]
We can take a lift $\widetilde{s} \in H^0 (I_{\widetilde{\Gamma}} (2))$ of $s$.
Then the corresponding quadric hypersurface $\mathcal{Q} \subset \P^4_{W(k)}$
is a flat lift of $Q$, and contains $\widetilde{\Gamma}$.
This completes the proof.
\end{proof}

\begin{prop}
\label{prop:two-rayV22toQ3}
Let $k$ be a field and $X$ a $V_{22}$-variety over $k$.
Let $C$ be a smooth conic on $X$ and $\mu \colon \widetilde{X} \rightarrow X$ a blow-up of $C$ on $X$.
Then the following hold.
\begin{enumerate}
\item 
The anti-canonical divisor $- K_{\widetilde{X}}$ is base-point free.
We denote the corresponding fiber-space morphism by $\Psi \colon \widetilde{X} \rightarrow V$. 
\item 
The morphism $\Psi$ is a flopping contraction, and the flop of $\Psi$ exists.
We denote the flop by $\Psi^{+} \colon \widetilde{Q} \rightarrow V$.
\item 
Let $\tau \colon \widetilde{Q} \rightarrow Q$ be the contraction of the $K_{\widetilde{Q}}$-negative extremal ray.
Then $\tau$ is the blow-up of a smooth curve $\Gamma$ on $\mathcal{Q}$.
Then $Q_{\overline{k}}$ is a quadric threefold in $\P^4_{\overline{k}}$.
Moreover, $\Gamma_{\overline{k}} \subset Q_{\overline{k}}$ is a quadratically normal smooth rational sextic curve.
\item 
The hyperplane bundle $Q_{\overline{k}} \subset \P^4_{\overline{k}}$ descends to $k$, i.e.,\ $Q$ is a quadric threefold in $\P^4$.
In particular, $\Gamma \subset Q \subset \P^4$ is a quadratically normal smooth sextic curve of genus $0$.
\end{enumerate}
\end{prop}

\begin{proof}
First, we show (1)-(3).
To show (1)-(3), we may assume that $k$ is an algebraically closed field.
If $k$ is of characteristic zero, these are proved in \cite[Theorem 2.2]{Kuznetsov-ProkhorovGm}.
Basically, assuming H. Tanaka's results, the proof proceeds in the same way as in the characteristic zero case, but we include the proof here for completeness.
We assume that $\chara k >0$.
(1) follows from \cite[Proposition 7.2]{Tanaka2}.
(2) follows from \cite[Proposition 7.2, Proposition 7.3 and Corollary 7.9]{Tanaka2} and \cite[Theorem 1.1]{Tanakaflop}.
We shall show (3).
By \cite[Theorem 7.4]{Tanaka2}, $Q_{\overline{k}}$ is a quadric threefold in $\P^4_{\overline{k}}$, and $\Gamma_{\overline{k}} \subset Q_{\overline{k}}$ is a smooth rational sextic curve.
Therefore, it suffices to show that $\Gamma_{\overline{k}}$ is quadratically normal.
Let $H$ be a hyperplane section of $Q_{\overline{k}} \subset \P^4_{\overline{k}}$, and $E$ the exceptional divisor of $\tau$. 
We denote the $\mu$-exceptional divisor by $D$, and its strict transform on $\widetilde{Q}$ by $D'$.
By \cite[Table 3]{Tanaka2}, we obtain $E \sim -2 K_{\widetilde{Q}} - 3 D'$.
Since 
$\tau^{\ast} (3H) \sim -K_{\widetilde{Q}} + E \sim -3K_{\widetilde{Q}}-3D'$,
We have $\tau^{\ast} (H) \sim -K_{\widetilde{Q}}-D'$. (Note that $\Pic (\widetilde{Q})$ is torsion-free.
Note also that the equivalence holds true even if $\chara k =0$ (see \cite[Section 2]{Kuznetsov-ProkhorovGm}.)
Hence, we obtain
\[
2 \tau^{\ast} (H) - E \sim 2(-K_{\widetilde{Q}}-D') -(-2K_{\widetilde{Q}}-3D') = D'.
\]
Therefore, we obtain 
  $h^0(2H-\Gamma)= h^0 (2 \tau^{\ast} (H)- E)  = h^0 (D) =1$,
which means $\Gamma$ is quadratically normal by Lemma \ref{lem:quadraticallynormal}.

Finally, we show (4). 
Since $\tau^{\ast} (H) \sim -K_{\widetilde{Q}}-D'$, we have
$\tau_{\ast} (-K_{\widetilde{Q}} -D')_{\overline{k}} \sim H_{\overline{k}}$.
Therefore, $\cO (\tau_{\ast} (-K_{\widetilde{Q}} -D'))$ gives an $k$-form of $\cO(H_{\overline{k}})$, which is a very ample line bundle and defines an embedding $Q \hookrightarrow \P^4$.
This completes the proof.
\end{proof}

\begin{prop}
\label{prop:two-rayQ3toV22}
Let $k$ be a field, and $Q \subset \P^4$ a quadric 3-fold over $k$.
Let $\Gamma \subset Q \subset \P^4$  be a quadratically normal smooth sextic curve of genus $0$ over $k$, 
and $\tau \colon \widetilde{Q} \rightarrow Q$ a blow-up of $\Gamma$ on $Q$.
Then the following hold.
\begin{enumerate}
\item 
The anti-canonical divisor $- K_{\widetilde{Q}}$ is base-point free.
We denote the corresponding fiber-space morphism by $\Phi \colon \widetilde{Q} \rightarrow V$. 
\item 
The morphism $\Phi$ is a flopping contraction, and the flop of $\Phi$ exists.
We denote the flop by $\Phi^{+} \colon \widetilde{X} \rightarrow V$.
\item 
Let $\mu \colon \widetilde{X} \rightarrow X$ be the contraction of the $K_{\widetilde{X}}$-negative extremal ray.
Then $X$ is a $V_{22}$-variety over $k$, and $\tau$ is a blow-up of a smooth conic $C$ on $X$.
\item 
This construction gives the converse of the construction in Proposition \ref{prop:two-rayV22toQ3}.
\end{enumerate}
\end{prop}

\begin{proof}
All the assertions can be reduced to the case where $k$ is algebraically closed.
If $k$ is of characteristic zero, these are proved in \cite[Theorem 2.6]{Kuznetsov-ProkhorovGm}.
Basically, a similar proof works, but we also use a new input: the numerical classification of Fano 3-folds in positive characteristic (see \cite{Tanaka1} and \cite{Tanaka2}). Therefore, we provide a complete proof.

We assume that $\chara k>0$.
Let $H$ be a hyperplane section of $Q \subset \P^4$, and $E$ the exceptional divisor of $\tau$.
By Lemma \ref{lem:sexticcubic}, the linear system $|-K_{\widetilde{Q}}| = |3\tau^{\ast}(H)-E|$ on $\widetilde{Q}$ is base-point free, i.e.,\ (1) follows.

Since $(-K_{\widetilde{Q}})^3 = 16$ by \cite[Lemma 3.21]{Tanaka2}, $-K_{\widetilde{Q}}$ is big.
Let $(\mathcal{Q}, \widetilde{\Gamma})$ be a $W(k)$-lift of $(Q, \Gamma)$ given by Lemma \ref{lem:liftquadricsextic}.
Let $\widetilde{\tau} \colon \widetilde{\mathcal{Q}} \rightarrow \mathcal{Q}$ be the blow-up of $\widetilde{\Gamma}$ on $\mathcal{Q}$.
Then $-K_{\widetilde{\mathcal{Q}}_K}$ contracts a curve by \cite[Theorem 2.6]{Kuznetsov-ProkhorovGm}.
Therefore, $-K_{\widetilde{Q}}$ also contracts a curve, i.e.,\ $-K_{\widetilde{Q}}$ is not ample.

Assume that $\Phi$ is a divisorial contraction.
For any $\Phi$-contracted curve $f$, we have $-K_{\widetilde{Q}} \cdot f =0$.
Since $\tau^{\ast}(H)-K_{\widetilde{Q}}$ is ample, we have $\tau^{\ast}(H) \cdot f>0$ and $E \cdot f>0$.
By Lemma \ref{lem:quadraticallynormal}, $|2 \tau^{\ast}(H)-E|$ consists of a 1 point.
Let $m \geq 1$ be the maximal integer such that $|2 \tau^{\ast}(H)-mE|$ is a 1 point, and $F_0$ be the unique member.
Since 
\[
(2 \tau^{\ast}(H) - m E) \cdot f < (3 \tau^{\ast} (H)-E)\cdot f = -K_{\widetilde{Q}} \cdot f =0,
\]
we obtain $f \subset F_0$.
By the choice of $m$, $F_0$ is the strict transform of the quadric $Q'|_{Q}$ on $Q$ passing through $\Gamma$.
Since $\Gamma$ is not contained in a hyperplane, $Q'|_{Q}$ is irreducible, and so is $F_0$.
Therefore, $\Phi$ contracts $F_0$.
On the other hand, we have 
\[
(-K_{\widetilde{Q}})^3 = 16, (-K_{\widetilde{Q}})^2\cdot E = 20, (-K_{\widetilde{Q}})\cdot E^2 = -2, \textup{ and } E^3 = -16
\]
by \cite[Lemma 3.21]{Tanaka2}.
Therefore, we have
\[
(-K_{\widetilde{Q}})^2 \cdot (2 \tau^{\ast} (H)-mE)
= 
(-K_{\widetilde{Q}})^2 \cdot (\frac{2}{3}(-K_{\widetilde{Q}}+E) -mE)
=24-20m,
\]
which contradicts $(-K_{\widetilde{Q}})^2\cdot F_0 =0$.
As a result, $\Phi$ is a flopping contraction, and we obtain (2) by the existence of flop (\cite[Theorem 1.1]{Tanakaflop}).

Next, we shall prove (3).
As in (3), let $\mu \colon \widetilde{X} \rightarrow X$ be a $K_{\widetilde{X}}$-negative extremal contraction (,which is unique since $\rho (\widetilde{X}) =2$ and $-K_{\widetilde{X}}$ is not ample but nef and big).
Since $\rho (\widetilde{X}) =2$, we have $\dim X \neq 0$.

We assume that $\dim X =1$ or $2$.
In this case, $X$ is $\P^1$ or $\P^2$ (see \cite[Proposition 3.19 and 3.20]{Tanaka2}).
Let $D$ be a divisor on $\widetilde{X}$ that is the pull-back of an ample generator of $\Pic (X)$. 
Let $F$ be the unique member of $|2\tau^{\ast} (H)-E|$ (cf.\ Lemma \ref{lem:quadraticallynormal}).
Note that $\Pic (\widetilde{Q}) (\simeq \Pic (\widetilde{X}))$ is generated by $ \cO_{\widetilde{Q}} (\tau^{\ast}(H))$ and $\cO_{\widetilde{Q}} (E)$. 
Therefore, they are generated by $\cO_{\widetilde{Q}} (-K_{\widetilde{Q}})$ and $\cO_{\widetilde{Q}} (F)$.
We may write $D \sim \alpha (-K_{\widetilde{X}}) - \beta F'$ for $\alpha, \beta \in \Z$, where $F'$ is the strict transform of $F$ on $\widetilde{X}$.
Note that we have
\begin{equation}
\label{eqn:intersectionflopQ3}
(-K_{\widetilde{X}})^i \cdot (F')^{3-i} = (-K_{\widetilde{Q}})^i \cdot (F)^{3-i}
\end{equation}
for $1 \leq i \leq 3$ (see the proof of \cite[Lemma 4.5 (2)]{Tanaka2}).
Moreover, we have
\[
(-K_{\widetilde{Q}})^3 = 16, (-K_{\widetilde{Q}})^2\cdot F = 4, (-K_{\widetilde{Q}})\cdot F^2 = -2, \textup{ and } F^3 = -4.
\]
By \cite[Proposition 3.19 and Proposition 3.20]{Tanaka2},
\[
(-K_{\widetilde{X}})\cdot D^2 = (-K_{\widetilde{X}}) \cdot (\alpha (-K_{\widetilde{X}}) - \beta F')^2 = 16 \alpha^2 -8 \alpha \beta -2 \beta^2
\]
 is equal to 0 or 2.
These equations have no integral solutions, and we obtain a contradiction.

Therefore, we obtain $\dim X =3$, and $\mu$ is a contraction of a divisor $D$ to a point or a curve (cf.\ \cite[Proposition 3.22]{Tanaka2}).
As before, we set $D \sim \alpha (-K_{\widetilde{X}}) - \beta F'$.
First, we suppose that $D$ is contracted to a point.
By \cite[Proposition 3.22]{Tanaka2}, we have
\[
(-K_{\widetilde{X}})\cdot D^2 = 16 \alpha^2 -8 \alpha \beta - 2 \beta^2 =-2,
\]
\[
(-K_{\widetilde{X}})^2\cdot D = 16 \alpha -4\beta \in \{1,2,4\},
\]
\[
D^3 \in \{1,2,4\}.
\]
The first and the second equations have the unique solution $\alpha=0, \beta=-1$. 
{\cora
In this case, $\mu$ is a blow-up at a point in a smooth Fano threefold.
By \cite[Lemma 3.21]{Tanaka2}, $X$ equals a cubic threefold.
But then $\frac{1}{2}\mu^{\ast} (-K_{X}) \sim -\frac{1}{2}K_{\widetilde{X}} + F'$ is integral, and we obtain a contradiction.
}

Therefore, $\mu$ is a contraction of $D$ to a smooth curve $C$ on the smooth variety $X$ (cf.\ \cite[Proposition 3.22]{Tanaka2}).
Then, by the same argument to \cite[Lemma 2.2]{CM13}, we have $\beta = -r$.
Since $D \sim \alpha (-K_{\widetilde{X}}) + r F'$ and $D$ is a fixed prime divisor, we obtain $\alpha=0$ and $r=1$, and so $D=F'$. Then we have
\[
(\mu^{\ast} (-K_{X}))^3 = (-K_{\widetilde{X}})\cdot (-K_{\widetilde{X}} + D)^2
= 16 + 2 \cdot 4 -2 =22,
\]
which means that $X$ is a $V_{22}$-variety over $k$.
Moreover, we have
\[
\deg C = (-K_{X}) \cdot C = - (-K_{\widetilde{X}}+D) \cdot D^2 = - ((-K_{\widetilde{X}}+D)\cdot(K_{\widetilde{X}})\cdot D)=2,
\]
where the second equality follows from \cite[the proof of Lemma 3.21 (2)]{Tanaka2}, the third equality follows from 
$(\mu^{\ast} (-K_{X}))^2 \cdot D=0$,
and the fourth equality follows from 
(\ref{eqn:intersectionflopQ3}) and $D=F'$.
This completes the proof of (2) and (3).

The assertion (4) is clear from the construction and the uniqueness of the flop.
\end{proof}

\begin{prop}
\label{prop:integraltwo-rayV22toQ3}
Let $B$ be a connected excellent Dedekind scheme and $K$ the fraction field of $B$.
We denote the residue field of a closed point $s\in B$ by $k_s$.
Let $\mathcal{X}$ be a $V_{22}$-scheme over $B$,
$\mathcal{C}$ a relative smooth conic on $\mathcal{X}$, and $\mu \colon \widetilde{\mathcal{X}} \rightarrow \mathcal{X}$ a blow-up of $\mathcal{C}$ on $\mathcal{X}$.
Then the following hold.
\begin{enumerate}
\item 
The relative anti-canonical divisor $- K_{\widetilde{\mathcal{X}}/B}$ is semi-ample.
We denote the corresponding fiber-space morphism by $\Psi \colon \widetilde{\mathcal{X}} \rightarrow \mathcal{V}$. 
\item 
The morphisms $\Psi$, $\Psi_K$, and $\Psi_{k_s}$ for closed points $s\in B$ are flopping contractions.
Moreover, the flop $\Psi^{+} \colon \widetilde{\mathcal{Q}} \rightarrow \mathcal{V}$ exists, and $\Psi^{+}_K$ and $\Psi^{+}_{k_s}$ are the flops of $\Psi_K$ and $\Psi_{k_s}$, respectively.
\item 
There exists a $K_{\widetilde{\mathcal{Q}}/B}$-negative extremal ray contraction $\tau \colon \widetilde{\mathcal{Q}} \rightarrow \mathcal{Q}$ such that $\mathcal{Q}$ is a quadric threefold over $B$, and $\tau$ is a blow-up of a relative smooth sextic curve $\widetilde{\Gamma} \subset \mathcal{Q}$ of genus $0$.
Moreover, $\widetilde{\Gamma}$ is quadratically normal on each fiber.
\item 
The generic fiber and special fibers of this construction are the same as those we obtained in Proposition \ref{prop:two-rayV22toQ3}.
\end{enumerate}
\end{prop}

\begin{proof}
We may assume that $B$ is the spectrum of a Henselian discrete valuation ring $R$.
Then this follows from the same argument as in \cite[Proposition 4.12]{IKTT} by using Proposition \ref{prop:two-rayV22toQ3}.
Note that, in the equi-characteristic $p >0$ case, to prove (1), we also use \cite[Theorem 1.1]{Cascini-Tanaka} instead of \cite{Witaszek}.
Also, in the equi-characteristic $0$ case, semi-ampleness is reduced to that of fibers by the Kawamata-Viehweg vanishing $H^{i} (-mK_{\widetilde{X}_b}) = 0$ for $i>0$, $m\geq1$, and $b\in B$.
We also note that, in loc. cit.,\ we extend $R$ to show the existence of a relative smooth quintic curve and to ensure that the resulting scheme is a quadric threefold.
In our situation, we assume that there exists a conic $\mathcal{C}$.
Moreover, by Proposition \ref{prop:two-rayV22toQ3}, we know the generic and the special fibers of the resulting scheme $\mathcal{Q}$ are quadric threefolds.
Therefore, the closure of a hyperplane section of $\mathcal{Q}_{K}$ gives a very ample line bundle associating $\mathcal{Q} \hookrightarrow \P^4_R$,
and we do not need to extend $R$.
\end{proof}

\begin{prop}
\label{prop:integraltwo-rayQ3toV22}
Let $B$ be a connected excellent Dedekind scheme, and $K$ the fraction field of $B$.
We denote the residue field of a closed point $s\in B$ by $k_s$.
Let $\mathcal{Q} \subset \P_B(E)$ be a quadric threefold over $B$, 
$\widetilde{\Gamma} $ a relative smooth sextic curve of genus $0$ on $\mathcal{Q}$, and $\tau \colon \widetilde{\mathcal{Q}} \rightarrow \mathcal{Q}$ a blow-up of $\widetilde{\Gamma}$ on $\mathcal{Q}$.
We assume that $\widetilde{\Gamma}_K$ and $\widetilde{\Gamma}_{k_s}$ for closed points $s \in B$ are quadratically normal.
Then the following hold.
\begin{enumerate}
\item 
The relative anti-canonical divisor $- K_{\widetilde{\mathcal{Q}}/B}$ is base-point free.
We denote the corresponding fiber-space morphism by $\Phi \colon \widetilde{\mathcal{Q}} \rightarrow \mathcal{V}$. 
\item 
The morphisms $\Phi, \Phi_K$, and $\Phi_{k_s}$ for closed points $s \in B$ are flopping contractions.
Moreover, the flop $\Phi^+ \colon \widetilde{\mathcal{X}} \rightarrow \mathcal{V}$ exists, and $\Phi^+_K$ and $\Phi^+_{k_s}$ are the flops of $\Phi_K$ and $\Phi_{k_s}$, respectively.
\item 
There exists a $K_{\widetilde{X}/B}$-negative extremal ray contraction $\mu \colon \widetilde{\mathcal{X}}\rightarrow \mathcal{X}$ such that $\mathcal{X}$ is a $V_{22}$-scheme over $B$, and $\mu$ is a blow-up of a relative smooth conic $\mathcal{C} \subset \mathcal{X}$.
\item 
The generic fiber and special fibers of this construction are the same as those we obtained in Proposition \ref{prop:two-rayQ3toV22}.
\item 
This construction gives the converse of Proposition \ref{prop:integraltwo-rayV22toQ3}.
\end{enumerate}
\end{prop}

\begin{proof}
This follows from the same argument as in \cite[Proposition 4.12]{IKTT} by using Proposition \ref{prop:two-rayQ3toV22}.
See also the proof of Proposition \ref{prop:integraltwo-rayV22toQ3}.
\end{proof}

As a corollary, we obtain a $W(k)$-liftability of $V_{22}$-varieties, which was proved in \cite[Proposition 10.10]{Tanaka2}.

\begin{cor}
\label{cor:V22lift}
Let $k$ be an algebraically closed field of characteristic $p>0$, and $X$ a $V_{22}$-variety over $k$. 
Then $X$ is $W(k)$-liftable, and $b_i(X)$ is 1 (resp.\ 0) if $i$ is even (resp.\ odd).
\end{cor}

\begin{proof}
By Proposition \ref{prop:conicexists}, there exists a smooth conic $C$ on $X$.
By applying the two-ray game (Proposition \ref{prop:two-rayV22toQ3}) to $(X,C)$, we obtain quadric threefold and quadratically normal smooth rational sextic curve $(Q,\Gamma)$. 
By Lemma \ref{lem:liftquadricsextic}, we can take $W(k)$-lift $(\mathcal{Q}, \widetilde{\Gamma})$ of $(Q,\Gamma)$.
By applying the two-ray game (Proposition \ref{prop:integraltwo-rayQ3toV22}) to $(\mathcal{Q}, \widetilde{\Gamma})$, we obtain a desired $W(k)$-lift of $X$. 
The assertion on Betti numbers follows from the result over $\C$, which is well-known.
\end{proof}

\begin{prop}
\label{prop:lineexists}
Let $k$ be an algebraically closed field, and $X$ a $V_{22}$-variety over $k$. 
Then there exists a line $L$ on $X$.
\end{prop}

\begin{proof}
When $k$ is of characteristic zero, this is proved in \cite[Corollary 10.6]{Prokhorovconicnote}.
The general case is reduced to this case by Corollary \ref{cor:V22lift} (or \cite[Proposition 10.10]{Tanaka2}).
\end{proof}

\subsection{Two-ray game for lines}
\label{subsection:tworaylines}

\begin{prop}
\label{prop:two-rayV22toW5}
Let $k$ be a field, and $X$ a $V_{22}$-variety over $k$.
Let $L$ be a line on $X$, and $\mu \colon \widetilde{X} \rightarrow X$ a blow-up of $L$ on $X$.
Then the following hold.
\begin{enumerate}
\item 
The anti-canonical divisor $- K_{\widetilde{X}}$ is base-point.
We denote the corresponding fiber-space morphism by $\Psi \colon \widetilde{X} \rightarrow V$. 
\item 
The morphism $\Psi$ is a flopping contraction, and the flop of $\Psi$ exists.
We denote the flop by $\Psi^{+} \colon \widetilde{Y} \rightarrow V$.
\item 
Let $\tau \colon \widetilde{Y} \rightarrow Y$ be the contraction of the $K_{\widetilde{Y}}$-negative extremal ray.
Then $Y$ is a $V_5$-variety over $k$, and $\tau$ is a blow-up of a smooth rational quintic curve $Z$ on $Y$.
\end{enumerate}
\end{prop}

\begin{proof}
If $k$ is of characteristic zero, this is a classical result (\cite{Iskovskikhdoubleprojection}, \cite{Prokhorovtworay}, \cite[Theorem 4.3.7]{Iskovskikh-Prokhorov}).
(1) follows from \cite[Proposition 5.1]{Tanaka2}.
(2) follows from \cite[Proposition 5.1, Proposition 5.2 and Corollary 8.3]{Tanaka2} and \cite[Theorem 1.1]{Tanakaflop}.
(3) follows from \cite[Theorem 8.1]{Tanaka2}.
\end{proof}

\begin{lem}
\label{lem:normalityofrationalquintic}
Let $k$ be an algebraically closed field, and $Y \subset \P^6$ a $V_5$-variety over $k$.
Let $Z$ be a smooth rational quintic curve on $Y.$
Then $Z$ is normal, i.e.,\ $Z$ spans a linear subspace $\P^5 \subset \P^6$.
\end{lem}

\begin{proof}
Suppose that $Z \subset \P^4$ for some linear subspace $\P^4 \subset \P^6$.
Then we have 
$Z \subset Y \cap \P^4 \subset \P^6$.
Since $Y$ is a threefold and $Z$ is a quintic curve, we obtain $Z =Y\cap \P^4$.
Since $\cO(K_{Y}) = \cO (-2H)$, we obtain $\cO_{Z}(K_{Z}) = \cO_{Z}$ which contradicts the fact $Z$ is rational.
This completes the proof.
\end{proof}

\begin{prop}
\label{prop:two-rayW5toV22}
Let $k$ be a field,
and $Y$ a $V_5$-variety over $k$, which is embedded into $\P^6$.
Let $Z$ be a smooth rational quintic curve on $Y$, and $\tau \colon \widetilde{Y} \rightarrow Y$ be a blow-up of $Z$ on $Y$.
Then the following hold.
\begin{enumerate}
    \item 
    The anti-canonical divisor $-K_{\widetilde{Y}}$ is base-point free.
    We denote the corresponding fiber space morphism by $\Phi \colon \widetilde{Y} \rightarrow V$.
    \item
    $\Phi$ is a flopping contraction, and the flop of $\Phi$ exists.
    We denote the flop by $\Phi^{+} \colon \widetilde{X} \rightarrow V$.
    \item 
    Let $\mu \colon \widetilde{X} \rightarrow X$ be the contraction of the $K_{\widetilde{X}}$-negative extremal ray.
    Then $X$ is a $V_{22}$-variety over $k$, and $\mu$ is  a blow-up of a line $L$ on $X$.
    \item 
    This construction gives the converse of the construction in Proposition \ref{prop:two-rayV22toW5}.
\end{enumerate}
\end{prop}

\begin{proof}
Note that, $Z$ is normal, i.e.,\ $Z$ spans a linear subspace $\P^5 \subset \P^6$ by Lemma \ref{lem:normalityofrationalquintic}.
All the assertions can be reduced to the case where $k$ is algebraically closed, so we show in this case.

If $k$ is of characteristic zero, this is proved by Prokhorov (\cite{Prokhorovtworay}).
In the general case, essentially the same proof works, combining with the recent classification of Fano 3-folds (\cite{Tanaka1}, \cite{Tanaka2}).
We sketch the proof.

Let $H$ be a hyperplane section of $Y \subset \P^6$, and $F_Y \sim H$ a particular hyperplane section of $Y$ containing $Z$ (which is spanned by $Z$).
Let $E$ be the exceptional divisor of $\tau$, and $F$ the strict transform of $F_Y$.
Then, we obtain
$(-K_{\widetilde{Y}})^{3} =18, (-K_{\widetilde{Y}})^2 \cdot F =3, (-K_{\widetilde{Y}})\cdot F^2 =-2, F^3 = -2$
(see \cite[Lemma 1.1 and 1.2]{Prokhorovtworay}. See also \cite[Lemma 3.21]{Tanaka2}).
Note that we have $\Pic (\widetilde{Y}) = \Z K_{\widetilde{Y}} + \Z F.$

(1) follows from the fact that any rational normal curve is an intersection of quadrics (see \cite[Lemma 1.3 (1)]{Prokhorovtworay}).

Next, we show that $\Phi$ is either a flopping contraction or an isomorphism.
Since $-K_{\widetilde{Y}}$ is big by (1) and $(-K_{\widetilde{Y}})^3 =18$, $\Phi$ is a birational morphism.
Suppose that there exists an effective irreducible divisor $D$ which is contracted by $\Phi$.
We may write $D \sim \alpha (-K_{\widetilde{Y}}) - \beta F$
for $\alpha, \beta \in \Z$.
By the assumption, we have $D \cdot (-K_{\widetilde{Y}})^{2} =0$, which implies $\beta = 6 \alpha$.
Since $\kappa ( -6 \alpha F) = \kappa (D + \alpha K_{\widetilde{Y}}) \geq \kappa (\alpha K_{\widetilde{Y}})$ and $-K_{\widetilde{Y}}$ is big, we obtain $\alpha >0$.
We have $D \sim \alpha (-4 \tau^{\ast}(H) + 5 E)$.
Since $f\cdot K_{\widetilde{Y}} = -1$ for any fiber $f$ of $E \rightarrow Z$, we obtain $D \cdot f < -1$.
This implies $D = E$, but it contradicts 
\[
(-K_{\widetilde{Y}})^2 E
= (-\tau^{\ast} K_{Y} - E)^2 E = -2 (K_{Y} \cdot Z) + (E)^3 =12
\]
(see \cite[Lemma 3.21 (2)]{Tanaka2}).

{\cora By \cite[Table 1.2]{Tanaka3} and $(-K_{\widetilde{Y}})^3=18$, we can show that $-K_{\widetilde{Y}}$ is not ample. Therefore, $\Phi$ is a flopping contraction.}
{\cora  Therefore, there exists a flop $\Phi^{+}\colon \widetilde{X} \rightarrow V$ of $\Phi$ by \cite[Theorem 1.1]{Tanakaflop}.
 Let $\mu \colon \widetilde{X} \rightarrow X$ be a $K_{\widetilde{X}}$-negative extremal contraction (which is unique since $\rho (\widetilde{X}) =2$ and $- K_{\widetilde{X}}$ is not ample but nef and big).
}
Since $\rho (\widetilde{X}) =2$, we have $\dim X \neq 0$.

Next, we assume that $\dim X= 1$ or $2$.
In this case, $X$ is $\P^1$ or $\P^2$ (see \cite[Proposition 3.19 and 3.20]{Tanaka2}). 
Let $D$ be a divisor on $\widetilde{X}$ which is the pull-back of an ample Cartier divisor which is a generator of $\Pic (X)$.
Let $F'$ be the strict transform of $F$ on $\widetilde{X}.$
We may write $D \sim \alpha (- K_{\widetilde{X}}) - \beta F'$ for $\alpha, \beta \in \Z$.
Note that we have
\begin{equation}
\label{eqn:intersectionflop}
(-K_{\widetilde{X}})^i \cdot (F')^{3-i} = (-K_{\widetilde{Y}})^i \cdot (F)^{3-i}
\end{equation}
for $1 \leq i \leq 3$ (see the proof of \cite[Lemma 4.5 (2)]{Tanaka2}).
By \cite[Proposition 3.19 and 3.20]{Tanaka2}, $(-K_{\widetilde{X}}) \cdot D^2 = 18 \alpha^2 - 6\alpha \beta -2 \beta^2$ is equal to 0 or 2.
These equations have no integral solutions, and we obtain a contradiction.

Therefore, we obtain $\dim X =3$, and $\mu$ is a contraction of a divisor $D$ to a point or a curve (cf. \cite[Proposition 3.22]{Tanaka2}).
As before, we set $D \sim \alpha (- K_{\widetilde{X}}) - \beta F'$.
First, we suppose that $D$ is contracted to a point.
Since we have 
\[
(-K_{\widetilde{X}})^2\cdot D =
18 \alpha - 3 \beta \in \{1,2,4\}
\]
by \cite[Proposition 3.22]{Tanaka2}, we obtain a contradiction.
Therefore, $\mu$ is a contraction of $D$ to a smooth curve $L$ on the smooth variety $X$  (cf.\ \cite[Proposition 3.22]{Tanaka2}), which is a smooth Fano threefold over $k$.
Let $r$ be the index of $X$.
Then, by the same argument to \cite[Lemma 2.2]{CM13}, we have $\beta = -r$.
Since $D \sim \alpha (-K_{\widetilde{X}}) + r F'$ and $D$ is a fixed prime divisor, we obtain $\alpha =0$ and $r =1$, and so $D=F'$.
Then we have
\[
(\mu^{\ast} (-K_{X}))^3 
=(-K_{\widetilde{X}}) \cdot (-K_{\widetilde{X}}+D)^2
= 18+2\cdot 3 + (-2) =22,
\]
this means $X$ is a $V_{22}$-variety over $k$.
Note that we use (\ref{eqn:intersectionflop}) for the third equality.
Moreover, we have
\[
\deg L = (-K_{X}) \cdot L = - (-K_{\widetilde{X}}+D) \cdot D^2 = - ((-K_{\widetilde{X}}+D) \cdot (K_{\widetilde{X}}) \cdot D) = 1
\]
where the second equality follows from \cite[the proof of Lemma 3.21 (2)]{Tanaka2}, the third equality follows from 
$(\mu^{\ast} (-K_{X}))^2 \cdot D=0$,
and the fourth equality follows from 
(\ref{eqn:intersectionflop}) and $D=F'$.
This completes the proof of (2) and (3).

The assertion (4) is clear from the construction and the uniqueness of the flop.
\end{proof}

\begin{prop}
\label{prop:integraltwo-rayV22toW5}
Let $B$ be a connected excellent Dedekind scheme, and $K$ the fraction field of $B$.
We denote the residue field of a closed point $s\in B$ by $k_s$.
Let $\mathcal{X}$ be a $V_{22}$-scheme over $B$,
$\mathcal{L}$ a relative line on $\mathcal{X}$, and $\mu \colon \widetilde{\mathcal{X}} \rightarrow \mathcal{X}$ a blow-up of $\mathcal{L}$ on $\mathcal{X}$.
Then the following hold.
\begin{enumerate}
    \item 
    The relative anti-canonical divisor $-K_{\widetilde{\mathcal{X}}/B}$ is semi-ample.
    We denote the corresponding fiber-space morphism by $\Psi \colon \widetilde{\mathcal{X}} \rightarrow \mathcal{V}$.
    \item 
    The morphisms $\Psi$, $\Psi_K$, and $\Psi_{k_s}$ for closed points $s \in B$ are flopping contractions.
    Moreover, the flop $\Psi^{+} \colon \widetilde{\mathcal{Y}} \rightarrow \mathcal{V}$ exists, and $\Psi^{+}_{K}$ and $\Psi^+_{k_s}$) are the flops of $\Psi_K$ and $\Psi_{k_s}$, respectively.
    \item 
    There exists a $K_{\widetilde{\mathcal{Y}}/B}$-negative extremal ray contraction 
    $\mu \colon \widetilde{\mathcal{Y}} \rightarrow \mathcal{Y}$ such that $\mathcal{Y}$ is a $V_5$-scheme over $B$ , and $\mu$ is a blow-up of a relative smooth rational quintic curve $\mathcal{Z}$ on $\mathcal{Y}$.
    \item 
    The generic fiber and the special fibers of this construction are the same as those we obtained from Proposition \ref{prop:two-rayV22toW5}.
    \end{enumerate}
\end{prop}

\begin{proof}
This follows from the same argument as in \cite[Proposition 4.12]{IKTT} by using Proposition \ref{prop:two-rayV22toW5}. See also the proof of Proposition \ref{prop:integraltwo-rayV22toQ3}.
\end{proof}

\begin{prop}
\label{prop:integraltwo-rayW5toV22}
Let $B$ be a connected excellent Dedekind scheme, and $K$ the fraction field of $B$.
We denote the residue field of a closed point $s\in B$ by $k_s$.
Let $\mathcal{Y}$ be a $V_5$-scheme over $B$,
$\mathcal{Z}$ a smooth rational quintic curve on $\mathcal{Y}$, and $\mu \colon \widetilde{\mathcal{Y}} \rightarrow \mathcal{Y}$ a blow-up of $\mathcal{Z}$ on $\mathcal{Y}$.
Then the following hold.
\begin{enumerate}
    \item 
    The relative anti-canonical divisor $-K_{\widetilde{\mathcal{Y}}/B}$ is semi-ample.
    We denote the corresponding fiber-space morphism by $\Phi \colon \widetilde{\mathcal{Y}} \rightarrow \mathcal{V}$.
    \item 
    The morphisms $\Phi$, $\Phi_K$, and $\Phi_{k_s}$ for closed points $s\in B$ are flopping contractions.
    Moreover, the flop $\Phi^{+} \colon \widetilde{\mathcal{X}} \rightarrow \mathcal{V}$ exists, and $\Phi^{+}_{K}$ and $\Phi^+_{k_s}$ are the flops of $\Phi_K$ and $\Phi_{k_s}$, respectively.
    \item 
    There exists a $K_{\widetilde{\mathcal{X}}/B}$-negative extremal ray contraction 
    $\mu \colon \widetilde{\mathcal{X}} \rightarrow \mathcal{X}$ such that $\mathcal{X}$ is a $V_{22}$-scheme over $B$ , and $\mu$ is a blow-up of a relative line $\mathcal{L}$ on $\mathcal{X}$.
    \item 
    The generic fiber and the special fibers of this construction are the same as those we obtained in Proposition \ref{prop:two-rayV22toW5}.
    \item 
    This construction gives the converse of Theorem \ref{prop:integraltwo-rayV22toW5}.
    \end{enumerate}
\end{prop}

\begin{proof}
This follows from the same argument as in \cite[Proposition 4.12]{IKTT} by using Proposition \ref{prop:two-rayW5toV22}.
See also the proof of Proposition \ref{prop:integraltwo-rayV22toQ3}.
\end{proof}

\begin{defn}
Let $k$ be an algebraically closed field, and $X$ a $V_{22}$-variety over $k$.
By \cite[Proposition 4.2.1]{Iskovskikh-Prokhorov} and \cite[Proposition 5.1]{Tanaka2}, we have either
\[
N_{L/X} \simeq \cO_L \oplus \cO_L (-1) \quad \textup{or} \quad
N_{L/X} \simeq \cO_L (1) \oplus \cO_L (-2).
\]
In the former case, we say $L$ is an \emph{ordinary line}, and the latter case, we say $L$ is a \emph{special line}.
\end{defn}

\begin{prop}
\label{prop:tworayfloppinglocus}
Let $k$ be an algebraically closed field, and $Y \subset \P^6$ a $V_5$-variety over $k$.
Let $Z \subset Y$ be a smooth rational quintic curve which is of Mukai--Umemura type, $\Ga$-type, or $\Gm$-type.
Let 
\[
Y \xleftarrow{\tau} \widetilde{Y} \xrightarrow{\Phi} V \xleftarrow{\Phi^{+}} \widetilde{X} \xrightarrow{\mu}X
\]
be as in Proposition \ref{prop:two-rayW5toV22},
where $\mu$ is a blow-up of a line $L$ on $X$.
Then the following hold.
\begin{enumerate}
\item 
The exceptional locus of $\Phi$ in $\widetilde{Y}$ is a union of strict transforms of a bisecant line of $Z \subset Y$. Moreover, this locus is irreducible.
\item 
The exceptional locus of $\Phi^{+}$ in $\widetilde{X}$ is an exceptional section of the $\mu$-exceptional divisor $D \subset \widetilde{X}$.
Moreover, $L$ is a special line.
\item
There is no line (other than $L$) on $X$ that intersects $L$.
\end{enumerate}
\end{prop}

\begin{proof}
If $\chara k =0$, these are proved in \cite[Proposition 5.4.3]{Kuznetsov-Prokhorov-Shramov}.
Therefore, we assume that $\chara k >0.$
We shall show (1).
The first part of (1) follows from the same argument as in \cite[Lemma 5.2.5]{Kuznetsov-Prokhorov-Shramov}.
Then the second part follows since a bisecant line of $Z \subset X $ is unique by Propositions \ref{prop:SigmaZp>2} and \ref{prop:SigmaZp=2}.

Next, we shall show (2).
Note that, by (1), the exceptional locus of $\Phi^{+}$ is irreducible.
We may assume that $Y \subset \P^6$ is given by (\ref{eqn:W5}), and
$Z$ is given by (\ref{eqn:ZMU}), (\ref{eqn:ZGageneral}) with $\xi=1$, or (\ref{eqn:ZGm}).
In this case, we can take a characteristic zero lift $(\mathcal{Y}, \mathcal{Z})$ of $(Y,Z)$, where $\mathcal{Y} \subset \P^6_{W(k)}$ is defined by (\ref{eqn:W5}) and $\mathcal{Z}$ is defined by the same equations as in (\ref{eqn:ZMU}), (\ref{eqn:ZGageneral}) with $\xi=1$, or (\ref{eqn:ZGm}).
Then by Proposition \ref{prop:integraltwo-rayW5toV22}, we can lift the two-ray game diagram as
\[
\mathcal{Y} \xleftarrow{\tau} \mathcal{\widetilde{Y}} \xrightarrow{\Phi} \mathcal{V} \xleftarrow{\Phi^{+}} \mathcal{\widetilde{X}} \xrightarrow{\mu} \mathcal{X},
\]
where $\mu$ is a blow-up of a relative line $\mathcal{L}$ on $\mathcal{X}$.
Since $\mathcal{L}_{\Frac (W(k))}$ is special, $\mathcal{L}_k =L$ is special by the upper-semi continuity.
Then the $\mu$-exceptional divisor $\mathcal{D}$ is a relative Hirzebruch surface $\Sigma_3$.
Moreover, the exceptional locus of $\Phi^{+}_{\Frac (W(k))}$ contains a exceptional section of $\mathcal{D}_{\Frac (W (k))}$.
By taking closure, the exceptional locus of $\Phi^{+}_{k}$ contains an exceptional section of $D$, and we obtain the result.

Finally, we show (3).
Let $L'$ be a line on $X$ that intersects $L$.
We denote its strict transform on $\widetilde{X}$ by $\widetilde{L'}$.
Then we have 
\[
(- K_{\widetilde{X}}, \widetilde{L'}) = (-K_{X}, L') - (D, \widetilde{L'}) = 1-1 =0,
\]
that means $\widetilde{L'}$ is a flopping curve other than an exceptional section of $D$.
It contradicts (2), so there is no such $L'$.
\end{proof}

\begin{prop}
\label{prop:tworaySigma}
In the situation in Proposition \ref{prop:tworayfloppinglocus}, we have an isomorphism
\[
\Sigma_{Z} (Y)^{\circ} \simeq \Sigma (X)_{\red}\setminus \{L\} 
\]
induced by a rational map $ \mu \circ (\Phi^{+})^{-1} \circ \Phi\circ \tau^{-1}$.
Here, $\Sigma_{Z} (Y)^{\circ} \subset \Sigma_{Z} (Y)$ is the open subscheme that parameterizes lines that are neither bisecants of $Z$, nor intersect any bisecant of $Z$.
\end{prop}

\begin{proof}
This follows from the same proof as in \cite[Lemma 5.2.8]{Kuznetsov-Prokhorov-Shramov}.
\end{proof}

\begin{lem}
\label{lem:sigmafaithful}
Let $k$ be an algebraically closed field of characteristic $p>0$, $Y$ a $V_5$-variety over $k$, and $Z \subset Y$ a smooth rational quintic curve of Mukai--Umemura type, $\Ga$-type, or $\Gm$-type.
Let $(X,L)$ be a $V_{22}$-variety over $k$ with a line on $X$ obtained by applying the two-ray game to $(Y,Z)$ (Proposition \ref{prop:two-rayW5toV22}).
Then the natural morphism
\[
\rho \colon\Aut_{X/k, \mathrm{red}} \rightarrow \Aut_{\Sigma(X)_{\mathrm{red}}/k}
\]
is injective.
\end{lem}

\begin{proof}
When $\chara k =0$, Prokhorov shows this result without the assumption on $Z$ (see \cite[Claim A.1.1]{Dedieu-Manivel}).
In his proof, he uses the smoothness of a fixed locus of an action of a finite group on a smooth variety.
The authors do not know whether this hold in characteristic $p>0$.
Therefore, 
instead of his method, we use the classification of quintic curves with infinite stabilizer.

By the proof of Proposition \ref{prop:tworaySigma},
the two-ray game maps a general line $L'$ on $X$ to a line $L'_{Y}$ on $Z$ that intersects $Z$ at one point.
By the assumption, any element $g \in \ker \rho (k) \subset \Aut (X,L) \simeq \Aut (Y,Z)$ stabilizes such a point on $Z$.
Since $\Sigma (X)$ is 1-dimensional (cf. Proposition \ref{prop:tworaySigma}),
$g$ has infinitely many fixed closed points on $Z$.
Therefore, $g$ acts trivially on $Z$.
Since $\Aut_{(Y,Z)/k} \rightarrow \Aut_{Z/k}$ is a closed immersion by Remarks \ref{rem:stabilizeractiononZ_p>2} and \ref{rem:stabilizeractiononZ_p=2}, we have $g=1$.
Therefore, $\rho$ is injective.
\end{proof}

\section{$V_{22}$-varieties in positive and mixed characteristic}
\label{section:V_22-variety_in_positive_mixed}

\subsection{Mukai--Umemura varieties}

\begin{lem}
\label{lem:MUSigma}
Let $k$ be an algebraically closed field of characteristic $p>0$, 
$Y$ a $V_5$-variety over $k$, and $Z \subset Y$ a smooth rational quintic curve of Mukai--Umemura type.
Let $(X, L)$ be the $V_{22}$-scheme with the line, that is obtained by applying the two-ray game (Proposition \ref{prop:two-rayW5toV22}) to $(Y,Z)$.
Then $\Sigma (X)_{\red} \setminus \{ L\}$ is isomorphic to $\mathbb{A}^1_k$.
\footnote{In Theorem \ref{thm:MukaiUmemuraclassification}, we show that $\Sigma(X)_{\red}$ is isomorphic to $\P^1$.}
\end{lem}

\begin{proof}
    This follows from Propositions \ref{prop:SigmaZp>2} and \ref{prop:tworaySigma}.
\end{proof}

\begin{thm}
\label{thm:MukaiUmemuraclassification}
Let $k$ be an algebraically closed field of characteristic $p \geq 0$.
Then the following hold.
\begin{enumerate}
\item 
If $p=2$ or $5$, then there is no $V_{22}$-variety $X$ over $k$ such that $\dim \Aut_{X/k} \geq 3$.
\item 
If $p\neq 2,5$, then there is a unique $V_{22}$-variety over $k$ such that $\dim \Aut_{X/k} \geq 3$. 
In this case, $X$ is obtained by applying the two-ray game (Proposition \ref{prop:two-rayW5toV22}) for $(Y,Z)$, where $Y$ is a $V_5$-variety over $k$ and $Z$ is a smooth rational quintic curve on $Y$ of Mukai--Umemura type.
Moreover, $\Sigma(X)_{\red}$ is isomorphic to $\P^1_k$,
$X$ is a Mukai--Umemura variety, and we have $\Aut_{X/k, \red} \simeq \PGL_{2,k}$.
\end{enumerate}
\end{thm}

\begin{proof}
Let $X$ be a $V_{22}$-variety over $k$ such that $\dim \Aut_{X/k} \geq 3$.
Since there exists a line on $X$ (Proposition \ref{prop:lineexists}), by \cite[Proposition 5.4]{Tanaka2}, any irreducible component of the Hilbert scheme of lines on $X$ is $1$-dimensional.
Take a line $L$ on $X$.
Then the stabilizer $\Aut_{X,L} \subset \Aut_{X}$ of $L$ has dimension greater than or equal to $2$.
By applying the two-ray game (Proposition \ref{prop:two-rayV22toW5}) for $(X,L)$, we obtain a $V_5$-variety $Y$ with a smooth rational quintic curve $Z$. 
By the two-ray game diagram, we have $\Aut_{X,L} \simeq \Aut_{Y,Z}$.
By Theorems \ref{thm:BGaquintic>2} and \ref{thm:BGaquintic2}, we obtain a contradiction when $p=2$ or $5$.
When $p\neq 2, 5$, by the same theorems, 
{\cora $(Y,Z)$ is unique up to isomorphism.}
By the two-ray game diagram again, $X$ is unique up to isomorphism.

Conversely, let $(X_0, L_0)$ be the $V_{22}$-variety with the line obtained by applying the two-ray game to $(Y_0,Z_0)$, where $Y_0$ is a $V_5$-variety over $k$ and $Z_0 \subset Y_0$ is a smooth rational quintic curve of Mukai--Umemura type.
We shall show that $X_0$ is a Mukai--Umemura variety.
This is a classical result if $\chara k=0$ (see \cite{Mukai-Umemura} and \cite[Claim A.1.1]{Dedieu-Manivel}).
By Lemma \ref{lem:MUSigma}, $\Sigma(X_0)_{\red}$ is a one-point compactification of $\mathbb{A}^1_k$.
By Lemma \ref{lem:sigmafaithful}, we have an injective morphism
\[
\rho \colon \Aut_{X_0/k, \red} \hookrightarrow 
\Aut_{\Sigma(X_0)_{\red}/k}.
\]
Note that we may assume that $Y_0 \subset \P^6$ is given by (\ref{eqn:W5}) and $Z_0$ by (\ref{eqn:ZMU}).
Let $(\mathcal{Y}_0, \mathcal{Z}_0)$ be the $W(k)$-lift of $(Y_0,Z_0)$ defined by the same equations as in (\ref{eqn:W5}) and (\ref{eqn:ZMU}).
By Proposition \ref{prop:integraltwo-rayW5toV22}, we obtain a $W(k)$-lift $\mathcal{X}_0$ of $X_0$. 
Since $\Aut_{\mathcal{X}_{0,\Frac (W(k))}/ \Frac (W(k))}$ is 3-dimensional, $\dim \Aut_{X_0/k} \geq 3$.
If $\Sigma (X_0)_{\red}$ has a (unique) singular point, then $\dim \Aut_{\Sigma(X_0)_{\red}/k} \leq 2$, and we obtain a contradiction. 
Therefore, $\Sigma (X_0)_{\red}$ is isomorphic to $\P^1$, and $\rho$ is a bijection.
Since its restriction to the identity component is also a bijection, $\Aut_{X_0/k, \red}$ is connected.
Moreover, by the bijection, $\Aut_{X_0/k, \red}$ is a semi-simple reductive group of semi-simple rank $1$ with a connected center.
Since $p\neq 2$, we have $\Aut_{X/k,\red} \simeq \PGL_{2,k}$.
\end{proof}

\begin{cor}
\label{cor:liftMukaiUmemura}
Let $k$ be an algebraically closed field of positive characteristic, and $X$ a Mukai--Umemura variety over $k$.
Then $X$ admits a $W(k)$-lift $\mathcal{X}$ that is a Mukai--Umemura scheme over $W(k)$.
\end{cor}

\begin{proof}
Let $\mathcal{Y}$ over $W(k)$ be the split $V_5$-scheme defined by $(\ref{eqn:W5})$,
and $\mathcal{Z} \subset \mathcal{Y}$ a relative smooth rational quintic curve defined by (\ref{eqn:ZMU}).
Let $\mathcal{X}$ be the $V_{22}$-scheme over $W(k)$ obtained by applying the two-ray game to $(\mathcal{Y}, \mathcal{Z})$ (Proposition \ref{prop:integraltwo-rayW5toV22}).
Then $\mathcal{X}$ satisfies the desired condition by Theorem \ref{thm:MukaiUmemuraclassification}.
\end{proof}

\begin{lem}
\label{lem:splitMU}
We set $R := \Z[1/10]$.
Let $\mathcal{Y} \subset \P^6_R$ be the split $V_5$-scheme over $R$ defined by (\ref{eqn:W5}), and
$\mathcal{Z} \subset \mathcal{Y}$ be the relative smooth rational quintic curve defined by (\ref{eqn:ZMU}).
Let $\mathcal{X}$ be the $V_{22}$-scheme over $R$ obtained by applying the two-ray game to $(\mathcal{Y},\mathcal{Z})$ (Proposition \ref{prop:integraltwo-rayW5toV22}).
Then for any perfect field $k$ over $R$, we have
$\Aut_{\mathcal{X}_{k}/k, \red} \simeq \PGL_{2,k}$.
\end{lem}

\begin{proof}
By Theorem \ref{thm:MukaiUmemuraclassification}, we have
$\Aut_{\mathcal{X}_{\overline{k}}/\overline{k}, \red} \simeq \PGL_{2, \overline{k}}$.
Moreover, by the construction of $\mathcal{X}$ and Theorem \ref{thm:BGaquintic>2},
$\Aut_{\mathcal{X}_k/k, \red}$ admits a split maximal torus over $k$.
Therefore, we obtain the desired assertion.
\end{proof}

\begin{defn}
\label{defn:splitMukai-Umemura}
Let $k$ be a perfect field of characteristic $p\neq 2,5$ and $X$ a Mukai--Umemura variety over $k$.
We say $X$ is a \emph{split Mukai--Umemura variety} if 
$X$ is $k$-isomorphic to the $V_{22}$-variety $\mathcal{X}_k$ defined in Lemma \ref{lem:splitMU}.
\end{defn}

\begin{thm}
\label{thm:Mukai-Umemurageneral}
Let $k$ be a field. 
\begin{enumerate}
\item 
If $k$ is of characteristic two or five, there is no Mukai--Umemura variety over $k$.
\item 
If $k$ is perfect field of characteristic $p \neq 2,5$,
then we have a one-to-one correspondence
\[
H^1 (G_k, \PGL_{2,k} (\overline{k}))
\simeq
\{ 
\textup{Mukai--Umemura variety over }k 
\}/k\textup{-isom}.
\]
Moreover, the trivial class is sent to a split Mukai--Umemura variety over $k$.
\end{enumerate}
\end{thm}

\begin{proof}
(1) follows from Theorem \ref{thm:MukaiUmemuraclassification}.
(2) follows from Lemma \ref{lem:splitMU} (cf.\ Definition \ref{defn:splitMukai-Umemura}).
\end{proof}

\begin{prop}
\label{prop:MUsplitcriteria}
Let $X$ be a Mukai--Umemura variety over a perfect field $k$ of characteristic $p \neq 2$, $5$.
Then the following are equivalent.
\begin{enumerate}
    \item 
    $X$ is a split Mukai--Umemura variety.
    \item
    There exists an isomorphism $\Sigma (X)_{\red} \simeq \P^1_k$.
    \item
    There exists a line $L$ on $X$.
    \item 
    We have $\Aut_{X/k, \red} \simeq \PGL_{2,k}$.
\end{enumerate}
\end{prop}

\begin{proof}
Let $X_0$ be a split Mukai--Umemura variety.
By the proof of Theorem \ref{thm:MukaiUmemuraclassification}, the natural morphism
\[
\Aut_{X_0/k, \red} (\overline{k}) \rightarrow \Aut_{\Sigma(X_0)_{\red}/k} (\overline{k}) \simeq \PGL_{2,k} (\overline{k})
\]
is a $G_k$-equivariant bijection.
Also, we note that $\Sigma(X_0)_{\red, \overline{k}} \simeq \P^1_{\overline{k}}$ by Theorem \ref{thm:MukaiUmemuraclassification}.
Moreover, by the construction of $X_0$, there exists a line $L_0$ on $X_0$, so $\Sigma (X_0) _{\red}\simeq \P^1_k$.
Therefore, (1) and (2) are equivalent.

The equivalence of (2) and (3) is clear (e.g.,\ \cite[Proposition 4.5.12]{Poonenrationalpoints}).

The equivalence of (2) and (4) follows from Lemma \ref{lem:splitMU} since the automorphism group of $\PGL_{2,k}$ is the inner automorphism group $\PGL_{2,k} (k)$ (e.g.,\ \cite[(1.5.2)]{Conradreductive}).
\end{proof}

\subsection{$V_{22}$-varieties of $\Ga$-type}

\begin{lem}
\label{lem:GaSigma}
Let $k$ be an algebraically closed field of characteristic $p\neq 2, 5$, 
$Y$ a $V_5$-variety over $k$, and $Z \subset Y$ a smooth rational quintic curve of $\Ga$-type.
Let $(X, L)$ be the $V_{22}$-scheme with the line, that is obtained by applying the two-ray game (Proposition \ref{prop:two-rayW5toV22}) to $(Y,Z)$.
Then $\Sigma (X)_{\red}$ is a union of two rational curves that intersect at one point $L$.
Moreover, the singular locus of $\Sigma (X)_{\red}$ is equal to $\{L\}$.
\end{lem}

\begin{proof}
This follows from Propositions \ref{prop:SigmaZp>2} and \ref{prop:tworaySigma}.
\end{proof}

\begin{thm}
\label{thm:Gaclassification}
Let $k$ be an algebraically closed field of characteristic $p$, and $X$ a $V_{22}$-variety of $\Ga$-type over $k$.
Then 
$p \neq 2,5$, and
$X$ is obtained by applying the two-ray game (Proposition \ref{prop:two-rayW5toV22}) to $(Y,Z)$, where $Y$ is a $V_5$-variety over $k$ and $Z$ is a smooth rational quintic curve on $Y$ of $\Ga$-type.
In particular, there is a unique $V_{22}$-variety of $\Ga$-type over $k$, and $\Sigma (X)_{\red}$ is as in Lemma \ref{lem:GaSigma}.
Moreover, we have
$
\Aut_{X/k, \red} \simeq \Ga \rtimes \mu_4,
$
where the semi-direct product is taken with respect to (\ref{eqn:semi-directGa}).
\end{thm}

\begin{proof}
Let $X$ be a $V_{22}$-variety of $\Ga$-type over $k$.
Since there exists a line on $X$ (Proposition \ref{prop:lineexists})
and $\Aut_{X/k, \red}^{\circ}\simeq \Ga$ is a connected solvable algebraic group, there exists a line $L$ on $X$ such that $L$ is $\Ga$-stable  \cite[Theorem VIII.21.2]{Humphreys}.
By applying the two-ray game (Proposition \ref{prop:two-rayV22toW5}) for $(X,L)$, we obtain a $V_5$-variety $Y$ with a smooth rational quintic curve $Z$. 
By the two-ray game diagram, we have $\Aut_{X,L} \simeq \Aut_{Y,Z}$.
By Theorems \ref{thm:BGaquintic>2} and \ref{thm:BGaquintic2}, we obtain a contradiction when $p=2$ or $5$.
When $p\neq 2,5$, by the same theorems, 
{\cora $(Y,Z)$ is unique up to isomorphism,}
and $Z$ is of $\Ga$-type.
By the two-ray game diagram again, $X$ is unique up to isomorphism if it exists.

Conversely, let $(X_0, L_0)$ be the $V_{22}$-scheme with the line obtained by applying the two-ray game to $(Y_0,Z_0)$, where $Y_0$ is a $V_5$-variety over $k$ and $Z_0 \subset Y_0$ is a smooth rational quintic curve of $\Ga$-type.
It suffices to compute $\Aut_{X_0/k, \red}$.
By
Theorem \ref{thm:BGaquintic>2},
we obtain 
\[
(\Ga \rtimes \mu_4 ) \simeq \Aut_{(Y_0,Z_0)/k} \simeq \Aut_{(X_0,L_0)/k}  \hookrightarrow \Aut_{X_0/k}.
\]
Since $L_0$ is a unique singular point of $\Sigma (X_0)_{\red}$, we have
$\Aut_{(X_0,L_0)/k, \red} \simeq \Aut_{X_0/k, \red}$.
Therefore, we obtain $\Aut_{X_0/k,\red} \simeq \Ga \rtimes \mu_4$, and in particular $X_0$ is a $V_{22}$-variety of $\Ga$-type.
\end{proof}

\begin{cor}
\label{cor:liftGa}
Let $k$ be an algebraically closed field of positive characteristic, and $X$ a $V_{22}$-variety of $\Ga$-type over $k$.
Then $X$ admits a $W(k)$-lift $\mathcal{X}$ that is a $V_{22}$-scheme of $\Ga$-type over $W(k)$.
\end{cor}

\begin{proof}
Let $\mathcal{Y}$ over $W(k)$ be the split $V_5$-scheme defined by $(\ref{eqn:W5})$,
and $\mathcal{Z} \subset \mathcal{Y}$ a relative smooth rational quintic curve defined by (\ref{eqn:ZGageneral}) with $\xi =1$.
Let $\mathcal{X}$ be the $V_{22}$-scheme over $W(k)$ obtained by applying the two-ray game to $(\mathcal{Y}, \mathcal{Z})$ (Proposition \ref{prop:integraltwo-rayW5toV22}).
Then $\mathcal{X}$ satisfies the desired condition by Theorem \ref{thm:Gaclassification}.
\end{proof}

By the Galois cohomology argument, $V_{22}$-varieties of $\Ga$-type over a perfect field $k$ can be classified by $H^1(G_k, \mu_{4})$.
Actually, a more specific classification can be made as follows.

\begin{thm}
\label{thm:GaV22overk}
Let $k$ be a perfect field of characteristic $p \neq 2,5$,
and $Y = Y_k \subset \P^6_k$ the split $V_5$-variety over $k$.
For any $\xi \in k^{\times}$,
let $Z_{\xi}^{(\mathrm{a})} \subset Y$ be the $\Ga$-type smooth rational quintic curve defined by (\ref{eqn:ZGageneral}).
Let $(X_{\xi}^{(\mathrm{a})}, L_{\xi})$ be the $V_{22}$-variety with the line obtained by applying the two-ray game construction (\ref{prop:two-rayW5toV22}) to $(Y,Z_{\xi}^{(\mathrm{a})})$.
Then we have a one-to-one correspondence
\begin{eqnarray*}
&k^{\times}/k^{\times 4}
\simeq 
\{ 
\textup{$V_{22}$-variety of $\Ga$-type over $k$}
\}/k\textup{-isom}, \\
& \overline{\xi} \mapsto X_{\xi}^{(\mathrm{a})}.
\end{eqnarray*}

In particular, any $V_{22}$-variety of $\Ga$-type over $k$ produces a split $V_5$-variety over $k$ via the two-ray game.
\end{thm}

\begin{proof}
By Theorem \ref{thm:Gaclassification}, $\Sigma (X)_{\red, \sing}$ consists of a unique $k$-rational point.
By the two-ray game construction, we have an isomorphism of $k$-algebraic groups
\[
\Aut_{X_{1}^{(\mathrm{a})}/k, \red} \simeq \Aut_{(X_1^{(\mathrm{a})},L_1)/k, \red} \simeq \Aut_{(Y,Z_{1}^{(\mathrm{a})})/k, \red} \simeq \G_{a} \rtimes \mu_4,
\]
where
the third isomorphism follows from Theorem \ref{thm:BGaquintic>2} (the semi-direct product is taken with respect to (\ref{eqn:semi-directGa})).
Note that, since $\Ga$ has no non-trivial $k$-form, we have
\[
H^1 (G_k , (\Ga \rtimes \mu_4) (\overline{k}))
\simeq
H^1 (G_k , \mu_4 (\overline{k})) \simeq k^{\times}/ k^{\times 4}.
\]
Moreover, by the computation using the isomorphism $(Y_{\overline{k}},Z_{1, \overline{k}}^{(\mathrm{a})}) \simeq (Y_{\overline{k}}, Z_{\xi, \overline{k}}^{(\mathrm{a})})$ given by the action of 
$\begin{pmatrix}
    \xi^{1/4} & 0 \\
    0&  1 
\end{pmatrix}$,
the class of $H^1 (\Gal (\overline{k}/k), \mu_{4} (\overline{k}))$ corresponding to $(Y, Z_{\xi}^{(\mathrm{a})})$ is given by
\[
\Gal (\overline{k}/k) \ni \sigma \mapsto \sigma (\xi^{1/4})  / \xi^{1/4}  \in \mu_4 (\overline{k}),
\]
which corresponds to $\overline{\xi} \in k^{\times} / k^{\times 4}$.
This completes the proof.
\end{proof}

\subsection{$V_{22}$-varieties of $\Gm$-type}
\label{subsection:Gm-V22}

\begin{lem}
\label{lem:GmSigma}
Let $k$ be an algebraically closed field,
$Y$ a $V_5$-variety over $k$, and $Z \subset Y$ a smooth rational quintic curve of $\Gm$-type.
Let $(X, L)$ be the $V_{22}$-scheme with the line, that is obtained by applying the two-ray game (Proposition \ref{prop:two-rayW5toV22}) to $(Y,Z)$.
Then the following hold.
\begin{enumerate}
    \item 
$\Sigma (X)_{\red}$ is a union of two smooth rational curves that intersect at two points $L$ and $L'$ with multiplicities $2$.
Moreover, $\Gm$ acts on $X$ and $\Gm$-stable lines on $X$ are equal to $L$ or $L'$.
\item
$L$ and $L'$ do not intersect each other.
\item 
There exists an automorphism $\iota \in \Aut(X)$ of order $2$ such that the $\iota$-conjugation on $\Gm$ is the inversion and $\iota (L) = L'$.
\end{enumerate}
\end{lem}

\begin{proof}
First, we show (1). By Propositions \ref{prop:SigmaZp>2}, \ref{prop:SigmaZp=2}, and \ref{prop:tworaySigma}, we obtain the desired description on $\Sigma (X)_{\red} \setminus \{L\}$.
Note that, by construction of $X$ and Theorems \ref{thm:BGaquintic>2} and \ref{thm:BGaquintic2}, we have
\[
\mathbb{G}_{m,k} \simeq \Aut_{(Y,Z)/k, \red} \simeq \Aut_{(X,L)/k, \red} \hookrightarrow \Aut_{X/k,\red}.
\]
We take a singular point $L'$ of $\Sigma (X)_{\red}\setminus \{L\}$.
Let $(Y',Z')$ be a $V_5$-variety over $k$ with a smooth rational quintic curve $Z'$, that is obtained by the two-ray game (Proposition~\ref{prop:two-rayV22toW5}).
Since $L'$ is $\mathbb{G}_{m,k}$-invariant, $Z' \subset Y'$ is a smooth quintic curve of $\Gm$-type.
Therefore, by applying Propositions \ref{prop:SigmaZp>2}, \ref{prop:SigmaZp=2}, and \ref{prop:tworaySigma} again, we obtain the desired description of $\Sigma (X)_{\red}$.

(2) follows from Proposition \ref{prop:tworayfloppinglocus}.

We shall show (3).
When $\chara k=0$, this follows from \cite[Lemma 4.1]{Kuznetsov-ProkhorovGm}.
Basically, we follow the strategy of \cite{Kuznetsov-ProkhorovGm}.
We may assume that $Y\subset \P^6$ is defined by (\ref{eqn:W5}) and $Z$ is defined by (\ref{eqn:ZGm}).
By applying the two-ray game (Proposition \ref{prop:two-rayW5toV22}) to $({Y}, {Z})$, we obtain $({X}, {L})$, where ${X}$ is a $V_{22}$-variety and ${L} \subset {X}$ is a line.
Since $\Gm$ acts on $({Y}, {Z})$, we have a morphism $\Gm \rightarrow \Aut_{{X}/k}$.

We first prove that there is a  $\Gm$-invariant smooth conic $C$ on $X$.
By the definition of $Z$ (\ref{eqn:ZGm}), there are only two $\Gm$-fixed points $(1:0:0:0:0:0:0)$ and $(0:0:0:0:0:1:0)$ on $Z$.
Further, the exceptional locus of $\Phi$ is the strict transform of the line $l_{(1:0:0)}=\{(s:t:0:0:0:0:0)\}$ by Propositions~\ref{prop:tworayfloppinglocus}, \ref{prop:SigmaZp>2} and \ref{prop:SigmaZp=2}.
Note that $l_{(1:0:0)}$ does not contain $(0:0:0:0:0:1:0)$.
Thus the exceptional curve over the point $(0:0:0:0:0:1:0)$ defines a $\Gm$-invariant smooth conic $C$ on $X$.

By applying the two-ray game (Proposition \ref{prop:two-rayV22toQ3}) to $(X,C)$, we obtain $(Q,\Gamma)$, where $Q \subset \P^4$ is a quadric threefold over $k$ and $\Gamma \subset Q$ is a quadratically normal smooth rational sextic over $k$.
Moreover, we have $\Gm \hookrightarrow \Aut_{(Q,\Gamma)/k}$.
By the same argument as in \cite[Lemma 3.5 and Lemma 3.6]{Kuznetsov-ProkhorovGm}, we may assume the following:
\begin{itemize}
    \item
    The curve $\Gamma \subset \P^4$ is the image of the map $\P^1 \to \P^4$ defined by
\[        (s:t) \mapsto (s^6:s^5t:s^3t^3:st^5:t^6).
\]
    \item
    The quadric $Q \subset \P^4_{x_0,x_1,x_2,x_3,x_4}$ is defined by the equation $v Q_0 + Q_1$ for some $v \in k \setminus \{0,1\}$.
    Here, $Q_0 := x_0 x_4-x_2^2$ and $Q_1:= x_2^2 - x_1x_3$.
    \item 
    $\Gm$ acts on $\P^4$ by $t \cdot (x_0:x_1:x_2:x_3:x_4) = (x_0:tx_1:t^3x_2:t^5x_3:t^6x_6)$.
\end{itemize}
Note that we have an involution 
\begin{equation}
\label{eqn:involutionP4}
\iota \colon \P^4 \rightarrow \P^4, (x_0:x_1:x_2:x_3:x_4) \mapsto (x_4:x_3:x_2:x_1:x_0)
\end{equation}
which stabilizes $Q$ and $\Gamma$.
We also note that $\iota\circ (t\cdot) \circ \iota  = (t^{-1} \cdot)$.
Therefore, we obtain the corresponding isomorphism 
$\iota \in \Aut_{(X,C)/k} (k)$,
whose conjugation on $\Gm \subset \Aut_{(X,C)/k}$ is the inversion.
Since we have 
$\Aut_{(X,L)/k} \simeq \Gm$ by Theorems \ref{thm:BGaquintic>2} and \ref{thm:BGaquintic2},
$\iota$ does not preserve $L$, and hence we have $\iota (L) = L'$ by (1).
This completes the proof.
\end{proof}

\begin{thm}
\label{thm:Gmclassification}
Let $k$ be an algebraically closed field of characteristic $p\geq 0$.
Let $X$ be a $V_{22}$-variety of $\Gm$-type over $k$.
Then $X$ is isomorphic to $X_{u}$ for one 
\begin{equation}
\label{eqn:P1-4points}
u = (u:1) \in \P^1 (k) \setminus \{(0:1), (1:1), (5:4), (1:0)\}.
\end{equation}
Here, $X_u$ is a $V_{22}$-variety obtained by applying the two-ray game (Proposition \ref{prop:two-rayW5toV22}) to $(Y,Z_u)$, where $Z_u$ is a smooth rational quintic curve defined by (\ref{eqn:ZGm}).

Conversely, 
for any $u$ as in (\ref{eqn:P1-4points}),
$X_u$ is a $V_{22}$-variety of $\Gm$-type, so $\Sigma (X_u)_{\red}$ is as in Lemma \ref{lem:GmSigma}.
Furthermore, we have
\[
\Aut_{X_u/k, \red} \simeq\Gm \rtimes \Z/2\Z,
\]
where the semi-direct product is taken with respect to the inversion on $\Gm$, and the action of $\{1\} \rtimes \Z/2\Z$ permutes the singular points of $\Sigma (X_u)_{\red}$.

Moreover, for any elements $u \neq u'$ in the right-hand side of (\ref{eqn:P1-4points}), $X_u$ and $X_{u'}$ are not isomorphic.
Therefore, we obtain a bijection
\[
\{\textup{$V_{22}$-variety of $\Gm$-type over $k$} \}/k\textup{-isom}
\simeq \P^1(k) \setminus \{(0:1),(1:1), (5:4), (1:0)\}.
\]
\end{thm}

\begin{proof}
Let $X$ be a $V_{22}$-variety of $\Gm$-type over $k$.
By the same argument as in Theorem \ref{thm:Gaclassification}, there exists a line $L$ on $X$ that is $\Gm$-stable.
By applying the two-ray game (Proposition \ref{prop:two-rayV22toW5}) to $(X,L)$, we obtain a $V_5$-variety $Y$ with a smooth rational quintic curve $Z$.
By the same argument as in Theorem \ref{thm:Gaclassification} again, up to the action of $\PGL_{2,k}(k)$, $Z$ is equal to $Z_u$ as in (\ref{eqn:ZGm}) for some
$u$ as in (\ref{eqn:P1-4points}).
Thus, $X$ is isomorphic to $X_u$.

Next, we shall show that $X_u$ is a $V_{22}$-variety of $\Gm$-type over $k$.
Let $L_u$ be the line on $X_u$ such that $(X_u, L_u)$ corresponds to $(Y,Z_u)$ via the two-ray game.
Since $L_u$ is
an isolated singular point of $\Sigma(X)$ by Lemma \ref{lem:GmSigma}, we obtain
\[
\Aut_{X_u/k, \red}^{\circ} \hookrightarrow \Aut_{(X_u, L_u)/k, \red} \simeq \Aut_{(Y,Z_u)/k, \red} \simeq \Gm, 
\]
where the left arrow is an isomorphism onto the component.
Therefore, $X_u$ is a $V_{22}$-variety of $\Gm$-type.

Next, we specify the algebraic group $\Aut_{X_u/k, \red}$.
By Lemma \ref{lem:GmSigma}, there exists an involution $\iota \in \Aut_{X_u/k} (k)$ which permutes two singular points in $\Sigma (X_u)_{\red}$.
That means the right arrow of the exact sequence
\[
0 \rightarrow \Aut_{(X_u,L_u)/k, \red} \rightarrow \Aut_{X_u/k, \red} \rightarrow \Z/2 \Z 
\]
induced from the action on $\Sigma (X)_{\red}$ (see Proposition \ref{lem:GmSigma})
is a split surjection, so we obtain the desired isomorphism.

Next, we show the final statement.
Let $u \neq u'$ be as in the statement.
Let $(X_u, L_u)$ and $(X_{u'}, L_{u'})$ be as above.
Suppose, for contradiction, that there exists an isomorphism $f \colon X_u \simeq X_{u'}$.
Let $\iota_{u} \in \Aut(X_u)$ be an involution given by Lemma \ref{lem:GmSigma} (3).
By replacing $f$ with $f\circ \iota_u$ if necessary, we may assume that $f^{\ast} (L_{u'}) \simeq L_u$.
Since $(X_u, L_u) \simeq (X_{u'}, L_{u'})$, we have an isomorphism $ (Y,Z_{u}) \simeq (Y, Z_{u'}).$
By Proposition~\ref{prop:W5automorphisms}, $Z_u$ is mapped to $Z_u'$ by an element of 
$\PGL_2(k)$ when $\chara k \neq 2$, and by an element of $\SL_2(k)$ when $\chara k = 2$.
It contradicts Theorems \ref{thm:BGaquintic>2} and \ref{thm:BGaquintic2}.
\end{proof}

\begin{cor}
\label{cor:liftGm}
Let $k$ be an algebraically closed field of positive characteristic and $X$ a $V_{22}$-variety of $\Gm$-type over $k$.
Then $X$ admits a $W(k)$-lift $\mathcal{X}$ that is a $V_{22}$-scheme of $\Gm$-type over $W(k)$.
\end{cor}

\begin{proof}
By Theorem \ref{thm:Gmclassification}, $X$ is obtained by applying the two-ray game (Proposition \ref{prop:two-rayW5toV22}) to $(Y,Z_u)$, where $Y \subset \P^6$ is a $V_5$-variety given by (\ref{eqn:W5}) over $k$ and $Z_u \subset Y$ is a smooth rational quintic curve of $\Gm$-type defined by (\ref{eqn:ZGm})
for some 
\[
u \in \P^1 (k) \setminus \{ (0:1), (1:1), (5:4), (1:0)\}.
\]
We take $\widetilde{u} \in W(k)$ that is a lift of $u\in k$.
Note that we have $u, u-1, u-\frac{5}{4} \in W(k)^{\times}$ when $\chara k \neq 2$, 
and $u, u-1 \in W(k)^{\times}$ when $\chara k = 2$.
Let $\mathcal{Y}$ be the split $V_5$-scheme over $W(k)$ defined by $(\ref{eqn:W5})$, which is a lift of $Y$, and $\mathcal{Z}_{\widetilde{u}} \subset \mathcal{Y}$ a relative smooth rational quintic curve defined by (\ref{eqn:ZuGmdvr}), which is a lift of $Z_u$.
Let $\mathcal{X}$ be the $V_{22}$-scheme over $W(k)$ obtained by applying the two-ray game to $(\mathcal{Y}, \mathcal{Z}_{\widetilde{u}})$ (Proposition \ref{prop:integraltwo-rayW5toV22}).
Then $\mathcal{X}$ satisfies the desired condition by Theorem \ref{thm:Gmclassification}.
\end{proof}

\begin{cor}
\label{cor:noV22Autdim2}
Let $k$ be an algebraically closed field of characteristic $p$, and $X$ a $V_{22}$-variety over $k$.
Then we have $\dim \Aut_{X/k} \neq 2$.
\end{cor}

\begin{proof}
Suppose, for contradiction, that $\dim \Aut_{X/k} =2$.
Let $G \subset (\Aut_{X/k}^{\circ})_{\red}$ be a maximal solvable group.
By the assumption, we have $\dim G \geq 1$.
By the fixed-point theorem (\cite[Theorem VIII.21.2]{Humphreys}), there exists a line $L$ on $X$ such that $L$ is $G$-stable.
By applying the two-ray game (Proposition \ref{prop:two-rayV22toW5}) to $(X,L)$, we obtain a $V_5$-variety $Y$ over $k$ with a smooth rational quintic curve $Z$.
Since we have an embedding
$G \hookrightarrow \Aut_{(Y,Z)/k}$, 
$Z$ is Mukai--Umemura type, $\Ga$-type, or $\Gm$-type (see Definition \ref{defn:typeofquintic}).
By Theorems \ref{thm:MukaiUmemuraclassification}, \ref{thm:Gaclassification}, and \ref{thm:Gmclassification},
we have $\dim \Aut_{X/k} \in \{1,3\},$ and this leads to a contradiction.
\end{proof}

{\cred
Next, we shall classify $V_{22}$-varieties of $\Gm$-type over a general field $k$.

\begin{defn}
\label{defn:splitGm}
Let $k$ be a field.
For any 
\[
u = (u:1) \in \P^1 (k) \setminus \{(0:1), (1:1), (5:4), (1:0)\},
\]
let $X_u$ be the $V_{22}$-variety over $k$ obtained by applying the two-ray game (Proposition \ref{prop:two-rayW5toV22}) to $(Y,Z_u)$, where $Y \subset \P^6$ is the split $V_5$-variety over $k$ defined by (\ref{eqn:W5}) and $Z_u$ is a smooth rational quintic curve over $k$ defined by (\ref{eqn:ZGm}).
We say a $V_{22}$-variety of $\Gm$-type $X$ over $k$ is split if $X$ is isomorphic to $X_{u}$ for some $u$ as above.
Note that $u$ is uniquely determined for $X$ by Theorem \ref{thm:Gmclassification}.
\end{defn}

\begin{lem}
\label{lem:Gmsplitformexists}
Let $k$ be a perfect field, and $X$ a $V_{22}$-variety of $\Gm$-type over $k$.
Then there exists a split $V_{22}$-variety of $\Gm$-type $X'$ over $k$ such that 
$X_{\overline{k}} \simeq X'_{\overline{k}}$.
\end{lem}

\begin{proof}
By Theorem \ref{thm:Gmclassification}, there exists a finite Galois extension $l/k$ such that $X_l$ is a split $V_{22}$-variety of $\Gm$-type over $l$, i.e.,\ there exists $u \in \P^1 (l) \setminus \{(0:1), (1:1), (1:0), (5:4)\}$ such that $X_{l}$ is isomorphic to the split $V_{22}$-variety of $\Gm$-type $X_u$ over $l$. 
Since $X_{u}$ has a $k$-form $X$, for any $\sigma \in \Gal (l/k)$, we have an isomorphism $\sigma (X_u) \simeq X_u$.
On the other hand, by pulling back the two-ray game diagram connecting $(X_u, L_u)$ and $(Y,Z_u)$ via $\sigma^{\ast} \colon \Spec k \rightarrow \Spec k$, we have $\sigma (X_u) \simeq X_{\sigma (u)}$.
Here, $L_u$ is the line on $X_u$ such that $(X_u, L_u)$ corresponding to $(Y,Z_u)$ via the two-ray game.
By Theorem \ref{thm:Gmclassification}, $u= \sigma (u)$ for any $\sigma$, hence we have $u \in K$.
Therefore, $X_u$ descends to a split $V_{22}$-variety of $\Gm$-type $X_{u}^{(k)}$ over $k$ as desired.
\end{proof}

\begin{thm}
\label{thm:Gmclassificationgeneral}
Let $k$ be a perfect field.
We set 
\begin{align*}
 &A_k :=    \{
    X\colon \Gm\textup{-}V_{22}\textup{-variety over }k
    \}/k\textup{-isom}.\\
&A_{k,u} := 
\{
    X\colon \Gm\textup{-}V_{22}\textup{-variety over }k \mid
    X_{\overline{k}} \simeq X_{u,\overline{k}} 
\}/k\textup{-isom}.
\end{align*}
Here, $X_u$ is the split $V_{22}$-variety of $\Gm$-type over $k$ defined in Definition \ref{defn:splitGm}.
Then the following hold.
\begin{enumerate}
\item 
We have a decomposition 
\[
A_k = \bigsqcup_{u=(u: 1) \in \P^1 (k)\setminus \{(0:1), (1:1),(1:0), (5:4)\}}
A_{k,u}.
\]
\end{enumerate}
In the following, we fix $u \in k \subset \P^1(k)$ as above.
\begin{enumerate}[resume]
    \item 
    There exists a bijection $\alpha_u \colon
    A_{k.u}
    \simeq 
    H^1 (G_k, \Gm (\overline{k}) \rtimes \Z/2\Z)$.
    \item 
    There exists a surjection 
\[
\Phi_u \colon  A_{k.u}
    \simeq  H^1 (G_k, \Gm (\overline{k}) \rtimes \Z/2\Z) \rightarrow H^1 (G_k, \Z/2\Z) \simeq \{ \Z/2\Z\textup{-torsor over }k\}/k\textup{-isom}
\]
    such that $\Phi_u (X)$ corresponds to the $\Z/2\Z$-torsor $\Sigma(X)_{\red, \sing}$.
    \item 
    We have 
    \[
    \{ \Z/2\Z\textup{-torsor over }k\}/k\textup{-isom} = \{ (\Z/2\Z)_k\} \cup \{ 
    \Spec l \mid [l \colon k]=2
    \},
    \]
    where $(\Z/2\Z)_k$ is the constant group scheme over $k$, which is a trivial torsor, and $l$ varies over all isomorphism classes of quadratic extensions over $k$.
    
    Further, we have $\Phi_u^{-1} ((\Z/2\Z)_k) = \{X_u \}$
    and there exists a bijection
    \[
    \Psi_{u,l} \colon
    \Phi_u^{-1} (\Spec l) \simeq H^1 (G_k, U_{l} (1) (\overline{k})) \simeq k^{\times}/N_{l/k}(l^{\times}),
    \]
    where $U_l (1)$ is the Unitary group of rank 1 associated with $l/k$.
    \item 
    For any $X \in A_{k}$, $\Aut_{X/k, \red}$ is equal to $\Gm \rtimes \Z/2\Z $ or $U_{l} (1) \rtimes \Z/2\Z$
    for some quadratic extension $l/k$.
    Moreover, the following are equivalent.
    \begin{enumerate}
        \item 
        $X$ is split.
        \item
        $\Sigma(X)_{\red, \sing}$ consists of two $k$-rational points.
        \item 
        $\Aut_{X/k, \red} \simeq \Gm \rtimes \Z/2\Z$.
    \end{enumerate}
\end{enumerate}
\end{thm}

In summary, we obtain a bijection
\[
A_k \simeq \bigsqcup_{u\in \P^1(k)\setminus \{(0:1), (1:1), (1:0), (5:4)\}} \left(
\{X_u\} \sqcup \bigsqcup_{[l:k]=2} ( k^{\times}/N_{l/k} (l^{\times}) )
\right)
\]

\begin{proof}
(1) follows from Theorem \ref{thm:Gmclassification} and Lemma \ref{lem:Gmsplitformexists}.

To show (2), it suffices to show that $\Aut_{X_u/k, \red}$ is isomorphic to $\Gm\rtimes \Z/2\Z.$
Since $Z_u$ is of $\Gm$-type, by the construction of $X_u$, we obtain an immersion $\Gm \hookrightarrow \Aut_{X_u/k}$.
By Theorem \ref{thm:Gmclassification}, we have $\Aut_{X_u/k, \red}^{\circ} \simeq \Gm$.
Moreover, $\Aut_{X_u/k,\red} \setminus \Aut_{X_u/k, \red}^{\circ}$ is a $\Gm$-torsor, which is split by Hilbert 90.
Therefore, we obtain $\Aut_{X/k, \red} \simeq \Gm \rtimes \Z/2\Z$ as desired.

We show (3).
Since there exists a section $\Z/2\Z \hookrightarrow \Gm(k) \rtimes \Z/2\Z$, the natural projection $\beta \colon H^1 (G_k, \Gm (\overline{k}) \rtimes \Z/2\Z) \rightarrow H^1 (G_k, \Z/2\Z)$ is surjective.
Let $\gamma  \colon H^1 (G_k, \Z/2\Z) \simeq \{ \Z/2\Z\textup{-torsor over }k\}$ be the usual bijection, and put $\Phi_u := \gamma \circ \beta \circ \alpha_u$.
Since the projection $\Gm \rtimes \Z/2\Z \rightarrow \Z/2\Z$ is isomorphic to
\[
\Aut_{X_u/k,\red} \rightarrow \Aut_{\Sigma(X_u)_{\red,\sing}/k} \simeq \Z/2\Z
\]
by Theorem \ref{thm:Gmclassification} and
$\Sigma (X_u)_{\red}$ is a trivial torsor over $\Z/2\Z$ by the construction,
$\Phi_u$ satisfies the desired property.

Next, we show (4).
The first statement is clear.
We shall compute fibers of $\Phi_u$.
By Hilbert 90 and the exact sequence of pointed sets
\[
H^1 (G_k, \Gm(\overline{k})) \rightarrow H^1 (G_k, \Gm (\overline{k}) \rtimes \Z/2\Z) \rightarrow H^1 (G_k, \Z/2\Z),
\]
we have $\Phi_u^{-1} ((\Z/2\Z)_k) = \{X_u\}$.
Next, we fix a quadratic extension $l/k$.
We take $X \in \Phi_u^{-1} (\Spec l)$ as the variety corresponding to $s(\gamma^{-1} (\Spec l))$, where 
\[
s \colon H^1 (G_k, \Z/2\Z) \rightarrow H^1 (G_k, \Gm \rtimes \Z/2\Z)
\] 
is induced by $1 \rtimes \mathrm{id}$.
Let 
\[
\alpha_X = \overline{c_X} \in H^1(G_k, \Gm (\overline{k}) \rtimes \Z/2\Z)
\]
be the class associated with $X$ (where $c_X \colon G_k \rightarrow \Gm (\overline{k}) \rtimes \Z/2\Z$ is a representative of $\alpha_X$).
Let $(\Gm (\overline{k}) \rtimes \Z/2\Z)^{(c_X)}$ be the inner twists of $\Gm (\overline{k}) \rtimes \Z/2\Z$ with respect to $c_X$, i.e.,\ $\Gm (\overline{k}) \rtimes \Z/2\Z$ with the $G_k$-action $\cdot_{c_X}$ given by 
\[
\sigma \cdot_{c_X} (g) = c_X(\sigma) \sigma (g) c_X (\sigma)^{-1}.
\]
Note that we have a bijection
\begin{align*}
H^1 (G_k, \Gm (\overline{k}) \rtimes \Z/2\Z)&\simeq H^1 (G_k, (\Gm (\overline{k}) \rtimes \Z/2\Z)^{(c_X)}), \\
&c \mapsto c \cdot c_{X}^{-1}.
\end{align*}

By the choice of $X$, we may take a representative $c_X$ such that $c_X$ factors as 
\[
c_X \colon G_k \rightarrow \Z/2\Z \xrightarrow{1 \times \mathrm{id}} \Gm \rtimes \Z/2\Z,
\]
where the kernel of the first map is $G_l \subset G_k$.
By the direct computation, we have
\[
(\Gm (\overline{k}) \rtimes \Z/2\Z)^{(c_X)} \simeq U_l(1) (\overline{k}) \rtimes \Z/2\Z, 
\]
where the $G_k$-action on the right-hand side is induced from the $k$-structure of $U_l (1)$.
Indeed, for $g = (t,a)$, where $t \in \Gm(\overline{k})$ and $a \in \Z/2\Z$, 
we have $\sigma \cdot_{c_X} (t,0) = (\sigma(t),0)$ when $a = 0$, 
and $\sigma \cdot_{c_X} (t,1) = (\sigma(t)^{-1},1)$ when $a = 1$.
This is no other than the action on the $U_l (1)(\overline{k}) \rtimes \Z/2\Z$.

Hence, we obtain the following diagram with horizontal exact sequences:
\[
  \begin{CD}
       H^1 (G_k,\Gm (\overline{k}))  @>{\iota}>>  H^1(G_k, \Gm(\overline{k}) \rtimes \Z/2\Z) @>{\beta}>>  H^1(G_k, \Z/2\Z)   \\
     @.    @V{\simeq}V{\cdot c_{X}^{-1}}V  @V{\simeq}V{\cdot c_{X}^{-1}}V   \\
       H^1(G_k,U_l(1) (\overline{k})) @>{\iota'}>>  H^1 (G_k, U_l(1)(\overline{k}) \rtimes \Z/2\Z) @>>>  H^1 (G_k, \Z/2\Z)
  \end{CD}
\]
Since the right vertical arrow sends $\gamma^{-1} (\Spec l)$ to $0$, we have
\[
\cdot c_{X}^{-1} ( \beta^{-1} (\Spec l)) = \iota' (H^1 (G_k, U_l(1)(\overline{k}))).
\]
We shall show that $\iota'$ is injective.
Since the morphism
\begin{equation}
\label{eqn:H0surj}
H^0 (G_k, U_l(1) (\overline{k}) \rtimes \Z/2\Z) \rightarrow H^0 (G_k, \Z/2\Z)
\end{equation}
is surjective, $\iota^{'-1} (1,0)$ consists of 1 element $1$.
We can show that, for any $\overline{d} \in H^1 (G_k, U_l(1) (\overline{k}))$,
the inner twists $(U_l(1) (\overline{k}) \rtimes \Z/2\Z)^{(d)}$ is isomorphic to $U_{l} (1) (\overline{k}) \rtimes \Z/2\Z$.
Indeed, for any $g = (t,a)$, where $t \in \Gm(\overline{k})$ and $a \in \Z/2\Z$, 
we have $\sigma \cdot_{d} (t,0) = (\sigma(t),0)$ when $a = 0$, 
and $\sigma \cdot_{d} (t,1) = (d(\sigma)^2 \sigma(t), 1)$ when $a = 1$.
Moreover, we have
\[
H^1 (G_k, U_l(1) (\overline{k})) \simeq k^{\times}/N_{l/k} (l^{\times}).
\]
Hence, both components of $(U_l(1) (\overline{k}) \rtimes \Z/2\Z)^{(d)}$ are trivial $U_l(1)$-torsors as desired.
Therefore, the multiplication $d^{-1}$ induces compatible bijections on each term of
\[
  \begin{CD}
       H^1 (G_k,U_l(1) (\overline{k}))  @>{\iota'}>>  H^1(G_k, U_l(1)(\overline{k}) \rtimes \Z/2\Z) @>{\beta}>>  H^1(G_k, \Z/2\Z),
  \end{CD}
\]
which is an exact sequence of pointed sets.
By the surjectivity of (\ref{eqn:H0surj}), $\iota'^{-1} (\iota'(\overline{d}))$ is a singleton $\{\overline{d}\}$.
Therefore, $\iota'$ is injective, and we obtain the desired bijection
\[
\cdot c_{X}^{-1} (\beta^{-1} (\Spec l) ) \simeq  H^1 (G_k, U_l (1)) \simeq k^{\times}/ N_{l/k} (l^{\times}).
\]

Finally, we show (5).
We use the same notation as in (4), i.e.,\ we take $X\in \Phi_{u}^{-1} (\Spec l)$.
Since $X$ depends on $u$ and $l$, we write it as $X_u^l$ from now on.
By definition, the inner twist 
\[
(\Gm (\overline{k}) \rtimes \Z/2\Z)^{(c_X)} \simeq U_l(1) (\overline{k}) \rtimes \Z/2\Z
\]
is isomorphic (as $G_k$-group) to $\Aut_{X_u^l/k, \red} (\overline{k})$ for any quadratic extension $l/k$.
Moreover, for $\overline{d} \in H^1 (G_k, U_l (1) (\overline{k}))$, the inner twist 
$(U_l(1) \rtimes \Z/2\Z)^{(d)} \simeq U_l(1) \rtimes \Z/2\Z$ is isomorphic (as $G_k$-module) to $\Aut_{X_u^{l,(d)}/k, \red} (\overline{k})$, where $X_u^{l,(d)}$ is the twisted form of $X_u^{l}$ with respect to $\iota'(\overline{d})$.
By the proof of (4), we have $A_{k,u} = \{(X_u, X_u^{l,(d)})\}_{l,d}$.
Hence, we obtain the first statement.
Moreover, we obtain the equivalence of (a) and (c).
The implication of (a) $\Rightarrow$ (b) follows from the construction of $X_u$ (see the proof of Lemma \ref{lem:Gmsplitformexists}).
The converse direction follows from (1)-(4).
This completes the proof.
\end{proof}
}
{\cblue
\begin{rem}
\label{rem:Gmclassificationcanonical}
\begin{enumerate}
\item
In the proof of Theorem \ref{thm:Gmclassificationgeneral},
the morphism $\beta \circ \alpha_u$ is canonically defined since it is equal to the composite
\[
A_{k,u} \simeq H^1 (G_k, \Aut_{X_u/k, \red} (\overline{k})) \rightarrow H^1 (G_k, \pi_0 (\Aut_{X_u/k, \red}) (\overline{k})).
\]
Moreover, $\gamma$ is the natural bijection.
Therefore, $\Phi_u$ is a canonical morphism.
On the other hand, a bijection
\[
\Psi_{u,l} \colon \Phi_u^{-1} (\Spec l) \simeq k^{\times}/N_{l/k}(l^{\times})
\]
is non-canonical.
However, the proof shows that we can define $\Psi_{u,l}$ by choosing a section $s$ of 
\[
\Aut_{X_u/k, \red} (k) \rightarrow \pi_0 (\Aut_{X_u/k,\red}) (k) \simeq \Z/2\Z.
\]
If a bijection comes from $s$, then $\Psi_{u,l}^{(-1)} (1)$ is the $V_{22}$-variety of $\Gm$-type corresponding to the cocycle
\[
\Gal (l/k) \simeq \Z/2\Z \xhookrightarrow{s} \Aut_{X_u/k, \red} (k) \hookrightarrow \Aut_{X_u/k, \red} (l).
\]
\item 
In general, $k^{\times}/N_{l/k} l^{\times}$ is an infinite group (e.g.,\ when $k$ is a number field).
In particular, for fixed $V_{22}$-verity of $\Gm$-type $X$ over $k$, and a fixed quadratic extension $l/k$, the set
$\{
X' \mid X'_{l} \simeq X_l
\}/\simeq_k$
is an infinite set in general.
\end{enumerate}
\end{rem}
}

\begin{cor}
\label{cor:Gmclassificationfinite}
Let $\F_q$ be a finite field.
Then we have a surjection
\[
\{\textup{$V_{22}$-variety of $\Gm$-type over $\F_q$} \}/\F_q\textup{-isom}
\rightarrow \P^1(\F_q) \setminus \{(0:1),(1:1), (5:4), (1:0)\}.
\]
Moreover, each fiber over $u= (u:1)$ consists of just two $V_{22}$-varieties of $\Gm$-type $X_u$ and $X_{u}'$, where $X_u$ is as in Definition \ref{defn:splitGm},
and $X_u'$ is a twisted form of $X_u$ with $\Sigma (X_u')_{\red,\sing} \simeq \Spec \F_{q^2}$. 
In particular, 
there is a $V_{22}$-variety of $\Gm$-type over $\F_{q}$ if and only if $q \geq4$.
\end{cor}

\begin{proof}
This follows from Theorem \ref{thm:Gmclassification}.
\end{proof}

\begin{cor}
\label{cor:V22overF2}
Let $X$ be a $V_{22}$-variety over $\F_2$.
Then $\Aut(X_{\overline{\F_2}})$ is a finite group.
\end{cor}

\begin{proof}
This follows from Remark \ref{rem:autdimgeq1}, Theorems \ref{thm:MukaiUmemuraclassification}, \ref{thm:Gaclassification}, and Corollary \ref{cor:Gmclassificationfinite}.
\end{proof}

\subsection{$V_{22}$-variety over $\Z$}

\begin{thm}
\label{thm:noV22withlargeautomoverZ}
There is no smooth projective scheme over $\Z$ whose generic fiber is a Mukai--Umemura variety, $V_{22}$-variety of $\Ga$-type, or a $V_{22}$-variety of $\Gm$-type over $\Q$.
\end{thm}

\begin{proof}
Suppose, for a contradiction, that there exists a smooth projective scheme $\mathcal{X}$ over $\Z$ such that $\mathcal{X}_{\Q}$ is a $V_{22}$-variety with $\dim \Aut_{\mathcal{X}_{\Q}/\Q} \geq 1$.
Then we have $\dim \Aut_{\mathcal{X}_{\F_2}/\F_2} \geq 1$.
This contradicts Corollary \ref{cor:V22overF2}. This completes the proof.
\end{proof}

\begin{thm}
\label{thm:V22overZ}
There exists a $V_{22}$-scheme over $\Z$.
\end{thm}

\begin{proof}
Let $Q \subset \P^4_{\Z}$ with coordinates $(x:y:z:w:u)$ be the closed subscheme defined by $xy+zw+u^2=0$.
Let $\Gamma$ be the scheme theoretic image of the morphism
\begin{eqnarray*}
&\P^1_{\Z} \rightarrow \P^4_{\Z}, \\
&(s:t) \mapsto (-t^6: s^6+s^5t+2s^3t^3+t^6:-s^4t^2:s^6+2s^4t^2+st^5:s^5t+s^3t^3+t^6).
\end{eqnarray*}
Then $\Gamma$ is contained in $Q$.
Moreover, we can show that $\Gamma_{s} \subset \P^4_{s}$ is a smooth rational quadratically normal sextic over $s$ for any point $s \in \Spec \Z$.
By applying the two-ray game (Proposition \ref{prop:integraltwo-rayQ3toV22}) to $(Q,\Gamma)$, we obtain the desired $V_{22}$-scheme.
\end{proof}

\section{Shafarevich conjecture for $V_{22}$-varieties with large automorphism group}
\label{section:Shafarevich}

To discuss the Shafarevich conjecture for $V_{22}$-varieties satisfying $\dim \Aut \geq 1$, we introduce terminology for the types of smooth reductions.

\begin{defn}
\label{defn:V22reduction}
    Let $R$ be a discrete valuation ring, $K$ the fraction field of $R$, and $k$ the residue field of $R$.
    Let $X$ be a $V_{22}$-variety over $K$.
    Suppose that there exists a $V_{22}$-scheme $\mathcal{X}$ over $R$ satisfying $\mathcal{X}_K \simeq X$. In this case, we say $X$ admits \emph{good reduction at $R$}.
    Moreover, if $\mathcal{X}$ is a Mukai--Umemura scheme over $R$, we say $X$ admits \emph{Mukai--Umemura good reduction at $R$}. 
    \emph{$\Ga$-good reduction} and \emph{$\Gm$-good reduction} are defined similarly.
    Moreover, when $X_{\Frac \widehat{R}}$ admits good reduction, we say $X$ admits good reduction at $\widehat{R}$. 
    We use the same terminology for Mukai--Umemura, $\Ga$, and $\Gm$-good reduction as well.
\end{defn}

\begin{rem}
\label{rem:grcomplete}
Let $R, K, k$ be as in Definition \ref{defn:V22reduction}.
Let $X$ be a $V_{22}$-scheme over $K$.
\begin{enumerate}
\item
Suppose that there exists a smooth projective scheme $\mathcal{X}$ over $R$ with $\mathcal{X}_K \simeq X$.
Then $\mathcal{X}$ is a $V_{22}$-scheme over $R$, and hence $X$ has good reduction at $R$, 
since $b_2 (\mathcal{X}_{k}) =1$ 
and $(-K_{\mathcal{X}_k/k})^3 = 22$.
\item 
Suppose that $R$ is excellent.
If $X$ admits good reduction at $\widehat{R}$, then $X$ admits good reduction at $R$ by \cite[Proposition 7.1]{V5} and Corollary \ref{cor:V22lift}.
The same holds for Mukai--Umemura, $\Ga$, and $\Gm$-good reduction.
\end{enumerate}
\end{rem}

\subsection{Shafarevich conjecture for Mukai--Umemura varieties}

\begin{prop}
\label{prop:goodreductioncriteriaMukaiUmemura}
Let $K$ be a $p$-adic field and $X$ a Mukai--Umemura variety over $K$. 
\begin{enumerate}
\item
Suppose that $p =2$ or $5$. 
Then $X$ does not admit good reduction at $\cO_K$.
\item 
Suppose that $p \neq 2,5$.
Then the following are equivalent.
\begin{enumerate}
    \item
    $X$ admits good reduction at $\cO_K$.
    \item 
    $X$ is a split Mukai--Umemura variety over $K$.
\end{enumerate}
\end{enumerate}
\end{prop}

\begin{proof}
(1) Suppose, for contradiction, that there exists a smooth projective scheme $\mathcal{X}$ over $\cO_K$ such that $\mathcal{X}_K \simeq X$.
Then we have $\dim \Aut_{\mathcal{X}_k/k} \geq 3$ and it contradicts Theorem \ref{thm:Gmclassificationgeneral}.
This completes the proof.

(2) We first show (b) $\Rightarrow$ (a).
Let $\mathcal{Y} \subset \P^6_{R}$ be the split $V_5$-scheme over $R$ defined by (\ref{eqn:W5}). 
Let $\mathcal{Z} \subset \mathcal{Y}$ be the relative smooth quintic curve $\mathcal{Z}_0^{(\mathrm{a})}$ defined by (\ref{eqn:ZGageneraldvr}).
Let $\mathcal{X}$ be the $V_{22}$-scheme over $\cO_K$ obtained by applying the two-ray game (Proposition \ref{prop:integraltwo-rayW5toV22}) to $(\mathcal{Y},\mathcal{Z})$.
By construction, $\mathcal{X}_K$ is a split Mukai--Umemura variety, so a split Mukai--Umemura variety admits good reduction.

Next, we show (a) $\Rightarrow$ (b).
Let $k$ denote the residue field of $K$.
Let $\mathcal{X}$ be a smooth projective scheme over $\cO_K$ with $\mathcal{X}_K \simeq X$.
Let $K'$ be a finite unramified extension of degree $2$ over $K$. 
Since $\Sigma (X/K)_{\red}$ is a curve of genus $0$ over $K$, $\Sigma (X_{K'}/K')_{\red}$ admits a rational point, i.e.,\ $X_{K'}$ is a split Mukai--Umemura variety (see Proposition \ref{prop:MUsplitcriteria}). 
We take a line $L$ on $X_{K'}$.
Let $\mathcal{L} \subset \mathcal{X}_{\cO_{K'}}$ be the Zariski closure of $L$ and $(\mathcal{Y}, \mathcal{Z})$ the $V_5$-scheme over $\cO_{K'}$ with the relative smooth rational quintic curve $\mathcal{Z} \subset \mathcal{Y}$ obtained by applying the two-ray game (Proposition \ref{prop:integraltwo-rayV22toW5}) to $(\mathcal{X}_{\cO_{K'}}, \mathcal{L})$.
Since $p \neq 2$, $\mathcal{Y}$ is a split $V_5$-scheme over $\cO_{K'}$ by 
\cite[Proof of Proposition 7.4]{V5}.
Let $\Sigma (\mathcal{X}_{\cO_{K'}}/\cO_{K'})$ and $\Sigma_{\mathcal{Z}} (\mathcal{Y}/\cO_K')$ be as in Definition \ref{defn:Sigma}.
Then by the same argument as in Proposition \ref{prop:tworaySigma}, we obtain an open immersion
\[
\Sigma (\mathcal{X}_{\cO_{K'}}/\cO_{K'})_{\red} \setminus \{\mathcal{L}\} \hookrightarrow \Sigma_{\mathcal{Z}} (\mathcal{Y}/\cO_{K'}), 
\]
where the $\cO_{K'}$-section of $\Sigma (\mathcal{X}_{\cO_{K'}}/\cO_{K'})$ corresponding to $\mathcal{L}$ will also be denoted by $\mathcal{L}$.

By \cite[Proposition 5.2]{V5}, we obtain an isomorphism 
$
\Sigma (\mathcal{Y}/\cO_{K'}) \simeq \P^2_{\cO_{K'}}.
$
By the composition 
\[
\iota \colon
\Sigma (\mathcal{X}_{\cO_{K'}}/\cO_{K'})_{\red} \setminus \{\mathcal{L}\} \hookrightarrow \Sigma_{\mathcal{Z}} (\mathcal{Y}/\cO_{K'}) \hookrightarrow \Sigma (\mathcal{Y}/\cO_{K'}) \simeq \P^2_{\cO_{K'}}
\]
we regard $\Sigma (\mathcal{X}_{\cO_{K'}}/\cO_{K'})_{\red} \setminus \{\mathcal{L}\}$ as the locally closed subscheme of $\P^2_{\cO_{K'}}$.
Let $\Sigma '$ be the scheme theoretic image of $\iota$, which is contained in $\Sigma_{\mathcal{Z}} (\mathcal{Y}/\cO_{K'})\subset \P^2_{\cO_{K'}}$.
Since $\Sigma'$ is an integral scheme over $\cO_{K'}$ with dense generic fiber, it follows that $\Sigma'$ is flat over $\cO_{K'}$.
By Propositions \ref{prop:SigmaZp>2} and \ref{prop:tworaySigma} (cf. Lemma \ref{lem:MUSigma} and Theorem \ref{thm:MukaiUmemuraclassification}), 
$
\Sigma'_{K'}   \subset \P^2_{K'}
$ is a smooth conic and $\Sigma'_{k'} \subset \P^2_{k'}$ contains a smooth conic as a closed subscheme.
Note that $\Sigma'_{k'} \subset \P^2_{k'}$ is a conic since so is $\Sigma'_{K'}$ and $\Sigma'$ is flat over $\cO_{K'}$.
Therefore, $\Sigma'_{k'} \subset \P^2_{k'}$ equals the smooth conic, and $\Sigma'$ is a relative smooth conic over $\cO_{K'}$.
In particular, $\Sigma (\mathcal{X}_{\cO_{K'}}/ \cO_{K'})_{\red} \setminus \{\mathcal{L}\}$ is smooth over $\cO_{K'}$.
Let $U \subset \Sigma (\mathcal{X}/\cO_{K})_{\red}$ be the complement of the image of $\{\mathcal{L}\}$ under the natural finite morphism
\[
\pi\colon
\Sigma (\mathcal{X}_{\cO_{K'}}/ \cO_{K'})_{\red} \rightarrow \Sigma (\mathcal{X/\cO_{K}})_{\red};
\]
then $U \subset \Sigma (\mathcal{X}/\cO_{K})_{\red}$ is open.
Since $U_{\cO_{K'}}$ is smooth over $\cO_{K'}$, $U$ is smooth over $\cO_{K}$.
Since $\pi ({\mathcal{L}})$ is a finite scheme over $\cO_K$ of at most degree 2,
$U_{k} \hookrightarrow \Sigma (\mathcal{X}_k /k)_{\red}$ is a non-empty open subscheme of smooth genus 0 curve whose complement is
a reduced effective divisor of degree $\leq 2$.
Therefore, we may take a $k$-point on $U_k$, which lifts to a $\cO_{K}$-point on $U$ by Hensel's lemma.
\footnote{By applying the two-ray game using this lift and repeating the above argument, we can show that $\Sigma (\mathcal{X}/\cO_{K})_{\red}$ is actually a smooth scheme over $\cO_K$.}
In summary, we obtain a relative line on $\mathcal{X}$.
By Proposition \ref{prop:MUsplitcriteria}, $\mathcal{X}_K$ is split.
This completes the proof.
\end{proof}

\begin{thm}
\label{thm:Mukai-Umemura_Shafarevich}
Let $F$ be a number field and $S$ a finite set of finite places of $F$.
We set
\[
    \Shaf_{\PGL_2,F,S}
    := \left\{
    X \left|
    \begin{array}{l}
      \textup{$X \colon$ Mukai--Umemura variety over $F$,} \\
      \textup{$X$ admits good reduction at $\cO_{F,\p}$} \\
        \textup{for any  $\p \notin S$}
    \end{array}
    \right\}\right.
    /F\textup{-isom.}
\]
Then we have
\[
\# \Shaf_{\PGL_{2},F,S} =
\begin{cases}
0 & \textup{if } 10 \notin \cO_{F,S}^{\times}, \\
2^{\#S+r-1}  & \textup{if } 10 \in \cO_{F,S}^{\times},
\end{cases}
\]
where $r$ is the number of real places of $F$.
\end{thm}

\begin{proof}
If $10 \notin \cO_{F,S}^{\times}$, then there exists a finite place $\p$ of $\cO_{F,S}$ with $\p | (2)$ or $\p | (5)$.
Therefore, we have $\Shaf_{\PGL_{2},F,S} = \emptyset$ by Proposition \ref{prop:goodreductioncriteriaMukaiUmemura} (1).

We suppose that $10 \in \cO_{F,S}^{\times}$.
Let $X \in \Shaf_{\PGL_{2},F,S}$ and 
\[
\alpha_{X} \in H^1 (G_F, \PGL_2 (\overline{F})) \simeq \Br (F)[2]
\]
be the class corresponding to $X$ under the bijection in Theorem \ref{thm:Mukai-Umemurageneral}.
Recall that we have the exact sequence
\begin{equation}
0 \rightarrow \Br (F) \rightarrow  \prod_v \Br (F_{v}) \rightarrow \Q/\Z \rightarrow 0,
\end{equation}
where $v$ runs over all the places of $F$.
We denote the image of $\alpha_{X}$ in the middle term by $(\alpha_{X,v})_v$.
By Theorem \ref{thm:Mukai-Umemurageneral} and Proposition \ref{prop:goodreductioncriteriaMukaiUmemura},
for a non-Archimedean place $v$ with $v \nmid (2), (5)$, $\alpha_
{X,v} =0$ if and only if $X$ admits good reduction at $\cO_{F,v}$.
Since $v \notin S \Rightarrow v \nmid (2), (5)$, we obtain a bijection
\begin{equation*}
\begin{aligned}
\Shaf_{\PGL_{2},F,S} &\rightarrow
\left.\left
\{  
(\alpha_v)_v \in
\prod_{v \in S \cup T} (\frac{1}{2}\Z/\Z)
\right| \sum \alpha_v  =0
\right\},\\
&X \mapsto (\alpha_{X,v})_v,
\end{aligned}
\end{equation*}
where $T$ is the set of real places of $F$.
Since the cardinality of the right-hand side is equal to $2^{\#S+r-1}$, this completes the proof.
\end{proof}

\subsection{Shafarevich conjecture for $V_{22}$-varieties of $\Ga$-type}

The following proposition shows that the Shafarevich conjecture for $V_{22}$-varieties of $\Ga$-type fails.

\begin{prop}
\label{prop:GaV22shaffails}
Let $F$ be a number field
and $X$ a $V_{22}$-variety of $\Ga$-type over $F$.
Let $S$ be a finite set of finite places of $F$ such that $10 \in \cO_{F,S}^{\times}$.
Then there exists a $V_{22}$-scheme $\mathcal{X}$ over $\cO_{F,S}$ such that $\mathcal{X}_{F} \simeq X$.
In particular, there are infinitely many $F$-isomorphism classes of $V_{22}$-varieties of $\Ga$-type over $F$ that extend to $\cO_{F}[1/10]$.
\end{prop}

\begin{proof}
By Theorem \ref{thm:GaV22overk}, $X$ is isomorphic to $X_{\xi}^{(\mathrm{a})}$ for some $\overline{\xi} \in F^{\times}/ F^{\times 4}$.
Moreover, by replacing representative if necessary, we may assume that $\xi \in \cO_{F,S}$.
Let 
$
\mathcal{Y} := Y_{\cO_{F,S}} \subset \P^6_{\cO_{F,S}}
$
be the split $V_5$-scheme over $\cO_{F,S}$ defined by (\ref{eqn:W5}).
We denote the generic fiber $\mathcal{Y}_{F}$ by $Y$.
Let $Z_{\xi}^{(\mathrm{a})} \subset Y$ be the smooth rational quintic curve defined by $(\ref{eqn:ZGageneral})$.
Then the Zariski closure $\mathcal{Z}_{\xi}^{(\mathrm{a})}$ of $Z_{\xi}^{(\mathrm{a})}$, which is defined by the same equation as in (\ref{eqn:ZGageneraldvr}), is a relative smooth rational quintic curve whose reduction at a height 1 prime ideal $\p$ of $\cO_{F,S}$ is given by (\ref{eqn:ZGageneral}) (resp.\ (\ref{eqn:ZMU})) if $\xi \notin \p$ (resp.\ $\xi \in \p$).
By applying the two-ray game to $(\mathcal{Y}, \mathcal{Z}_{\xi}^{(\mathrm{a})})$ (Proposition \ref{prop:integraltwo-rayW5toV22}), we obtain a $V_{22}$-scheme $\mathcal{X}_{\xi}^{(\mathrm{a})}$ over $\cO_{F,S}$. This completes the proof of the first claim.
Since $F^{\times}/F^{\times 4}$ is an infinite set for any number field, the final statement follows from Theorem \ref{thm:GaV22overk}.
\end{proof}

\begin{prop}
\label{prop:GaV22degenerates}
{\cora 
Let $R$ be a discrete valuation ring, $K$ the fraction field of $R$, and $k$ the residue field of $R$.
Assume that $K$ is perfect.
Let $X$ be a $V_{22}$-variety of $\Ga$-type over $K$ (note that we have $\chara K \neq 2,5$ by Theorem \ref{thm:Gaclassification}).
Then $X$ admits good reduction at $R$ if and only if $\chara k \neq 2,5$.
Moreover, for any $V_{22}$-scheme $\mathcal{X}$ over $R$ with $\mathcal{X}_K \simeq X$, $\mathcal{X}_k$ is either a $V_{22}$-variety of $\Ga$-type over $k$ or a Mukai--Umemura variety.
}
\end{prop} 

\begin{proof}
{\cora
By the same proof as in Proposition \ref{prop:GaV22shaffails}, $X$ admits good reduction at $R$ if $\chara k \neq 2,5$.
The only-if part follows from the final assertion by Theorem \ref{thm:Gaclassification}, so it suffices to show the final assertion.}
    Take the line $L \subset X$ corresponding to the singular point of $\Sigma (X)$.
    Let $\mathcal{L}$ be the Zariski closure of $L$ in $\mathcal{X}$.
    By applying the two-ray game to $(\mathcal{X},\mathcal{L})$ (Proposition \ref{prop:integraltwo-rayV22toW5}), we obtain a $V_5$-scheme and a relative smooth rational quintic curve $(\mathcal{Y},\mathcal{Z})$ over $R$.
By Lemma \ref{lem:split_after_finite_extension},
    by extending $R$ if necessary,
    we may assume that $\mathcal{Y}$ is the split $V_5$-scheme over $R$ embedded in $\P^6_{R}$.
By Proposition \ref{prop:degenerationMUorGa}, $\mathcal{Z}$ is equal to the relative smooth quintic curve defined by (\ref{eqn:ZGageneraldvr}) for some $\xi \in R \setminus \{0 \}$.
If $\xi \in R^{\times}$, then $\mathcal{X}_{k}$ is a $V_{22}$-variety of $\Ga$-type by Theorem \ref{thm:Gaclassification}. 
If $\xi \in R \setminus R^{\times}$, then $\mathcal{X}_{k}$ is a Mukai--Umemura variety by Theorem \ref{thm:MukaiUmemuraclassification}.
This completes the proof.
\end{proof}

{\cblue
As shown in Proposition \ref{prop:GaV22shaffails}, the original Shafarevich conjecture does not hold for $\Ga$–$V_{22}$-varieties.
However, if we restrict ourselves to $\Ga$–good reductions, an analogous finiteness result can be proved (in an explicit form).

\begin{thm}
\label{thm:shafGatoGa}
    Let $F$ be a number field and $S$ a finite set of finite places of $F$.
    Consider the set 
    \[
    \Shaf_{\Ga,F,S}' := \left\{
    X \left|
    \begin{array}{l}
      \textup{$X \colon$ $V_{22}$-variety of $\Ga$-type over $F$} \\
      \textup{$X$ admits $\Ga$-good reduction at $\cO_{F,\p}$} \\
        \textup{for any finite place $\p$ of $F$ outside $S$}
    \end{array}
    \right\}\right.
    /F\textup{-isom.}
    \]
    Then $\Shaf_{\Ga,F,S}'$ is a finite set.
    Moreover, if $\cO_{F,S}$ is a principal ideal domain, then we have
    \[
    \# \Shaf_{\Ga,F,S}' = 
    \begin{cases}
      0   &\textup{if } 10 \notin \cO_{F,S}^{\times},\\
     \# \cO_{F,S}^{\times}/\cO_{F,S}^{\times 4} & \textup{otherwise}.
    \end{cases}
    \]
    \end{thm}

\begin{proof}
If $10 \notin \cO_{F,S}^{\times}$, then we have $\Shaf_{\Ga,F,S}' = \emptyset$ by Proposition \ref{prop:GaV22degenerates}.
Therefore, we assume $10 \in \cO_{F,S}^{\times}$.
To show the finiteness of $\Shaf_{\Ga,F,S}'$, we may extend $S$ so that $\cO_{F,S}$ is a principal ideal domain.
Therefore, it suffices to show the latter statement.
Note that we may replace $\Ga$-good reduction at $\cO_{F,\p}$ with $\Ga$-good reduction at $\cO_{F,(\p)}$ by Remark \ref{rem:grcomplete} (2).
We freely use Theorem \ref{thm:GaV22overk} and its notation.
By the standard gluing argument based on Proposition \ref{prop:GaV22degenerates} (cf.\ \cite[Lemma 4.1]{Javanpeykar-Loughran:GoodReductionFano}), for any $X_{\xi}^{(\a)} \simeq X \in \Shaf_{\Ga,F,S}'$ (where $\overline{\xi}\in F^{\times}/F^{\times 4}$), 
we may take a $V_{22}$-scheme $\mathcal{X}_X$ over $\cO_{F,S}$ such that $\mathcal{X}_{X,F} \simeq X$ and $\mathcal{X}_{X,k(\p)}$ is a $V_{22}$-variety of $\Ga$-type over $k(\p)$ for any finite place $\p \notin S$ of $F$.
We fix a finite place $\p\notin S$ of $F$.
Let $\mathcal{L} \subset \mathcal{X}_{X, \cO_{F,\p}}$ be the Zariski closure of the line $L \subset X_{F_\p}$ corresponding to the $F$-rational point in $\Sigma (X_{F_\p})_{\red, \sing}$.
Let $(\mathcal{Y},\mathcal{Z})$ be the $V_5$-scheme over $\cO_{F,\p}$ with the smooth rational quintic curve obtained by applying the two-ray game (Proposition \ref{prop:two-rayV22toW5}) to $(\mathcal{X}_{X, \cO_{F,\p}}, \mathcal{L})$.
By the construction of $X_{\xi}^{(\a)}$ (see Theorem \ref{thm:GaV22overk}), $\mathcal{Y}_{F_{\p}}$ is a split $V_5$-variety, and $\mathcal{Z}_{F_\p}$ is $\PGL_{2} (F_{\p})$-equivalent to $Z_\xi^{(\a)}$ as in (\ref{eqn:ZGageneral}).
Since $\p \nmid (2)$, $\mathcal{Y}_{\cO_{F,\p}}$ is a split $V_5$-scheme over $\cO_{F,\p}$ by \cite[Proof of Proposition 7.4]{V5}.
Therefore, by Proposition \ref{prop:degenerationMUorGa}, $\mathcal{Z}_{\cO_{F,\p}}$ is $\PGL_{2} (\cO_{F,\p})$-conjugate to $\mathcal{Z}_{\xi'}^{(\a)}$ in (\ref{eqn:ZGageneraldvr}) for some $\xi' \in \cO_{F,\p} \setminus 0$.
Since $\mathcal{X}_{X, k(\p)}$ is a $V_{22}$-variety of $\Ga$-type, we obtain $\xi'\in \cO_{F,\p}^{\times}$.
Since $\mathcal{Z}_{\xi',F_{\p}}^{(\a)} = Z_{\xi'}^{(\a)}$ is 
mapped to $Z_{\xi}^{(\a)}$ by an element of $\PGL_2 (F_{\p})$, the $V_{22}$-varieties of $\Ga$-type $X_{\xi, F_{\p}}^{(\a)}$ and $X_{\xi', F_{\p}}^{(\a)}$ over $F_{\p}$ are isomorphic.
By Theorem \ref{thm:GaV22overk}, we have $\xi \in \xi' F_{\p}^{\times 4}$.
Therefore, we have $4 \mid v_{\p} (\xi)$ for any finite place $\p \notin S$.
Since $\cO_{F,S}$ is a principal ideal domain, we may take $\varpi_{\p} \in \cO_{F,S}$ with $(\varpi_{\p}) = \p$.
Then 
$
\xi  \prod_{\p \notin S} \varpi_\p^{-v_{\p}(\xi)} \in \cO_{F,S}^{\times}.
$
Therefore, $\overline{\xi} \in F^{\times}/F^{\times 4}$ is contained in
$
\cO_{F,S}^{\times} / \cO_{F,S}^{\times 4} \subset F^{\times}/F^{\times 4}.
$

Conversely, suppose that $\overline{\xi}$ has a representative $\xi \in \cO_{F,S}^{\times}.$
Let $\mathcal{Y}$ be the split $V_5$-scheme over $\cO_{F,S}$ defined by (\ref{eqn:W5}), and $\mathcal{Z}_{\xi}^{(\a)} \subset \mathcal{Y}$ a smooth rational quintic curve defined by the same equation as in (\ref{eqn:ZGageneraldvr}).
By applying the two-ray game (Proposition \ref{prop:two-rayW5toV22}) to $(\mathcal{Y}, \mathcal{Z}_{\xi}^{(\a)})$, we obtain a $V_{22}$-scheme $\mathcal{X}_{\xi}^{(\a)}$ over $\cO_{F,S}$ such that $\mathcal{X}_{\xi, F}^{(\a)} = X_{\xi}^{(\a)}$.
Moreover, since $\xi \in \cO_{F,S}^{\times}$, $\mathcal{X}_{\xi, k(\p)}^{(\a)}$ is a $V_{22}$-variety of $\Ga$-type for any finite place $\p \notin S$ of $F$.
Therefore, we have $X_{\xi}^{(\a)} \in \Shaf_{\Ga,F,S}'$.
By the above argument, we obtain a bijection
$\cO_{F,S}^{\times}/\cO_{F,S}^{\times 4} \simeq \Shaf_{\Ga,F,S}'$ as desired.
\end{proof}
}

\subsection{Shafarevich conjecture for $V_{22}$-varieties of $\Gm$-type}

\begin{defn}
\label{defn:Gmreduction}
Let $R$ be a discrete valuation ring, $K$ the fraction field of $R$, and $k$ the residue field of $R$.
Let $X$ be a $V_{22}$-variety of $\Ga$-type over $K$.
Suppose that there exists a $V_{22}$-scheme $\mathcal{X}$ over $R$ such that $\mathcal{X}_{K} \simeq X$.
Take a finite extension $K'/K$ 
such that $\Sigma(X_{K'})_{\red,\sing}$ consists of two rational points $[L_1]$ and $[L_2]$.
Let $R'$ be a discrete valuation ring containing $R$ with the fraction field $K'$, and $k'$ the residue field of $R'$.
We denote the Zariski closure of $L_i$ in $\mathcal{X}_{R'}$ by $\mathcal{L}_i \subset \mathcal{X}_{R'}$.
\begin{enumerate}
    \item 
    If the reduction $\mathcal{L}_{1, k'}, \mathcal{L}_{2,k'}$ of $\mathcal{L}_1, \mathcal{L}_2$ on $\mathcal{X}_{k'}$ are different to each other, we say $X$ admits \emph{good reduction of standard type at $R$}.
    In this case, we call $\mathcal{X}$ a \emph{standard model of $X$ over $R$}.
    \item 
    If $\mathcal{L}_{1.k'}, \mathcal{L}_{2,k'} \subset \mathcal{X}_{k_L}$ are the same, we say $X$ admits \emph{good reduction of twisted type at $R$}.
    In this case, we call $\mathcal{X}$ a \emph{twisted model of $X$ over $R$}.
\end{enumerate}
Clearly, these definitions do not depend on the choice of $K'$ and $R'$.
\end{defn}

\begin{prop}
\label{prop:class.stdtwisted}
In the setting of Definition \ref{defn:Gmreduction}, we have the following.
\begin{enumerate}
    \item 
    If $\mathcal{X}$ satisfies the condition of Definition \ref{defn:Gmreduction} (1),
    then $\mathcal{X}_{k}$ is either a $V_{22}$-variety of $\Gm$-type or a Mukai--Umemura variety.
    Therefore, $X$ admits $\Gm$-good reduction or Mukai--Umemura good reduction.
    \item 
    If $\mathcal{X}$ satisfies the condition of Definition \ref{defn:Gmreduction} (2),
    then we have $\chara k \neq 2,5$, and $\mathcal{X}_{k}$ is either a $V_{22}$-variety of $\Ga$-type or a Mukai--Umemura variety.
    Therefore, $X$ admits $\Ga$-good reduction or Mukai--Umemura good reduction.
\end{enumerate}
\end{prop}

\begin{proof}
By taking a strict Henselization, we may assume that $R$ is a strictly Henselian discrete valuation ring.
Moreover, we may assume that $K'=K$, i.e.,\ the singular points of $\Sigma (X)_{\red}$ are two $K$-rational points.
We denote the corresponding lines on $X$ by $L_{1}$ and $L_2$.
Let $\mathcal{L}_i$ be the Zariski closure of $L_i$ in $\mathcal{X}$.
Let $(\mathcal{Y},\mathcal{Z})$ be a $V_5$-scheme over $R$ with a relative smooth quintic curve $\mathcal{Z}$ on $\mathcal{Y}$ obtained by applying the two-ray games for $(\mathcal{X}, \mathcal{L}_{1})$ (see Proposition \ref{prop:integraltwo-rayV22toW5}).
By Lemma \ref{lem:split_after_finite_extension}, after replacing $K$ by a finite extension and $R$ by its normalization if necessary, we may assume that $\mathcal{Y}$ is split and $\sqrt{2} \in R$.
Suppose that $\chara k \neq 2$.
Then $(Y, \mathcal{Z})$ satisfies the conditions in Proposition \ref{prop:degenerationGm}.
Suppose that the condition in Definition \ref{defn:Gmreduction} (1) (resp.\ (2)) holds.
By Remark \ref{rem:lineGmreduction} and the proof of Lemma \ref{lem:GmSigma} (1), the case (1) (resp.\ the case (2)) of Proposition \ref{prop:degenerationGm} holds, and $\mathcal{Z}_k$ is of $\Gm$-type or Mukai--Umemura type (resp.\ of $\Ga$-type or Mukai--Umemura type).
Thus, we obtain the desired result in this case.
In the case when $\chara k =2$, the result follows from Proposition \ref{prop:degenerationGmp=2} and Remark \ref{rem:lineGmreductionp=2} in a similar way.
\end{proof}
\begin{rem}
\label{rem:twistednew}
As we shall see in Proposition \ref{prop:grcforsplitGm}, twisted reduction indeed occurs even in equal-characteristic zero.
The standard reduction of $V_{22}$ varieties of $\Gm$-type in equal-characteristic zero is well-studied (e.g.,\ see \cite{Kuznetsov-ProkhorovGm}), 
but we could not find a reference on twisted reduction (in particular, on $\Ga$-good reduction).
\end{rem}
{\cora
\begin{prop}
\label{prop:grcforsplitGm}
Let $R$ be an excellent discrete valuation ring, $K$ the fraction field of $R$, and $k$ the residue field of $R$.
Let $p := \chara k$, and we denote the normalized additive valuation on $R$ by $\Ord$.
For
\[\
u \in \P^1 (K) \setminus \{(0:1), (1:1), (5:4), (1:0)\},
\]
let $X_u$ be the split $V_{22}$-variety of $\Gm$-type over $K$ defined in Definition \ref{defn:splitGm}.
We identify $u$ with an element of $K = \P^1 (K) \setminus \{(1:0)\}$.
Then the following hold.
\begin{enumerate}
    \item 
The $V_{22}$-variety of $\Gm$-type $X_u$ admits good reduction of standard type at $R$ if and only if $u \in R^{\times}$ and $u-1 \in R^{\times}$ hold.
    \item 
The $V_{22}$-variety of $\Gm$-type $X_u$ admits good reduction of twisted type at $R$ if and only if $p\neq 2,5$ and $\Ord (u-\frac{5}{4}) \geq 4$.
\end{enumerate}
\end{prop}
\begin{proof}
We show (1).
Assume that $u\in R$ and $u-1 \in R^{\times}$.
Let $\mathcal{X}_u$ be the $V_{22}$-scheme over $R$ obtained by applying the two-ray game (Proposition \ref{prop:integraltwo-rayW5toV22}) to $(\mathcal{Y}, \mathcal{Z}_u)$, where $\mathcal{Y} \subset \P^6_{R}$ is the split $V_5$-scheme over $R$ defined by (\ref{eqn:W5}) and $\mathcal{Z}_u$ is a relative smooth rational quintic curve over $R$ defined by (\ref{eqn:ZuGmdvr}).
Then $\mathcal{X}_u$ is a standard model of $X_u$ by Remarks \ref{rem:lineGmreduction} and \ref{rem:lineGmreductionp=2} (see the proof of Proposition \ref{prop:class.stdtwisted}).
Conversely, suppose that $X_u$ admits a standard model $\mathcal{X}$ over $R$.
By extending $R$ if necessary, we may assume that $\sqrt2 \in R$ and 
$\Sigma (X_u)_{\red, \sing}$ consists of a rational point.
Let $\mathcal{L} \subset\mathcal{X}$ be the Zariski closure of a line on $X_u$ corresponding to the point.
Let $(\mathcal{Y}, \mathcal{Z})$ be a $W_5$-scheme over $R$ with a relative smooth rational quintic curve obtained by applying the two-ray game (\ref{prop:integraltwo-rayV22toW5}) to $(\mathcal{X}, \mathcal{L})$.
By Lemma \ref{lem:split_after_finite_extension} (1), extending $R$ again,
we may assume that $\mathcal{Y}$ is a split $W_5$-scheme over $R$.
Then by Propositions \ref{prop:degenerationGm} and \ref{prop:degenerationGmp=2}, we obtain $u, u-1 \in R^{\times}$ as desired.
Note that, by the above proof and Lemma \ref{lem:split_after_finite_extension} (2), if $p \neq 2,$ then any standard model $\mathcal{X}$ is isomorphic to $\mathcal{X}_u$ (without extending $R$).

For (2), by Proposition \ref{prop:class.stdtwisted}, Theorems \ref{thm:MukaiUmemuraclassification} and \ref{thm:GaV22overk}, we may assume that $p \neq 2,5$.
Then the proof proceeds in the same way as in (1) by using Lemma \ref{lem:split_after_finite_extension} (2) instead of Lemma \ref{lem:split_after_finite_extension} (1) (without extending $R$).
\end{proof}
\cora}

{\cred
\begin{lem}
\label{lem:dvrXuaut}
Let $K$ be a $p$-adic field with the residue field $k$.
Let $u \in \cO_K^{\times}$ be a unit such that
$u-1 \in \cO_K^{\times}$. 
Moreover, we assume that $u- \frac{5}{4} \in \cO_K^{\times}$ if $p\neq 2,5$.
Let $\mathcal{X}_u$ be the $V_{22}$-scheme over $\cO_K$ obtained by applying the two-ray game (Proposition \ref{prop:integraltwo-rayW5toV22}) to $(\mathcal{Y}, \mathcal{Z}_u)$, where $\mathcal{Y} \subset \P^6_{\cO_K}$ is the split $V_5$-scheme over $\cO_K$ defined by (\ref{eqn:W5}) and $\mathcal{Z}_u$ is a relative smooth rational quintic curve over $\cO_K$ defined by (\ref{eqn:ZuGmdvr}).
Then there exists an embedding of group schemes
\[
\iota\colon \mathbb{G}_{m,\cO_K} \rtimes (\Z/2\Z)_{\cO_K} \hookrightarrow  \Aut_{\mathcal{X}_u/\cO_K}
\]
such that $\iota_K$ and $\iota_k$ induces an isomorphisms
\[
\mathbb{G}_{m,K} \rtimes (\Z/2\Z)_K \simeq \Aut_{\mathcal{X}_{u,K}/K,\red}
\quad 
\textup{and}
\quad
\mathbb{G}_{m,k} \rtimes (\Z/2\Z)_k \simeq \Aut_{\mathcal{X}_{u,k}/k,\red}
.
\]
\end{lem}
\begin{proof}
By the construction of $\mathcal{X}_u$, 
we obtain an immersion
\begin{equation}
\label{eqn:GmembedOK}
\mathbb{G}_{m,\cO_K} \hookrightarrow \Aut_{\mathcal{X}_u/\cO_K}.
\end{equation}
We want to show the existence of an involution on $\mathcal{X}_u$ as in Lemma \ref{lem:GmSigma}.
Basically, we follow the same strategy as in Lemma \ref{lem:GmSigma}.
By \cite[Proposition 3.2 and Lemma 4.1]{Kuznetsov-ProkhorovGm}, $\mathcal{X}_{u, \overline{K}}$ admits a unique $\mathbb{G}_{m,\overline{K}}$-invariant smooth conic.
By the uniqueness, this descends to a $\mathbb{G}_{m,K}$-invariant smooth conic $C \subset \mathcal{X}_{u,K}$.
Let $\mathcal{C}\subset \mathcal{X}_u$ be the Zariski closure.
Then by the proof of Lemma \ref{lem:GmSigma}, $\mathcal{C}$ is a relative smooth $\mathbb{G}_{m,\cO_K}$-invariant conic on $\mathcal{X}_u$.
By applying the two-ray game (Proposition \ref{prop:integraltwo-rayV22toQ3}) to $(\mathcal{X}_u, \mathcal{C})$, we obtain $(\mathcal{Q},\Gamma)$, 
where $\mathcal{Q} \subset \P^4_{\cO_K}$ is a relative quadric threefold over $\cO_K$ and $\Gamma \subset \mathcal{Q}$ is a relative quadratically normal smooth genus 0 sextic over $\cO_K$.
By the construction, we have an embedding
$\mathbb{G}_{m,\cO_K} \hookrightarrow \Aut_{(\mathcal{Q},\Gamma)/\cO_K}.$
Moreover, by the proof of Lemma \ref{lem:GmSigma}, the induced morphism $\mathbb{G}_{m,K} \rightarrow \Aut_{\Gamma_K/K}$ is non-trivial.
Since an automorphism group of a genus 0 curve without rational points is non-split, 
$\Gamma_K$ admits a rational point, i.e.,\ $\Gamma_K$ is a rational sextic.
Therefore, $\Gamma$ is a relative rational sextic over $\cO_K$.
By the same argument as in \cite[Lemma 3.5 and Lemma 3.6]{Kuznetsov-ProkhorovGm}, by using the action of $\PGL_{2} (\cO_K)$, we may assume the following:
\begin{itemize}
    \item
    The curve
    $\Gamma \subset \P^4_{\cO_K}$ is defined by
    \begin{equation}
    \begin{gathered}
        \P^1 \rightarrow \P^4\\
        (s:t) \mapsto (s^6:s^5t:s^3t^3:st^5:t^6)
    \end{gathered}
\end{equation}
    \item
    The quadric $\mathcal{Q} \subset \P^4_{[x_0,x_1,x_2,x_3,x_4],\cO_K}$ is defined by the equation 
    $
    v Q_0 + Q_1
    $
    for some $v \in \cO_K^{\times}$ with $v-1 \in \cO_K^{\times}$.
    Here, $Q_0 := x_0 x_4-x_2^2$ and $Q_1:= x_2^2 - x_1x_3$.
    \item 
    $\mathbb{G}_{m,\cO_K}$ acts on $\P^4_{\cO_K}$ by
    \[
    t \cdot (x_0:x_1:x_2:x_3:x_4) = (x_0:tx_1:t^3x_2:t^5x_3:t^6x_6).
    \]
\end{itemize}
Then the involution on $\P^4_{\cO_K}$ defined by the same equation as (\ref{eqn:involutionP4}) induces the desired involution on $\mathcal{X}_u$.
Combining with (\ref{eqn:GmembedOK}), we obtain the desired embedding $\iota$.
\end{proof}

By using Lemma \ref{lem:dvrXuaut}, we obtain a $p$-adic integral analogue of Corollary \ref{cor:Gmclassificationfinite} as following:

\begin{lem}
\label{lem:unramifiedextensionGmV22}
Let $K$ be a $p$-adic field with the residue field $k$.
Let $u \in \cO_K^{\times}$ be a unit such that
$u-1 \in \cO_K^{\times}$. 
Moreover, we assume that $u- \frac{5}{4} \in \cO_K^{\times}$ if $p\neq 2,5$.
Let $\mathcal{X}_u$ be the $V_{22}$-scheme of $\Gm$-type over $\cO_K$ defined in Lemma \ref{lem:dvrXuaut}.
Let $L/K$ be a unramified quadratic extension.
Then the set
\[
\{
\mathcal{X} \colon V_{22} \textup{-scheme over }\cO_K \mid \mathcal{X}_{\cO_L} \simeq \mathcal{X}_{u, \cO_L}
\}/\cO_K\textup{-isom}
\]
consists of two elements $\mathcal{X}_u$ and $\mathcal{X}_u'$.
Moreover, $\mathcal{X}_{u,K}'$ and $\mathcal{X}_{u,k}'$ is not split.
\end{lem}

\begin{proof}
By Lemma \ref{lem:dvrXuaut}, we have
\begin{equation}
\label{eqn:isomautdvr}
\Aut_{\mathcal{X}_{u}/\cO_K} (\cO_L) \simeq \mathbb{G}_{m,\cO_K} (\cO_L) \rtimes \Z/2\Z.
\end{equation}
Since $\cO_L/\cO_K$ is a finite \'{e}tale Galois extension, the set in the statement is isomorphic to
\[
H^1 (\mathcal{G}, \mathbb{G}_{m,\cO_K} (\cO_L) \rtimes \Z/2\Z).
\]
Here, we set $\mathcal{G}:= \Gal (\cO_L/\cO_K)\simeq \Z/2\Z$.
Let $\overline{c_{\mathcal{X}}}$ be the class
\[
\overline{c_{\mathcal{X}}} \in H^1 (\mathcal{G}, \mathbb{G}_{m,\cO_K} (\cO_L) \rtimes \Z/2\Z)
\]
corresponding to $\mathcal{X}$, 
and $p$ the natural projection
\[
H^1 (\mathcal{G}, \mathbb{G}_{m,\cO_K} (\cO_L) \rtimes \Z/2\Z) \rightarrow H^1 (\mathcal{G}, \Z/2\Z).
\]
Note that we have
$H^1 (\Spec \cO_K, \Gm) =0$ and $H^1 (\Spec \cO_L, \Gm) =0.$
By the Hochschild-Serre spectral sequence, we obtain
$
H^1 (\mathcal{G}, \Gm (\cO_L)) =0.
$
Therefore, if $p (\overline{c_{\mathcal{X}}})$ is the trivial class, then $\overline{c_{\mathcal{X}}}$ is the trivial class, which corresponds to $\mathcal{X}_u$.
We assume that $p (\overline{c_{\mathcal{X}}})$ is the non-trivial class.
By the same argument as in Theorem \ref{thm:Gmclassificationgeneral}, we obtain a diagram
\[
  \begin{CD}
       H^1 (\mathcal{G},\mathbb{G}_{m,\cO_K} (\cO_L))  @>>>  H^1(\mathcal{G},\mathbb{G}_{m,\cO_K} (\cO_L) \rtimes \Z/2\Z) @>{p}>>  H^1(\mathcal{G}, \Z/2\Z)   \\
     @.    @V{\simeq}V{\cdot c_{\mathcal{X}}^{-1}}V  @V{\simeq}V{\cdot c_{\mathcal{X}}^{-1}}V   \\
       H^1(\mathcal{G},U_{\cO_L/\cO_K}(1) (\cO_L)) @>>>  H^1 (\mathcal{G},U_{\cO_L/\cO_K}(1) (\cO_L) \rtimes \Z/2\Z) @>{p}>>  H^1 (\mathcal{G}, \Z/2\Z).
  \end{CD}
\]
Here, $U_{\cO_L/\cO_K} (1)$ is the rank 1 relative unitary group, which is defined as the exact sequence
\[
1 \rightarrow U_{\mathcal{\cO_L/\cO_K}} (1) \rightarrow \Res_{\cO_L/\cO_K} \mathbb{G}_{m,\cO_L} \xrightarrow{N_{\cO_L/\cO_K}} \mathbb{G}_{m,\cO_K} \rightarrow 1,
\]
where $N_{\cO_L/\cO_K}$ is the norm map.
By taking $\cO_L$-valued points and then cohomology, we obtain 
\[
H^1 (\mathcal{G}, U_{\mathcal{\cO_L/\cO_K}} (1) (\cO_L)) \simeq \cO_K^{\times}/ N_{\cO_L/\cO_K} (\cO_L^{\times}).
\]
Since $L/K$ is an unramified extension, the right-hand side is trivial.
Therefore, $p^{-1} (p (\overline{c_{\mathcal{X}}}))$ is a singleton $\{\overline{c_X}\}$.
We denote the corresponding scheme by $\mathcal{X}_u'$.
The image of $\overline{c_{\mathcal{X}}}$ in $H^1 (\Gal (L/K), \Z/2\Z)$ via
\begin{eqnarray*}
H^1 (\mathcal{G}, \Aut_{\mathcal{X}_u/\cO_K} (\cO_L)) \rightarrow H^1 (\Gal (L/K), \Aut_{\mathcal{X}_{u,K}/K} (L)) \rightarrow H^1 (\Gal (L/K), \Z/2\Z)
\end{eqnarray*}
is non-trivial element by the property of an isomorphism (\ref{eqn:isomautdvr}) (see Lemma \ref{lem:dvrXuaut}). Hence, $\mathcal{X}'_{u,K}$ is not split.
The case of $\mathcal{X}'_{u,k}$ is analogous.
\end{proof}
}

{
\cblue
\begin{thm}
\label{thm:GmShafarevich}
Let $F$ be a number field 
and $S$ a finite set of finite places of $F$.
Then the set
    \[
    \Shaf_{\Gm,F,S} := \left\{
    X \left|
    \begin{array}{l}
      \textup{$X \colon$$V_{22}$-variety of $\Gm$-type over $F$} \\
      \textup{$X$ admits good reduction at $\cO_{F,\p}$}\\
        \textup{for any finite place $\p$ of $F$ outside $S$}
    \end{array}
    \right\}\right.
    /F\textup{-isom}
    \]
is a finite set.
\end{thm}
}
\begin{proof}
{ \cblue
We may assume that $S$ contains any finite place $\p$ of $F$ with $\p \mid (2)$ or $\p \mid (5)$.
In the proof, we freely use Theorem \ref{thm:Gmclassificationgeneral} and its notation.
For any
\[
u=(u:1) \in P_F := \P^1 (F) \setminus \{(0:1), (1:1),(1:0),(5:4)\},
\]
and a quadratic extension $E/F$,
we fix a bijection
\[
\Psi_{u,E} \colon \Phi_u^{-1} (\Spec E) \simeq F^{\times}/ N_{E/F} (E^{\times})
\]
by choosing a section 
\[
\Z/2\Z \simeq \pi_0 (\Aut_{X_u/F, \red}) (F) \hookrightarrow \Aut_{X_u/F, \red}(F)
\]
(see Remark \ref{rem:Gmclassificationcanonical} (1)).
We write $\Phi_{u}^{-1} (\Spec E)$ as
$
\{
X_u^{E,\alpha} \mid\alpha \in F^{\times}/N_{E/F} (E^{^\times}).
\}
$
Then we obtain
$
A_F = \{(X_u)_u, (X_{u}^{E,\alpha})_{u, E,\alpha}\}.
$
We denote $X_u$ by $X_u^{F,1}$ for the compatibility.

\noindent{\bf Claim 1.}
The set
\[
P_F^{\circ} :=
\left\{
u \in P_F \left| 
\begin{array}{l}
\textup{ there exists a field } E \textup{ with } [E:F] \leq 2 \textup{ and } \alpha \in F^{\times}/N_{E/F}(E^{\times})\\
\textup{ such that } X_{u}^{E,\alpha} \in \Shaf_{\Gm,F,S}
\end{array}
\right\}\right.
\]
is a finite set.
}
\begin{proof}
{\cblue
For any $X \simeq X_u^{E,\alpha} \in \Shaf_{\Gm,F,S}$, we may take a smooth projective scheme model $\mathcal{X}_{X}$ over $\cO_{F,S}$ by Remark \ref{rem:grcomplete} (2) and the standard gluing argument (cf.\ \cite[Lemma 4.1]{Javanpeykar-Loughran:GoodReductionFano}).
Let $\p \notin S$ be a finite place of $F$, and $\q$ a finite place of $E$ with $\q \mid \p$.
{\cora Since $X_{E_{\q}}$ is a split $V_{22}$-variety of $\Gm$-type $X_u$ (over $E_{\q}$), by Proposition \ref{prop:grcforsplitGm}, we have $u , u-1 \in \cO_{E,\q}^{\times}$.}
Therefore, $u$ is contained in  the set
$\{
u \in \cO_{F,S}^{\times}\mid 1-u \in \cO_{F,S}^{\times}
\}$,
which is a finite set by Siegel--Mahler--Lang's theorem (\cite[Chapter VII, \S4]{Lang}).
This completes the proof.
}
\end{proof}
{\cred
For any $u\in P_F$, we put
$
\Shaf_u := \Shaf_{\Gm,F,S} \cap A_{F,u}.
$
Note that $\Shaf_u \neq \emptyset$ if and only if $u \in P_F^{\circ}$, so we have
$
\Shaf_{\Gm,F,S} = \bigsqcup_{u \in P_F^{\circ}} \Shaf_u.
$
By Claim 1, we are reduced to proving $\#\Shaf_u< \infty$ for each $u \in P_F^{\circ}$.

\noindent{\bf Claim 2.}
For fixed $u \in P_F^{\circ}$, 
the set
\[
B_u :=
\{
E \colon \textup{field with } [E:F]\leq2 \mid 
\exists \alpha \in F^{\times}/N_{E/F} F^{\times} \textup{ such that } X_u^{E,\alpha} \in \Shaf_u
\}
\]
is a finite set.
\begin{proof}
We set
\[
T_u := 
\left\{
\p\colon \textup{finite place of }F
\left|\
v_{\p} \left( u-\frac{5}{4} \right) \geq 1 
\right\}\right..
\]
For any $X_u^{E,\alpha} \simeq X \in \Shaf$, we may take a smooth projective scheme model $\mathcal{X}_{X}$ over $\cO_{F,S}$ as in the proof of Claim 1.
By the Hermite--Minkowski Theorem, it suffices to show that $E/F$ is unramified outside $S\cup T_u$.
We may assume that $[E:F]=2$.
Let $\p$ be a finite place of $F$ that ramifies in $E/F$, and $\q$ the finite place of $E$ with $\q^2 = \p$.
We want to show $\p \in S\cup T_u$, so we assume that $\p \notin S$.
Since $\Sigma(X)_{\red,\sing} \simeq \Spec E$ and $\Sigma (\mathcal{X}_{X, \cO_{F,\p}} / \cO_{F,\p})$ is proper, we obtain a morphism $\Spec \cO_{E, \q} \rightarrow \Sigma (\mathcal{X}_{X, \cO_{F,\p}} / \cO_{F,\p})$ whose scheme theoretic image is equal to the Zariski closure $V$ of $\Sigma (X_{E_\q})_{\red, \sing}$ in $\Sigma (\mathcal{X}_{X, \cO_{F,\p}} / \cO_{F,\p})$.
Therefore, $V_{k(\p)}$ is dominated by $\Spec \cO_{E,\q}/ \p \cO_{E,\q}$, which admits just one $k(\p)$-valued point.
Hence the reduced reduction $(V_k)_{\red}$ consists of one $k (\p)$-rational point. This means $\mathcal{X}_{X, \cO_{F,\p}}$ is a twisted model over $\cO_{F,\p}$.
{\cora 
Since $\mathcal{X}_{X,\cO_{E,\q}}$ is a twisted model of a split $V_{22}$-variety of $\Gm$-type,
by Proposition \ref{prop:grcforsplitGm}, we obtain $v_\q (u-\frac{5}{4}) \geq 4$.}
Therefore, we have $\p \in T_u$ as desired.
This completes the proof.
\end{proof}
For any $u \in P_F$ and a field extension $E/F$ with $[E:F] \leq 2$, we set
\[
\Shaf_{u,E} := \Phi_u^{-1} (\Spec E) \cap \Shaf_{\Gm,F,S} \subset \Shaf_u.
\]
Here, we set $\Phi_u^{-1} (\Spec F):=\{X_u\}$. By Claim 2, $\Shaf_{u,E}$ is empty for all but finitely many $E$.
Therefore, it suffices to show that $\Shaf_{u,E}$ is a finite set.
Since $\Shaf_{u,F}$ is a singleton, we may assume that $E \neq F$.
Let $T_u$ be as in the proof of Claim 2.
By the proof of Claim 2, $E/F$ is unramified outside $S \cup T_u$.
Let $\p \notin S \cup T_u$ be any finite place of $F$, and $X \in \Shaf_{u,E}$.
Let $\q \mid \p$ be a finite place of $E$.
{\cora By $\p \notin T_u$ and proposition \ref{prop:grcforsplitGm}, we have $u, u-1, u-\frac{5}{4} \in \cO_{E,\q}^{\times}$. 
Moreover, it follows from $\p \nmid (2)$ and what is written at the end of the proof of Proposition \ref{prop:grcforsplitGm} (1) that $\mathcal{X}_{X, \cO_{E,\q}} \simeq \mathcal{X}_{u, \cO_{E,\q}}$.
}
If $\q$ over $\p$ is split, then $X_{F_\p} = X_{E_{\q}} = X_u$.
Suppose that $\q$ over $\p$ is inert. 
Since we have $\Aut_{X/F, \red} \simeq U_{E/F} (1) \rtimes \Z/2\Z$, we obtain $\Aut_{X_{F_\p}/F_{\p}, \red} \simeq U_{E_{\q}/F_{\p}} (1) \rtimes \Z/2\Z$ and hence $X_{F_\p}$ is not split.
By Lemma \ref{lem:unramifiedextensionGmV22}, $\mathcal{X}_{X,\cO_{F,p}}$ is isomorphic to $\mathcal{X}_{u}'$, where we use the same notation as in Lemma \ref{lem:unramifiedextensionGmV22}.
Let
$\alpha_u \in F_{\p}^{\times}/N_{E_\q/F_\p} E_\q^{\times}$ be the element such that
$
\mathcal{X}_{u,F_{\p}}' \simeq X_{u}^{(E_\q, \alpha_u)}.
$
Note that this only depends on $u$, i.e.,\ does not depend on $X$.
By the Hasse norm theorem and the local class field theory, we obtain a finite-to-one morphism
\[
p:=\prod p_\p\colon
F^{\times}/N_{E/F} F^{\times} \rightarrow \prod_{\p \notin S \cup T_u} F_{\p}^{\times}/N_{E_\q/F_\p} E_{\q}^{\times}.
\]
Moreover, by the above consideration, for any $X_{u,E}^{(\alpha_X)} \simeq X\in \Shaf_{u,E}$, we have
\[
p_\p (\alpha_{X}) =
\begin{cases}
     \alpha_{u} \in F_\p^{\times}/N_{E_q/F_\p} E_{\q}^{\times}  & \textup{if } \p \textup{ is inert in } E,\\
    1\in F_\p^{\times} / N_{F_\p/F_\p} F_{\p}^{\times} =\{1\} &
    \textup{if } \p \textup{ is split in } E.
\end{cases}
\]
Therefore, 
$
p (\{\alpha_X \mid X\in \Shaf_{u,E}\})
$
is a singleton, so $\Shaf_{u,E}$ is a finite set as desired.
}
\end{proof}

\medskip
\noindent {\bf Acknowledgements}
The authors express their gratitude to Professor Shigeru Mukai for various invaluable discussions and also for providing us with the problem on Mukai--Umemura varieties in positive characteristic.
The authors, particularly the second author, also wish to thank Professor Kento Fujita for numerous insightful discussions about $V_{22}$.
The authors would like to thank Professors Tatsuro Kawakami, Hiromu Tanaka, Yuji Odaka, Takashi Kishimoto, and Adrien Dubouloz, Daniel Loughran, and Ariyan Javanpeykar for valuable discussions.

\printbibliography
\end{document}